\documentclass[11pt]{amsart}

\usepackage{amsmath,amssymb,amsthm,tikz, nicefrac, mathrsfs,upgreek} 
\RequirePackage{color}
\RequirePackage[colorlinks, urlcolor=my-blue,linkcolor=my-red,citecolor=my-red]{hyperref}
\definecolor{my-blue}{rgb}{0.0,0.0,0.6}
\definecolor{my-red}{rgb}{0.5,0.0,0.0}
\definecolor{my-green}{rgb}{0.0,0.5,0.0}
\definecolor{nicos-red}{rgb}{0.75,0.0,0.0}
\definecolor{nicos-green}{rgb}{0.0,0.75,0.0}
\definecolor{light-gray}{gray}{0.6}
\definecolor{really-light-gray}{gray}{0.8}
\definecolor{sussexg}{rgb}{0.0,0.5,0.5}
\definecolor{sussexp}{rgb}{0.5,0.0,0.5}

\usepackage{dsfont} %  To avoid compliation problems with the indicator function

\usepackage[utf8]{inputenc}

\usepackage{nicefrac}
\usepackage{multicol}
\PassOptionsToPackage{dvipsnames}{xcolor}
\usepackage{tikz,xcolor}

\usetikzlibrary{shapes,arrows}
\usetikzlibrary{hobby}
\usetikzlibrary{decorations.markings}

\newtheorem{theorem}{\color{my-red}{\sc Theorem}}[section]
\newtheorem{lemma}[theorem]{\color{my-red} \sc Lemma}
\newtheorem{proposition}[theorem]{\color{my-red} \sc Proposition}
\newtheorem{corollary}[theorem]{\color{my-red} \sc Corollary}
\newtheorem{conjecture}[theorem]{\color{my-red} \sc Conjecture}

\numberwithin{equation}{section}
\theoremstyle{remark}
\newtheorem{remark}[theorem]{\color{my-red} Remark}

\newcommand{\be}{\begin{equation}}
\newcommand{\ee}{\end{equation}}

\providecommand{\abs}[1]{\vert#1\vert}

\newcommand{\eone}{\textup{e}_1}
\newcommand{\etwo}{\textup{e}_2}
\newcommand{\TV}[1]{{\lVert #1 \rVert}_{\normalfont
\text{TV}}}

\def\mix{\textup{mix}}

\def\bE{\mathbb{E}}
\def\bN{\mathbb{N}}
\def\bP{\mathbb{P}}

\def\bR{\mathbb{R}}
\def\bZ{\mathbb{Z}}

\def\cL{\mathcal{L}}

\def\c{\textup{c}} % index ordering
\def\g{\textup{g}} % index ordering

 \def\Z{\bZ} 
 
\def\R{\bR}
\def\N{\bN}
\def\P{\bP}

\usepackage{stackengine}

\DeclareMathOperator{\Var}{Var}

\def\E{\bE}
\def\P{\bP} %% environment measure 

\definecolor{partcolor1}{rgb}{0.0,0.5,0.0}
\definecolor{partcolor2}{rgb}{0.0,0.5,0.0}

\definecolor{darkgreen}{rgb}{0.0,0.5,0.0}
\definecolor{darkblue}{rgb}{0.5,0.1,0.5}
\definecolor{deepblue}{rgb}{0.25,0.41,0.88}
\definecolor{nicosred}{rgb}{0.65,0.1,0.1}
\definecolor{light-gray}{gray}{0.7}
\allowdisplaybreaks[1]

\RequirePackage{datetime} %comment out when arxiving or submitting
\allowdisplaybreaks
\begin{document}
\usdate
\title[Mixing times for the TASEP in the maximal current phase]
{Mixing times for the TASEP in the maximal current phase}
\author{Dominik Schmid}
\address{Dominik Schmid, University of Bonn, Germany}
\email{d.schmid@uni-bonn.de}
\keywords{totally asymmetric simple exclusion process, mixing times, second class particles, corner growth model, last-passage times, competition interface}
\subjclass[2010]{Primary: 60K35; Secondary: 60K37, 60J27} 
\date{\today}
\begin{abstract}
We study mixing times for the totally asymmetric simple exclusion process (TASEP) on a segment of size $N$ with open boundaries. We focus on the maximal current phase, and prove that the mixing time is of order $N^{3/2}$,  up to logarithmic corrections. In the triple point, where the TASEP with open boundaries approaches the Uniform distribution on the state space, we show that the mixing time is precisely of order $N^{3/2}$. This is conjectured to be the correct order of the mixing time for a wide range of particle systems with maximal current. Our arguments rely on a connection to last-passage percolation, and recent results on moderate deviations of last-passage times.
\end{abstract}
\maketitle
\vspace*{-0.35cm}

\section{Introduction} \label{sec:Introduction}

Over the last decades, the one-dimensional totally asymmetric simple exclusion process (TASEP) is among the most investigated particle systems, with motivations and applications from statistical mechanics, probability theory and combinatorics \cite{BSV:SlowBond,BE:Nonequilibrium,CW:TableauxCombinatorics,C:KPZReview,L:Book2}. The model has the following simple description: a collection of indistinguishable particles are placed on different integer sites. Each site is endowed with a Poisson clock, independently of all others, which rings at rate 1. Whenever the clock at an occupied site rings, we move its particle to the right, provided that the target site is vacant. This last condition is called the exclusion rule. For the TASEP with open boundaries, we restrict the  dynamics to a segment, and let in addition particles enter at the left-hand side boundary at rate $\alpha$, and exit at the right-hand side boundary at rate $\beta$. \\ 

Despite its simple construction, the TASEP serves as a powerful model to describe traffic or the motions of particles in gases.
The TASEP was first introduced in the mathematical literature as an interacting particle system; see \cite{S:InteractionMP}, and \cite{L:Book2} for a classical introduction. Since the discovery that the TASEP has a description as a corner growth model (CGM) on $\Z^2$, and that it is an exactly solvable model in the Kardar-Parisi-Zhang universality class, the set of available techniques rapidly developed; see \cite{CS:OpenASEPWeakly,J:KPZ,R:PDEresult} for seminal work, and \cite{C:KPZReview} for a survey on the KPZ universality class. Today many precise results, for example on scaling limits for last passage times or sharp bounds on the transversal fluctuations of geodesics, are known \cite{BG:TimeCorrelation,BGZ:TemporalCorrelation,BSS:Coalescence,BC:PSConjecture}. Many of these results have in common that they benefit from exact expressions for various observables which come from integrable probability. These expressions are then combined with probabilistic techniques. \\

In recent years, mixing times for exclusion processes gained more attention, and sharp bounds in terms of  cutoff  and determining limit profiles were achieved; see for example \cite{BN:CutoffASEP, GNS:MixingOpen,  LL:CutoffASEP,L:CutoffSEP, W:MixingLoz}. 
However, the variety of results available from integrable probability was so far not exploited. In this paper, we intertwine the different perspectives on the TASEP with open boundaries  --- as an interacting particle system and as a corner growth model. We focus on mixing times in  the maximal current phase $\alpha,\beta \geq \frac{1}{2}$, which is from a physical perspective of particular interest due to its relation to the KPZ universality class; see \cite{CK:StationaryKPZ,CS:OpenASEPWeakly}. To our best knowledge, this is the first time a mixing time of order $N^{3/2}$, involving the exponent $2/3$ of the KPZ relaxation scale, is shown for an interacting particle system which belongs to the KPZ universality class. 

\subsection{Model and results}  \label{sec:ModelResults}

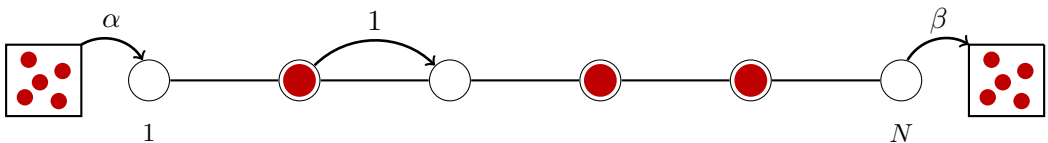
\begin{figure} 
\centering
\begin{tikzpicture}[scale=1]

\def\spiral[#1](#2)(#3:#4:#5){% \spiral[draw options](placement)(end angle:revolutions:final radius)
\pgfmathsetmacro{\domain}{pi*#3/180+#4*2*pi}
\draw [#1,
       shift={(#2)},
       domain=0:\domain,
       variable=\t,
       smooth,
       samples=int(\domain/0.08)] plot ({\t r}: {#5*\t/\domain})
}

\def\particles(#1)(#2){

  \draw[black,thick](-3.9+#1,0.55-0.075+#2) -- (-4.9+#1,0.55-0.075+#2) -- (-4.9+#1,-0.4-0.075+#2) -- (-3.9+#1,-0.4-0.075+#2) -- (-3.9+#1,0.55-0.075+#2);
  
  	\node[shape=circle,scale=0.6,fill=nicos-red] (Y1) at (-4.15+#1,0.2-0.075+#2) {};
  	\node[shape=circle,scale=0.6,fill=nicos-red] (Y2) at (-4.6+#1,0.35-0.075+#2) {};
  	\node[shape=circle,scale=0.6,fill=nicos-red] (Y3) at (-4.2+#1,-0.2-0.075+#2) {};
   	\node[shape=circle,scale=0.6,fill=nicos-red] (Y4) at (-4.45+#1,0.05-0.075+#2) {};
  	\node[shape=circle,scale=0.6,fill=nicos-red] (Y5) at (-4.65+#1,-0.15-0.075+#2) {}; }

  \def\annhil(#1)(#2){	  \spiral[black,thick](9.0+#1,0.09+#2)(0:3:0.42);
  \draw[black,thick](8.5+#1,0.55+#2) -- (9.5+#1,0.55+#2) -- (9.5+#1,-0.4+#2) -- (8.5+#1,-0.4+#2) -- (8.5+#1,0.55+#2); }

	\node[shape=circle,scale=1.5,draw] (Z) at (-3,0){} ;
    \node[shape=circle,scale=1.5,draw] (A) at (-1,0){} ;
    \node[shape=circle,scale=1.5,draw] (B1) at (1,0){} ;
	\node[shape=circle,scale=1.5,draw] (B2) at (3,0){} ;
	\node[shape=circle,scale=1.5,draw] (C) at (5,0) {};
 	\node[shape=circle,scale=1.5,draw] (D) at (7,0){} ; 

 		\node[shape=circle,scale=1.2,fill=nicos-red] (YZ2) at (3,0) {};
 	
	\node[shape=circle,scale=1.2,fill=nicos-red] (YZ) at (-1,0) {};

	\node[shape=circle,scale=1.2,fill=nicos-red] (YZ3) at (5,0) {};

		\draw[thick] (Z) to (A);	
	\draw[thick] (A) to (B1);	
	\draw[thick] (B1) to (B2);		
		\draw[thick] (B2) to (C);	
  \draw[thick] (C) to (D);

\particles(0)(0);
\particles(6.9+4.9+1)(0);

%\annhil(0)(0.7);
%\annhil(-13.4)(-0.7);
\draw [->,line width=1pt]  (A) to [bend right,in=135,out=45,->] (B1);

%\draw [->,line width=1pt]  (B2) to [bend right,in=135,out=45,->] (C);
  
  % \draw [->,line width=1pt] (B2) to [bend right,in=-135,out=-45] (C);
  % \draw [->,line width=1pt] (B2) to [bend right,in=-135,out=-45] (A);
 %     \draw [->,line width=1pt] (A) to [bend right,in=-135,out=-45] (Z2);
    \node (text1) at (0,0.8){$1$} ;    
 
	\node (text3) at (-2.5-1,0.8){$\alpha$}; 
	\node (text4) at (7.1+0.4,0.8){$\beta$};

    \node[scale=0.9] (text1) at (-3,-0.7){$1$} ;    
    \node[scale=0.9] (text1) at (7,-0.7){$N$} ;   
  	
  \draw [->,line width=1pt] (-3.9,0.475) to [bend right,in=135,out=45] (Z);
%   \draw [->,line width=1pt] (Z) to [bend right,in=135,out=45] (-3.9,-0.475); 
  % \draw [->,line width=1pt] (6.9+1.6,-0.475) to [bend right,in=135,out=45] (D);	
   \draw [->,line width=1pt] (D) to [bend right,in=135,out=45] (6.9+1,0.475);

	\end{tikzpicture}	
\caption{\label{fig:BDEP}The TASEP with open boundaries for parameters $\alpha,\beta >0$.}
 \end{figure}
We consider the TASEP with open boundaries on a segment $[N]:=\{1,\dots,N\}$ for some $N \in \N$, and  parameters $\alpha,\beta>0$. It is  the continuous-time Markov chain $(\eta_t)_{t\geq 0}$ on the state space $\Omega_N :=\{ 0,1\}^N$ with generator
\begin{align*}
\cL f(\eta) &=  \sum_{x =1}^{N-1}  \eta(x)(1-\eta(x+1))\left[ f(\eta^{x,x+1})-f(\eta) \right] \\
 &+ \alpha (1-\eta(1)) \left[ f(\eta^{1})-f(\eta) \right] \hspace{2pt}  + \beta \eta(N)\left[ f(\eta^{N})-f(\eta) \right] 
\end{align*}
 for all cylinder functions $f$. Here, we use the standard notation
\begin{equation*} 
\eta^{x,y} (z) = \begin{cases} 
 \eta (z) & \textrm{ for } z \neq x,y\\
 \eta(x) &  \textrm{ for } z = y\\
 \eta(y) &  \textrm{ for } z = x \, ,
 \end{cases}
 \quad
 \text{and}
 \quad
 \eta^{x} (z) = \begin{cases} 
 \eta (z) & \textrm{ for } z \neq x\\
1-\eta(z) &  \textrm{ for } z = x
 \end{cases}
\end{equation*}
to denote swapping and flipping of values in a configuration $\eta \in \Omega_N$ at sites $x,y \in [N]$. We say that site $x$ is \textbf{occupied} if $\eta(x)=1$, and \textbf{vacant} otherwise. A visualization of the TASEP with open boundaries is given in Figure \ref{fig:BDEP}. Note that the TASEP with open boundaries is irreducible, and thus has a unique stationary distribution which we denote by $\mu=\mu_N$. We are interested in the convergence to  $\mu$, which we quantify in terms of total-variation mixing times. To do so, we let for a probability measure $\nu$ on $\Omega_N$ 
\begin{equation}\label{def:TVDistance}
\TV{ \nu - \mu } := \frac{1}{2}\sum_{x \in \Omega_N} \abs{\nu(x)-\mu(x)} = \max_{A \subseteq \Omega_N} \left(\nu(A)-\mu(A)\right) 
\end{equation} be the \textbf{total-variation distance} of $\nu$ and $\mu$. We define the $\boldsymbol\varepsilon$\textbf{-mixing time} of $(\eta_t)_{t \geq 0}$ by
\begin{equation}\label{def:MixingTime}
t^N_{\text{\normalfont mix}}(\varepsilon) := \inf\left\lbrace t\geq 0 \ \colon \max_{\eta \in \Omega_{N}} \TV{\P\left( \eta_t \in \cdot \ \right | \eta_0 = \eta) - \mu_N} < \varepsilon \right\rbrace
\end{equation} for all $\varepsilon \in (0,1)$. Our goal is to study the order of $t^N_{\text{\normalfont mix}}(\varepsilon)$ when $N$ goes to infinity. \\

Depending on the boundary parameters $\alpha$ and $\beta$, we distinguish three different phases: For $\alpha < \min( \beta,\frac{1}{2})$ we are in the \textbf{low density phase}, and for $\beta < \min(\alpha,\frac{1}{2})$ we are in the \textbf{high  density phase}. The remaining case $\alpha,\beta \geq \frac{1}{2}$ is called the \textbf{maximal current phase}. 
These names are justified when studying the invariant measure $\mu$; see Section \ref{sec:InvariantAndCurrent} for a more detailed discussion. In short, the average density in $\mu$ stays above $\frac{1}{2}$ in the high density phase, it stays below $\frac{1}{2}$ in the low density phase, and it is close to $\frac{1}{2}$ in the maximal current phase, creating the largest possible flow of particles entering the segment. \\

In the following, we focus on the maximal current phase and provide bounds on the corresponding mixing times. Our first result is an upper bound on the mixing time for the TASEP with open boundaries in the maximal current phase.
\begin{theorem}\label{thm:MaxCurrent}
For all $\alpha,\beta \geq \frac{1}{2}$, and for all $\varepsilon \in (0,1)$, the mixing time of the TASEP with open boundaries satisfies
\begin{equation}
\limsup_{N \rightarrow \infty} \frac{t^N_\mix(\varepsilon)}{N^{3/2}\log(N)} \leq C_1
\end{equation} for some universal constant $C_1>0$.
\end{theorem}

In all other cases of $\alpha,\beta>0$ with $\alpha \neq \beta$, the mixing time of the TASEP with open boundaries is known to be of linear order \cite{GNS:MixingOpen}. When $\alpha=\beta=\frac{1}{2}$, called the \textbf{triple point} as the maximum current phase coincides with the high density phase and  the low density phase of the TASEP with open boundaries, the above result can be refined.

\begin{theorem}\label{thm:Triple}
For $\alpha=\beta = \frac{1}{2}$, and for all $\varepsilon \in (0,1)$,  the mixing time of the TASEP with open boundaries satisfies
\begin{equation}
\limsup_{N \rightarrow \infty} \frac{t^N_\mix(\varepsilon)}{N^{3/2}} \leq C_2
\end{equation} for some constant $C_2=C_2(\varepsilon)>0$.
\end{theorem}

Previously, the best upper bound on the mixing time of the TASEP in the triple point was of order $N^{3}$; see Theorem 1.6 in \cite{GNS:MixingOpen}.
We note that the upper bound on the mixing time in Theorem \ref{thm:Triple} is expected to give the correct order for many interacting particle systems; see Conjecture \ref{conj:Milton}. The next result states that the TASEP with open boundaries in the maximal current phase  has a mixing time of order at least $N^{3/2}$.
\begin{theorem}\label{thm:LowerBound}
For all $\alpha,\beta \geq \frac{1}{2}$,  and for all $\varepsilon \in (0,1)$, we have that the mixing time of the TASEP with open boundaries satisfies
\begin{equation}
\liminf_{N \rightarrow \infty} \frac{t^N_\mix(\varepsilon)}{N^{3/2}} \geq C_3
\end{equation} for some constant $C_3=C_3(\varepsilon)>0$.
\end{theorem}
We remark that a lower bound of order $N^{3/2}$ may be deduced using results on the relaxation time of the TASEP with open boundaries from the physics literature; see \cite{GE:ExactSpectralGap,GE:BetheAnsatzPASEP}, and a well-known connection between relaxation and mixing times; see Theorem 12.5 in \cite{LPW:markov-mixing}.  However, we provide a proof of Theorem \ref{thm:LowerBound} which does not rely on spectral techniques.

\subsection{Related work} \label{sec:RelatedWork}

Mixing times for exclusion processes are a topic of lots of recent interest; see Chapter 23 in \cite{LPW:markov-mixing} for a general introduction. In \cite{L:CutoffSEP},  Lacoin studies the symmetric simple exclusion process on the segment, in which particles perform simple random walks under the exclusion rule. He determines the first-order mixing time of the symmetric simple exclusion process and proves cutoff, a sharp transition for the mixing time. A lower bound giving the correct first-order term was previously shown by Wilson in \cite{W:MixingLoz}. For the asymmetric simple exclusion process, including the TASEP with closed boundaries $\alpha=\beta=0$, Labbé and Lacoin prove cutoff in \cite{LL:CutoffASEP}. Moreover, for several instances of exclusion processes, the limit profiles of the mixing time are studied; see  \cite{BN:CutoffASEP,LL:CutoffASEP,L:CycleDiffusiveWindow}. Very recently, mixing times are investigated for simple exclusion processes with a size-dependent bias, and in random environments \cite{LL:CutoffWeakly,LY:RandomEnvironment,S:MixingBallistic}.  \\

In general, many techniques to obtain sharp bounds on the mixing times of exclusion processes rely on reversibility of the process. In the absences of reversibility, mixing times for the asymmetric simple exclusion process on the circle are investigated in \cite{F:EVBoundsSEP}. Further, in \cite{GNS:MixingOpen}, mixing times for the simple exclusion processes with open boundaries are studied, and several different regimes of mixing times are identified. In particular, the mixing time in the high density phase and in the low density phase is linear in the size of the segment. Note that when allowing for open boundaries, the invariant measure of the simple exclusion process has, in general, not a simple closed form. During the last decades, elaborate tools like the Matrix product ansatz and staircase tableaux were developed to provide useful representations of the invariant measure \cite{BE:Nonequilibrium,CW:TableauxCombinatorics,DEHP:ASEPCombinatorics,M:TASEPCombinatorics}. \\

When restricting our attention to a totally asymmetric simple exclusion process, the set of available techniques increases drastically due to an alternative representation of the one-dimensional TASEP as an exponential corner growth model on $\Z^2$. Historically, the formulation as a CGM was a crucial tool in the seminal work by Rost to derive hydrodynamic limits, in the sense of a rigorous connection to the Burgers equation \cite{R:PDEresult}. 
This correspondence allows to equivalently study last passage times and geodesics in the exponential CGM to gain insights on the  interacting particle system.  Note that many observables and properties have equivalent expressions; see also Sections \ref{sec:TASEPInteracting} and \ref{sec:TASEPLPP}. \\

The current in the TASEP can be phrased using last passage times, for which sharp asymptotic results, in terms of Tracy-Widom limit laws, are known \cite{J:KPZ,PS:CurrentFluctuations,S:CouplingMovingInterface}. Second class particles in the TASEP have a natural interpretation as competition interfaces in the CGM representation. This correspondence was first observed by Ferrari and Pimentel in~\cite{FP:CompetitionInterface} for the TASEP in the rarefaction fan, and later extended  \cite{CP:BusemannSecondClass,FMP:CompetitionInterface,GRS:GeodesicCompetition}. 
In particular, this equivalence allows to prove a law of large numbers for second class particles using asymptotic results for Busemann functions \cite{CP:BusemannSecondClass,S:ExistenceGeodesics,W:Semiinfinite}. When starting the TASEP from an invariant product measure, a version of Burke’s property gives rise to a stationary corner growth model \cite{BCS:CubeRoot,S:LectureNotes}. Note that all of these correspondences are usually described between the TASEP on the full line and the CGM on $\Z^2$ with i.i.d.\ Exponential-$1$-weights. A  stationary TASEP on $\N$ with an open boundary is investigated \cite{PS:CurrentFluctuations}, and more recently in \cite{BBCS:Halfspace,BFO:HalfspaceStationary}.  In Section \ref{sec:TASEPLPP}, we provide the above equivalences for the TASEP on a segment with open boundaries; see also \cite{KT:DIrectedPolymerWall} for similar ideas in the physics literature.  \\

A huge benefit of the corner growth representation is the fact that it can be analyzed with tools from integrable probability. This means that we are endowed with exact formulas for many observables, for example 
scaling limits, involving the exponent $2/3$ of the KPZ relaxation scale, for the current of the TASEP on the integers when starting from a shock, or for the TASEP on the torus  \cite{BL:TASEPring,BC:PSConjecture,FS:SpaceTimeCovariance,GS:SixVertex,L:HeightOnRing}. In particular, the TASEP is one of few models which is provably in the KPZ universality class. Further, it is shown in a seminal work by Corwin and Shen that the KPZ equation arises for the TASEP in the triple point under a suitable weakly asymmetric scaling \cite{CS:OpenASEPWeakly}. However, note that many of these exact expressions for the TASEP tend to fail under minor changes of the setup.
This led to a series of works, including \cite{BG:TimeCorrelation,BGH:AreaConstraint,BGHH:Watermelon,
BGZ:TemporalCorrelation,BSS:Invariant,BSS:Coalescence,P:Duality,Z:OptimalCoalescence},  where for example the effect of changing one rate in the TASEP, known as the \textit{slow bond problem}, or the correlation of last passage times and the coalescence of geodesics are studied. The articles have in common that they use results from integrable probability and combine them with probabilistic concepts; see also \cite{BBS:NonBiinfinite,SS:Coalescence} for similar results avoiding the connection to integrable probability entirely.

\subsection{Intuition for the $N^{3/2}$-mixing time using the TASEP on the integers}\label{sec:Heuristics}

Before presenting the proof of our main results, let us give some heuristics for a mixing time of order $N^{3/2}$ in the maximal current phase. One possibility is to argue with the scaling exponent in the KPZ regime, but this  requires more prerequisites. Instead, we present two direct approaches comparing the TASEP with open boundaries to the TASEP on the integers. \\

Intuitively, we can view the segment $[N]$ as a part of the integer lattice, and extend the exclusion dynamics. A canonical way of extending the stationary TASEP in the maximal current phase to the integers is to start from a Bernoulli-$\frac{1}{2}$-product measure. This is  justified by Lemma \ref{lem:invariant} where we discuss properties of the invariant measure in the maximal current phase. On the integers, let us add a perturbation at the origin. Formally, this corresponds to placing a second class particle at site $0$.  Using the results by Prähofer and Spohn for the TASEP, and by Balázs and Seppäläinen for a general asymmetric setup, one sees that the time for a perturbation to travel a distance $n$ will be of order $n^{3/2}$  \cite{BS:OrderCurrent,PS:CurrentFluctuations}. This suggests that on the segment, the expected time for a perturbation to exit is at most of order $N^{3/2}$. Since the segment contains at most $N$ perturbations, respectively second class particles, this indicates a mixing time of order $N^{3/2}$ up to logarithmic corrections. \\ 

To see that we expect the leading order to be precisely $N^{3/2}$, we consider a different embedding using the TASEP speed process on $\Z$; see \cite{AAV:TASEPSpeed}, and \cite{BB:ColorPosition} for a recent approach using Hecke algebras. In this process, we start with all sites being occupied. A particle initially at site $x\in \Z$ receives label $x$. The particles now perform the TASEP dynamics, i.e., a particle at site $y$ with label $x$ swaps at rate $1$ with the particle at site $y+1$ of label $z$ whenever $x<z$. 
Treating the particles with labels $1$ to $N$ as perturbations of the system, we have to wait until these particles are at sites $>N$ or $<1$ to mix. Using a shift-invariance argument, it suffices to follow the speed of the particle with label $1$. It turns out that the particle has almost surely a speed $U_1$ which is uniformly chosen from $[-1,1]$. Naively, one may expect that the speeds are independent and that it takes a time of order $N^2$ until all perturbations exit. However, this is \textit{not} the case. As shown by Amir et al.\ in Theorem 1.8 of \cite{AAV:TASEPSpeed}, the labels $C_1 \subseteq \N$ of particles with the same speed $U_1$ as particle $1$ have the same law as the times of last increase of a random walk conditioned to remain positive with increments distribution
\begin{equation}
\P(X=1)=\P(X=-1)=\frac{1-U^2_1}{4} \quad \text{ and } \quad \P(X=0)= \frac{1+U^2_1}{2} \, .
\end{equation} In particular, the probability that $x\in C_1$ for some $x\in \N$ is of order $x^{-1/2}$. Hence, there are order $\sqrt{N}$ many particles among the particles of labels $2$ to $N$ with the same speed as the particle of label $1$. 
Using \cite{FP:CompetitionInterface} to coupling the motion of particle $1$ with a competition interface in last passage percolation, we expect that the position $X_t^1$ at time $t$ of the particle with label $1$ is of order $U_1 t$ and satisfies
\begin{equation}\label{eq:ClaimSpeed1}
 \Var(X_t^{1}) \sim c t^{4/3} 
\end{equation} for all $t \geq 0$, and some constant $c>0$. We formally introduce and discuss competition interfaces in Section \ref{sec:CompetitionInterfaces} for the TASEP with open boundaries. Moreover, conditioning on the speed $U_1$ of the particle of label $1$, we expect that a particle of label $y$ with speed $U_1$ follows the same competition interface as particle $1$, implying that the position $X_t^{y}$ at time $t$ of the particle with label $y$ satisfies
\begin{equation}\label{eq:ClaimSpeedY}
\E[X_t^{y}]  \sim  y+ U_1 t \quad \text{and} \quad  \Var(X_t^{y}) \sim c t^{4/3} 
\end{equation}
for all $t>0$, and some constant $c>0$. Thus, it suffices to consider one particle in each cohort of the same speed, and bound the time until it is at sites $>c^{\prime}N$ or $<-c^{\prime}N$ for sufficiently large constant $c^{\prime}>0$. Recall that each such cohort has a size of order $\sqrt{N}$. Since the speed is uniform on $[-1,1]$, there are at most order $x$ many cohorts of a speed in $[-(x+1)N^{-1/2},-xN^{-1/2}] \cup [xN^{-1/2},(x+1)N^{-1/2}]$. The claim that all particles with labels $1$ to $N$ are at sites $>N$ or $<1$ after order $N^{3/2}$ many steps follows from \eqref{eq:ClaimSpeed1} and \eqref{eq:ClaimSpeedY}, together with a union bound over $x \in [N^{1/2}]$.

\subsection{Outline of the paper} \label{sec:OutlinePaper}

This paper is structured as follows. In the remainder of this section, we state open questions on mixing times for exclusion processes. In Section~\ref{sec:TASEPInteracting}, we collect basic properties of the TASEP with open boundaries, seen as an interacting particle system. In Section~\ref{sec:TASEPLPP}, we give a representation of the TASEP with open boundaries as a corner growth model on the strip and describe how well-known concepts for corner growth models transfer to our setup.
In particular, we express second class particles by competition interfaces; see \cite{FMP:CompetitionInterface,FP:CompetitionInterface}. In Section \ref{sec:MixingTimesMaxCurrent}, we prove Theorem \ref{thm:MaxCurrent} by arguing that the exit time of the competition interface can be bounded by the time it takes for certain geodesics in last passage percolation (LPP) on the strip to coalesce. To estimate the coalescence times, we use a path decomposition together with moderate deviation estimates by Basu et al.\  for LPP on the strip and a connection to the stationary CGM on the half-quadrant; see~ \cite{BBCS:Halfspace,BG:TimeCorrelation,BGZ:TemporalCorrelation,BSV:SlowBond}. We show that every path with a large passage time can with high probability be modified such that we obtain a heavy path between any  pair of sites close to its endpoints. Following the ideas in \cite{DJP:CriticalStripe}, we see that every pair of these heavy paths, and hence also the geodesics, must with high probability intersect. In Section \ref{sec:MixingTimesTriplePoint}, we show Theorem~\ref{thm:Triple} on mixing times in the triple point.
We use a form of duality for the stationary CGM on the strip, which allows us to bound the hitting time of the boundary for semi-infinite geodesics using exit times for competition interfaces. 
In Section \ref{sec:LowerBounds}, we give a short proof of the lower bounds on the mixing time in Theorem \ref{thm:LowerBound} using the results of Section~\ref{sec:MixingTimesMaxCurrent}.

\subsection{Open questions}

We conclude the introduction by highlighting two open problems on mixing times for exclusion processes. The first open question concerns the asymmetric simple exclusion process with open boundaries. In this model, the particles on the segment have a drift, but may move to both directions within the segment, and may enter and exit at both endpoints. Again, one can define a range of parameters such that the exclusion process is in the maximal current phase. The following open question is Conjecture 1.7 in \cite{GNS:MixingOpen}, and we refer to Section 1.2 of \cite{GNS:MixingOpen} for a more comprehensive discussion of open problems on the simple exclusion process with open boundaries.
\begin{conjecture} The mixing time of the asymmetric simple exclusion process in the maximal current phase is of order $N^{3/2}$, and the cutoff phenomenon does not occur.
\end{conjecture}  More generally, it is believed that the mixing time of order $N^{3/2}$ is present in a much broader setup. The following conjecture is due to Milton Jara (personal communication):
\begin{conjecture}\label{conj:Milton} 
Consider a particle conserving system on the ring, which has only (nice) local interactions, and assume that the second derivative of the average current does not vanish. Then the mixing time is of order $N^{3/2}$,   up to poly-logarithmic corrections.
\end{conjecture} Note that the TASEP $(\eta_t)_{t\geq 0}$ with open boundaries can be thought of as a particle conserving system on a ring by considering the process $(\eta_t,\mathbf{1}-\eta_t)_{t\geq0}$; see Section 3 in \cite{GNS:MixingOpen}.

\section{The TASEP with open boundaries as an interacting particle system} \label{sec:TASEPInteracting}
In the following, we study the TASEP with open boundaries as an interacting particle system; see also \cite{L:Book2} for a general introduction. We collect basic definitions and properties of the TASEP with open boundaries from probability theory, statistical mechanics and combinatorics. This includes the canonical coupling, second class particles, invariant measures and the current. 
\subsection{Canonical coupling and second class particles} \label{sec:CanonicalCoupling}

We start with a joint coupling for the TASEP with open boundaries for all possible initial states and boundary parameters. Let $(\eta_t)_{t \geq 0}$ and $(\zeta_t)_{t \geq 0}$ be two TASEPs with open boundaries on a segment of size $N$ with respect to parameters $\alpha,\beta>0$ and $\alpha^{\prime},\beta^{\prime}>0$. We assume that $\alpha\geq \alpha^{\prime}$ and $\beta \leq\beta^{\prime}$ holds, as the construction of the coupling  in the other cases is similar. 
We choose the same rate $1$ Poisson clocks for each site $x\in [N-1]$ in the segment. Whenever a clock at some site $x$ rings, we attempt to move a particle from $x$ to $x+1$ in both processes under the exclusion constraint. At the boundaries, we consider independent rate $\alpha^{\prime}$ and rate $\beta$ Poisson clocks. 
Whenever the rate $\alpha^{\prime}$ Poisson clock rings, we place a particle in both processes at site $1$. 
Whenever the rate $\beta$ Poisson clock rings, we place in both processes an empty site at site $N$. 
If $\alpha\neq \alpha^{\prime}$, consider in addition an independent rate $(\alpha-\alpha^{\prime})$ Poisson clock, and place a particle at site $1$ in $(\eta_t)_{t \geq 0}$ whenever this clock rings. Similarly, if $\beta\neq \beta^{\prime}$, we consider an independent rate $(\beta^{\prime}-\beta)$ Poisson clock and place an empty site in $(\zeta_t)_{t \geq 0}$ at site $N$ whenever this clock rings. \\

 Let $\tilde{\mathbf{P}}$ denote the law of this coupling, called the \textbf{canonical coupling}. The following lemma is immediate from this construction; see also Lemma 2.1 in \cite{GNS:MixingOpen} for a similar construction for general simple exclusion processes with open boundaries.
\begin{lemma}\label{lem:MonotoneCanonicalCoupling} Let  $(\eta_t)_{t \geq 0}$ and $(\zeta_t)_{t \geq 0}$ be two TASEPs with open boundaries on a segment of size $N$ with respect to parameters $\alpha,\beta>0$ and $\alpha^{\prime},\beta^{\prime}>0$. Assume that $\alpha\geq \alpha^{\prime}$ and $\beta \leq\beta^{\prime}$ holds. Then under the canonical coupling $\tilde{\mathbf{P}}$, we have that
\begin{equation}
\tilde{\mathbf{P}} \left( \eta_t(x) \geq \zeta_t(x) \text{ for all } x \in [N]  \text{ and } t\geq 0 \mid \eta_0(y) \geq \zeta_0(y)\text{ for all } y \in [N]  \right) = 1 \, .
\end{equation}
\end{lemma}
Using the canonical coupling, we define the \textbf{disagreement process}  $(\xi_t)_{t \geq 0}$ between two TASEPs $(\eta_t)_{t \geq 0}$ (with parameters $\alpha,\beta>0$) and $(\zeta_t)_{t \geq 0}$ (with parameters $\alpha^{\prime},\beta^{\prime}>0$) on the segment of size $N$ as the process taking values in $\{0,1,2\}^{N}$ given by
\begin{equation}
\xi_t(x) := \mathds{1}_{\eta_t(x)=\zeta_t(x)=1} + 2 \mathds{1}_{\eta_t(x)\neq \zeta_t(x)}
\end{equation} for all $x \in [N]$ and $t\geq 0$. We say that site $x$ is occupied by a \textbf{first class particle} at time $t$ if $\xi_t(x)=1$. Similarly, we say that site $x$ is occupied by a \textbf{second class particle} at time $t$ if $\xi_t(x)=2$.  When $\alpha\geq \alpha^{\prime}$ and $\beta \leq\beta^{\prime}$, as well as $\eta_0(y) \geq \zeta_0(y)$ for all $y \in [N]$ holds, the canonical coupling $\tilde{\textbf{P}}$ between $(\eta_t)_{t \geq 0}$ and $(\zeta_t)_{t \geq 0}$, and Lemma \ref{lem:MonotoneCanonicalCoupling} ensure that the disagreement process $(\xi_t)_{t \geq 0}$ is given as the following Markov chain. \\

In the canonical coupling between $(\eta_t)_{t \geq 0}$ and $(\zeta_t)_{t \geq 0}$, we assign rate $1$ Poisson clocks to all sites $x\in[N-1]$, a rate $\alpha$ Poisson clock to the left boundary, and a rate $\beta$ Poisson clock to the right boundary. Further, if $\alpha\neq\alpha^{\prime}$, we assign a rate $(\alpha-\alpha^{\prime})$ Poisson clock to the left boundary, and if $\beta\neq\beta^{\prime}$ a rate $(\beta^{\prime}-\beta)$ Poisson clock to the right boundary. 
Suppose that at time $t$, the Poisson clock assigned to a site $x$ rings. Then we obtain the configuration $\xi_t$ from $\xi_{t_-}$ as follows. If $\xi_{t_-}(x)=1$ and $\xi_{t_-}(x+1)=2$ holds, or if $\xi_{t_-}(x)=2$ and $\xi_{t_-}(x+1)=0$, then we set
\begin{equation}
\xi_t(y) := \begin{cases} \xi_{t_-}(x+1) & \text{ if } y=x \\
\xi_{t_-}(x) & \text{ if } y=x+1 \\
\xi_{t_-}(y) & \text{ if } y \notin \{x,x+1\} \\
\end{cases}
\end{equation} for all $y\in [N]$. Otherwise, we set $\xi_{t}=\xi_{t_-}$. Suppose that at time $t$, the rate $\alpha$ Poisson clock assigned to the left boundary rings. Then we place a first class particle site $1$, i.e.\ we let $\xi_t(1)=1$, and leave the configuration unchanged at all other sites. Suppose that at time $t$, the rate $(\alpha^{\prime}-\alpha)$ Poisson clock assigned to the left boundary rings. If $\xi_{t_-}(1)=0$ holds, then set $\xi_t(1)=2$, and leave the configuration unchanged at all other sites. Otherwise, we set $\xi_{t}=\xi_{t_-}$.
Similarly, when the rate $\beta^{\prime}$ Poisson clock rings, we set $\xi_t(N)=0$ and leave the remaining configuration unchanged, whereas when the $(\beta-\beta^{\prime})$ Poisson clock rings, we set $\xi_t(N)=2$, provided that $\xi_{t_-}(N)=1$, and consider $\xi_{t}=\xi_{t_-}$ otherwise.   \\

The next lemma is a standard result for an upper bound on the mixing time using the exit times of second class particles in the disagreement process; see Corollary 2.5 in \cite{GNS:MixingOpen}.

\begin{lemma}\label{lem:MixingBoundByExiting}
Let $(\eta^{\mathbf{0}}_t)_{t \geq 0}$ and $(\eta^{\mathbf{1}}_t)_{t \geq 0}$ be two TASEPs with open boundaries on $[N]$ for some $N \in \N$, and the same boundary parameters $\alpha,\beta>0$,  in the canonical coupling $\tilde{\mathbf{P}}$ started from the all empty and the all full states $\mathbf{0}$ and $\mathbf{1}$, respectively. Let $(\xi_t)_{t \geq 0}$ be their disagreement process, and denote by $\tau$ the first time at which $(\xi_t)_{t \geq 0}$ contains no second class particles. If  $\tilde{\mathbf{P}}(\tau>s) \leq \varepsilon$ holds for some $\varepsilon>0$ and $s\geq 0$, then the mixing time of the TASEP with open boundaries satisfies $t^N_{\text{\normalfont mix}}(\varepsilon)\leq s$.
\end{lemma}

\subsection{Invariant measures and current}\label{sec:InvariantAndCurrent}

Recall the stationary distribution $\mu=\mu_N$ from Section \ref{sec:ModelResults}. When $\alpha=\beta=\frac{1}{2}$ holds, it is known that $\mu_N$ is the Uniform distribution on $\Omega=\{0,1\}^{N}$; see Proposition 2 in \cite{BCEPR:CombinatoricsPASEP}.  For general parameters $\alpha,\beta\geq \frac{1}{2}$, the invariant measure $\mu$ does not have a simple closed form, and it is a challenging task to find good  representations for $\mu$; see   \cite{CW:TableauxCombinatorics,DEHP:ASEPCombinatorics,M:TASEPCombinatorics}. The next lemma is a consequence of a seminal result due to Derrida et al.\ when $\alpha,\beta>1/2$, and due to Bryc and Wang for the entire maximum current phase \cite{BW:Density,DEL:DensityExcursion}. Their results give a precise description of the fluctuations of the number of particles in the invariant measure $\mu_N$. However, we only require the following simple consequence of Theorems 1.2 and 1.3 in \cite{BW:Density}.

% The following result originally goes back to Liggett \cite{L:ErgodicI} for $\alpha+\beta=1$, and can be found as Lemma 2.11 in \cite{GNS:MixingOpen} for general parameters.
\begin{lemma}\label{lem:invariant} Fix some $\alpha,\beta \geq \frac{1}{2}$. Let $c_1,c_2 \in [0,1]$ with $c_1 < c_2$. Then for all $s\in \R$, we have that the invariant measures $(\mu_N)_{N \in \N}$ satisfy
\begin{equation}\label{eq:InvaraintCLT}
\liminf_{N \rightarrow \infty} \mu_N\left( \frac{1}{\sqrt{N}}\sum_{i=c_1N}^{c_2N} \Big(\eta(i)  -  \frac{1}{2}\Big) > s \right) =  \P(Z > s) 
\end{equation} 
for some explicitly known random variable $Z$ whose law depends only on $c_1,c_2,\alpha$, and $\beta$.
\end{lemma}
\begin{remark}A similar result holds in the high and low density phase,  replacing the normalization $1/2$ in \eqref{eq:InvaraintCLT} by $(1-\beta)$ and $\alpha$, respectively. 
\end{remark}
Next, we consider the \textbf{current} $(\mathcal{J}^N_t)_{t \geq0}$ of particles in the TASEP with open boundaries on a segment of length $N$. Here, $\mathcal{J}^N_t$ denotes the number of times a particle was created at site $1$ until time $t\geq 0$. It is a classical result due to Derrida et al.\ in \cite{DEHP:ASEPCombinatorics} that the current satisfies a law of large numbers, where for any possible initial configuration, and any choice of the parameters $\alpha,\beta \geq \frac{1}{2}$
\begin{equation}
\lim_{t \rightarrow \infty} \frac{\mathcal{J}^N_t}{t} =    J_N
\end{equation} holds almost surely for some (deterministic) sequence $(J_N)_{N \in \N}$ with $\lim_{N \rightarrow \infty} J_N = \frac{1}{4}$. In the following, we are interested in the order of fluctuations for the current starting with empty initial conditions. For $\alpha=\beta=\frac{1}{2}$, Derrida et al.\ argue in \cite{DEM:DiffusionConstant}  that for fixed $N$, the variance of the current satisfies
\begin{equation*}%\label{eq:DerridaBound}
\lim_{t \rightarrow \infty} \frac{\Var(\mathcal{J}^N_t)}{t} = \Delta_N = \frac{1}{4\sqrt{\pi N}} \, ,
\end{equation*} where $ \Delta_N$ is called the diffusion constant. An inspection of their arguments suggests 
\begin{equation}%\label{eq:DerridaBound2}
\Var(\mathcal{J}^N_t) \leq C \Delta_N  (t +t^N_{\textup{rel}})
\end{equation} for all $t \geq0$ and some constant $C>0$, where $t^N_{\textup{rel}}$ denotes the relaxation time, i.e., the inverse spectral gap, of the TASEP with open boundaries on the segment of length $N$.  In \cite{GE:ExactSpectralGap,GE:BetheAnsatzPASEP}, the authors argue that $t^N_{\textup{rel}}$ is of order $N^{3/2}$ in the maximal current phase, which implies that at a  time of order $N^{3/2}$, the variance of the current for the TASEP with open boundaries on the segment of length $N$ should be of order $\sqrt{N}$. This motivates the following proposition.

%The following proposition shows that \eqref{eq:DerridaBound}  yields the correct order of fluctuations when $t$ is of order $N^{3/2}$, and when we start from equilibrium.
\begin{proposition}\label{pro:VarianceCurrent} Let $\alpha,\beta \geq \frac{1}{2}$, and let $t_N=cN^{3/2}$ for some constant $c>0$. Then for every $\varepsilon>0$, there exists some constant $C=C(c,\varepsilon)>0$ such that 
\begin{equation}
\liminf_{N \rightarrow \infty} \P\left(  \mathcal{J}^N_{t_N} - \frac{1}{4} t_N \in [ - C \sqrt{N}, C \sqrt{N} ] \ \Big| \ \eta_0 = \mathbf{0} \right) \geq 1 - \varepsilon \, .
\end{equation}
\end{proposition}
Once we built up the framework of last passage percolation (LPP) on the strip in Section~\ref{sec:TASEPLPP}, 
we will see that Proposition~\ref{pro:VarianceCurrent} follows directly from Proposition~\ref{pro:VarianceUpperBoundRephased} by rephrasing the above statement on the current in terms of last passage times; see also Remark~\ref{rem:VarianceCurrent}.
%\begin{remark} An inspection of the arguments in \cite{DEM:DiffusionConstant}  yields that $\Var(\mathcal{J}_t) $ should satisfy
%\begin{equation}
%\Var(\mathcal{J}_t) \leq C \Delta_N  (t +t_{\textup{rel}})
%\end{equation} for all $t \geq0$ and some constant $C>0$, where $t_{\textup{rel}}$ denotes the relaxation time of the TASEP with open boundaries. In \cite{GE:BetheAnsatzPASEP,GE:ExactSpectralGap}, the authors argue that $t_{\textup{rel}}=\Theta(N^{3/2})$ in the maximal current phase, which in total provides an intuition for  the above proposition.
%\end{remark}

\section{The TASEP with open boundaries as a corner growth model}\label{sec:TASEPLPP}

We will now give a description of the TASEP with open boundaries as an exponential corner growth model (CGM) on a strip in $\Z^2$. The presented objects and statements are well-known for the TASEP on the integers; see \cite{S:LectureNotes} for an overview,  and we will in the following provide the corresponding definitions and theorems for the TASEP with open boundaries.

\subsection{Construction of the TASEP with open boundaries as a CGM}\label{sec:DefinitionTASEPLPP}

Consider the integer lattice $\Z^2$, and a function $p \colon \Z^2 \rightarrow [0,\infty)$. We let $(\omega^p_x)_{x \in \Z^2}$ be a family of independent Exponential-distributed random variables, where $\omega^p_x$ has mean $p(x)$, i.e., we have $\P(\omega^p_x >s)=\exp(-p(x)^{-1}s)$ for all $s\geq 0$. We use the convention that $\omega^p_x \equiv 0$ if $p(x)=0$. For $x=(x_1,x_2) \in \Z^2$, set $|x|_1:= |x_1|+|x_2|$. For $x,y \in \Z^{2}$ with $y\succeq x$ with respect to the component-wise ordering $\succeq$, we say that $\pi_{x,y}$ is an up-right \textbf{lattice path} from $x$ to $y$ if
\begin{equation*}
\pi_{x,y} = \{ z^0=x, z^{1},\dots, z^{|y-x|_1}=y \, \colon \, z^{i+1}-z^{i} \in \{ \eone,\etwo\} \text{ for all } i \} \, ,
\end{equation*} where $\eone:=(1,0)$ and $\etwo := (0,1)$. We denote by $\Pi_{x,y}$ the set of all lattice paths connecting $x$ to $y$. For $x,y \in \Z^{2}$ and a function $p$, we denote the \textbf{passage time} of a lattice path $\pi_{x,y}$ by
\begin{equation}\label{def:PathLPT}
G_p(\pi_{x,y}):=  \sum_{z \in \pi_{x,y}} \omega^p_z \, ,
\end{equation}
and let $G_p(x,y)$ be the \textbf{last passage time} between $x$ and $y$ with
\begin{equation*}
G_p(x,y):= \max_{\pi_{x,y} \in \Pi_{x,y}} G_p(\pi_{x,y})
\end{equation*} if $y \succeq x$, and $G_p(x,y):= -\infty$ otherwise.  A key property of $G_p$ is its monotonicity, i.e., for two functions $p,\tilde{p} \colon \Z^2 \rightarrow [0,\infty)$ with $p(x) \geq \tilde{p}(x)$ for all $x$, there exists a coupling $\mathbf{P}$ with
\begin{equation}\label{eq:monoticityGeodesics}
\mathbf{P}\left( G_p(x,y) \geq G_{\tilde{p}}(x,y) \text{ for all } x,y \in \Z^2\right) = 1 \, .
\end{equation} 
We will frequently use the monotonicity  throughout the paper. Furthermore, we have the following  notions and conventions: Let $\mathcal{S}= \mathcal{S}(p):=\{x \in \Z^2 \colon p(x)\neq 0\}$ be the \textbf{support} of $p$. We will write $G(x,y)$ for $G_p(x,y)$ when the function $p$ is clear from the context. For $x=(n,n)$ and $y=(m,m)$ for some $n,m \in \N$, let $G(n,m):=G(x,y)$. When $n,m$ are not integer-valued, we take $\lfloor n \rfloor,\lfloor m \rfloor$ without explicitly mentioning. Further, for finite $A,B \subseteq \Z^{2}$, we set
\begin{equation}\label{def:LPTSets}
G(A,B):= \max_{x\in A,y\in B} G(x,y) \, .
\end{equation}

Next, in order to provide a connection to the TASEP with open boundaries, we assign labels to the particles as they enter the segment, i.e., the first particle entering gets label~$1$, the next particle label $2$, and so on. Similarly, the left-most particle in the segment at time $0$ receives label $0$, the next particle label $-1$, and so on. For $u \in \N$ and $t\geq 0$, let $\ell_t(u)$ be the position of the particle with label $u$ at time $t$. We set  $\ell_t(u)=\infty$ if the particle with label $u$ has left the segment by time $t$, and $\ell_t(u)=-\infty$ if label $u$ was not yet assigned to a particle at time $t$. For an initial state $\eta_0 \in \Omega_N$, we denote by $\gamma_0=\{\gamma_0^{i} \in \Z^2 \colon i \in \{0,\dots,N\}\}$ the \textbf{initial growth interface}  with $\gamma^0_0:= (0,0)$, and recursively
\begin{equation}\label{def:GrowthInterface}
\gamma_0^{i} := \begin{cases} \gamma_0^{i-1} + \eone & \text{ if } \eta_0(i)=0 \\
 \gamma_0^{i-1} - \etwo & \text{ if } \eta_0(i)=1 \, .
\end{cases}
\end{equation} Further, we let
\begin{equation*}
\gamma_t := \{ x \in \Z^2 \colon G(\gamma_0,x) \leq t \text{ and }  G(\gamma_0,x+(1,1)) > t\} 
\end{equation*} for all $t \geq 0$. The process $(\gamma_t)_{t \geq 0}$ is called the \textbf{growth interface}. The next lemma describes a coupling between the TASEP with open boundaries and last passage times. As for last passage percolation with $p\equiv 1$, this coupling  is a classical result due to Rost \cite{R:PDEresult}, we will only give a sketch of the proof.
\begin{lemma}\label{lem:CurrentVsGeodesic} Let $\alpha,\beta\geq \frac{1}{2}$. Consider the function $p$ given for all $x=(x_1,x_2)\in \Z^2$ by
\begin{equation}\label{def:EnvironmentStripe}
p(x) := \begin{cases} \alpha^{-1} & \text{ if } x_1=x_2 > 0 \\
\beta^{-1} & \text{ if } x_1- N= x_2  \\ %\text{ and } x \succeq y_0^{N}+(1,1)  \\
1& \text{ if }   x_2  \in [x_1-N+1,x_1-1] \\ % \geq 0  %\text{ and } x \succeq y+(1,1) \text{ for some } y \in \gamma_0  \\
0 & \text{ otherwise} \, .
\end{cases}
\end{equation}
Then there exists a coupling such that for all $u\in \N$, $n\in [N] \cup \{0\}$ and $t \geq 0$
\begin{equation}\label{eq:LPPvsTASEP}
\left\{\ell_t(u) > n \right\} =  \left\{G_{p}(\gamma_0, (u,u+n)) \leq t \right\} \, .
\end{equation} In other words, the event that the particle with label $u$ is at some position $>n$ at time $t$ agrees with the event that the last passage time between $\gamma_0$ and $(u,u+n)$ is at most $t$.
\end{lemma}
\begin{proof}[Sketch of the proof] For $x=(u,u+n)$ with $n \in \{1,\dots,N-1\}$, let $\omega_x^{p}$ be the time it takes for particle $u$ at $n$ to perform a jump after the first time that the site at position $n+1$ is empty. For $n=0$, let $\omega_x^{p}$ be the time it takes until label $u$ gets assigned to a particle after the first time~$t$ at which $\ell_t(u-1) > 1$. For $n=N$, let $\omega_x^{p}$ be the time until the particle with label $u$ exits the segment after the first time $t$ at which $\ell_t(u)=N$ holds. In all other cases, choose $\omega_x^{p}$ independently according to the correct marginals. Note that all random variables $(\omega_x^{p})_{x \in \Z^{2}}$ have the correct marginal distributions; see also Figure \ref{fig:environment}. To see that the relation \eqref{eq:LPPvsTASEP} holds, note that the last passage times satisfy the recursion 
\begin{equation*}
G_{p}(\gamma_0, (u,u+n)) =  \max\{ G_{p}(\gamma_0, (u-1,u+n)), G_{p}(\gamma_0, (u,u+n-1)) \} + \omega_x^{p} \, .
\end{equation*} Furthermore, notice that the particle of label $u$ will jump from site $n$ to site  $n+1$ after time $\omega_{x}^p$ plus the maximum of the time it takes until the particle with label $u-1$ jumped to site $n+2$, or was removed from the segment if $n=N-1$, and the time at which the particle of label $u$ jumped from site $n-1$ to $n$. Combining these two observations, the relation \eqref{eq:LPPvsTASEP} follows now by induction. 
\end{proof}

\begin{figure}
\begin{center}
\begin{tikzpicture}[scale=0.53]

\filldraw [fill=red!20] (1,1) rectangle ++(8,1);
\filldraw [fill=red!20] (2,2) rectangle ++(8,1);
\filldraw [fill=red!20] (3,3) rectangle ++(6,1);
\filldraw [fill=red!20] (4,4) rectangle ++(4,1);
\filldraw [fill=red!20] (5,5) rectangle ++(3,1);
\filldraw [fill=red!20] (6,6) rectangle ++(1,1);

\foreach \x in{1,...,8}{
	\draw[gray!50,thin](\x,\x) to (7+\x,\x); 
	\draw[gray!50,thin](\x,1) to (\x,\x);
	\draw[gray!50,thin](7+\x,\x) to (7+\x,8);  }
	
	\foreach \x in{1,...,7}{
	
	\draw[black!70,thick](\x,\x) -- (\x+1,\x) -- (\x+1,\x+1) -- (\x,\x+1)-- (\x,\x);
	\draw[black!70,thick](\x+7,\x) -- (\x+1+7,\x) -- (\x+1+7,\x+1) -- (\x+7,\x+1)-- (\x+7,\x);
	}
	
\draw[darkblue,line width =2pt] 	 (6.5,6.5) -- ++(0,-1) -- ++(1,0) -- ++(0,-2) -- ++(1,0) -- ++ (0,-1) -- ++ (1,0) ;
	
\draw[darkblue,line width =2pt,dashed] 	 (6.5,6.5) -- ++(0,1.5);	
\draw[darkblue,line width =2pt,dashed] 	 (9.5,2.5) -- ++(1.5,0);

	\node[scale=0.9]  (x1) at (1.53,1.55){$\alpha^{\text{-}1}$} ;
	\node (x2) at (2.5,1.5){$1$} ;
	\node (x3) at (3.5,1.5){$1$} ;
	\node (x4) at (4.5,1.5){$1$} ;
	\node (x5) at (5.5,1.5){$1$} ;
	\node (x6) at (6.5,1.5){$1$} ;
	\node (x7) at (7.5,1.5){$1$} ;
	\node[scale=0.9] (x8) at (8.53,1.5){$\beta^{\text{-}1}$} ;
	
	\node[scale=0.9] (x1) at (2.53,2.55){$\alpha^{\text{-}1}$} ;
	\node (x2) at (3.5,2.5){$1$} ;
	\node (x3) at (4.5,2.5){$1$} ;
	\node (x4) at (5.5,2.5){$1$} ;
	\node (x5) at (6.5,2.5){$1$} ;
	\node (x6) at (7.5,2.5){$1$} ;	
	
	\node[darkblue,scale=1.3] (y) at (11.5,2.5){$\gamma_t$} ;		

%	\node[red!40,scale=1.5] (y) at (3.5,5.5){$\mathcal{A}_t$} ;			
	
	\node[scale=0.9] (x1) at (3.53,3.55){$\alpha^{\text{-}1}$} ;
	\node (x2) at (4.5,3.5){$1$} ;
	\node (x3) at (5.5,3.5){$1$} ;
	\node (x4) at (6.5,3.5){$1$} ;
	\node (x5) at (9.5,3.5){$1$} ;
	\node[scale=0.9] (x6) at (10.53,3.5){$\beta^{\text{-}1}$} ;	
	
	\node[scale=0.9] (x1) at (4.53,4.55){$\alpha^{\text{-}1}$} ;
	\node (x2) at (5.5,4.5){$1$} ;
	\node (x3) at (6.5,4.5){$1$} ;	
	\node (x4) at (8.5,4.5){$1$} ;
	\node (x5) at (9.5,4.5){$1$} ;	
	\node (x6) at (10.5,4.5){$1$} ;
	\node[scale=0.9] (x7) at (11.53,4.5){$\beta^{\text{-}1}$} ;

	\node[scale=0.9] (x1) at (5.53,5.55){$\alpha^{\text{-}1}$} ;
	\node (x3) at (8.5,5.5){$1$} ;	
	\node (x4) at (9.5,5.5){$1$} ;
	\node (x5) at (10.5,5.5){$1$} ;	
	\node (x6) at (11.5,5.5){$1$} ;
	\node[scale=0.9] (x7) at (12.53,5.5){$\beta^{\text{-}1}$} ;	
	
	\node (x1) at (7.5,6.5){$1$} ;	
	\node (x2) at (8.5,6.5){$1$} ;
	\node (x3) at (9.5,6.5){$1$} ;	
	\node (x4) at (10.5,6.5){$1$} ;
	\node (x5) at (11.5,6.5){$1$} ;	
	\node (x6) at (12.5,6.5){$1$} ;
	\node[scale=0.9] (x7) at (13.53,6.5){$\beta^{\text{-}1}$} ;		
	
	\node[scale=0.9] (x1) at (7.53,7.55){$\alpha^{\text{-}1}$} ;		
	\node (x2) at (8.5,7.5){$1$} ;	
	\node (x3) at (9.5,7.5){$1$} ;
	\node (x4) at (10.5,7.5){$1$} ;	
	\node (x5) at (11.5,7.5){$1$} ;
	\node (x6) at (12.5,7.5){$1$} ;	
	\node (x7) at (13.5,7.5){$1$} ;
	\node[scale=0.9] (x8) at (14.53,7.5){$\beta^{\text{-}1}$} ;

\end{tikzpicture}
\end{center}
\caption{ \label{fig:environment}Visualization of \eqref{def:EnvironmentStripe} for $\eta_0=\mathbf{0}$. The numbers in the cells show $p$ in $(\omega^{x}_{p})_{x \in \Z^2}$, and $p\equiv 0 $ else. For $\eta_t=(1,0,1,1,0,1,0)$, the set $\{ x \in \Z^2 \colon G_{p}(\gamma_0,y) \leq t \} \cap \mathcal{S}$  is drawn in red, and the interface $\gamma_t$ in purple.}
\end{figure}
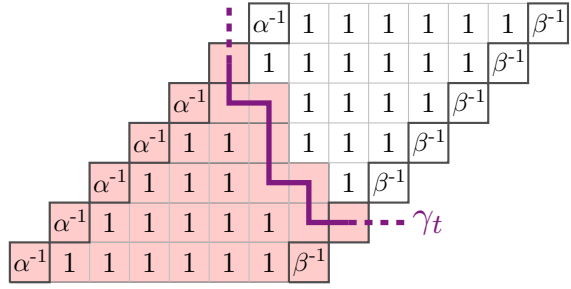
We conclude this paragraph by noting that within a box fully contained in the strip, the environment is homogeneous, i.e., we have $p \equiv 1$. The following bounds on the last passage times in homogeneous environments are Theorem 4.1 in \cite{BGZ:TemporalCorrelation}, and Theorem 2 in \cite{LR:BetaEnsembles}.
\begin{lemma} \label{lem:ShapeTheorem} Consider the environment with $p(x)=1$ for all $x \in \Z^2$, and denote by $G_1(x,y)$ the corresponding last passage time. Let $\varepsilon>0$. Then there exist $c,\bar{c} >0$, depending only on $\varepsilon$, such that for all $x=(x_1,x_2) \succeq (1,1)$ with $\varepsilon<x_1/x_2< \varepsilon^{-1}$,  and all $\theta>0$
\begin{equation}\label{eq:StatementShapeTheorem}
\P \left( |G_{1}(0,x) -   (\sqrt{x_1}+\sqrt{x_2})^2| \geq  \theta |x|^{1/3}_1   \right) \leq \bar{c} \exp\left(-c \min\big(\theta^{\frac{3}{2}},\theta  |x|^{1/3}_1 \big)\right) \, .
\end{equation} Furthermore, for general $x=(x_1,x_2)\in \N^2$, we have that
\begin{equation}\label{eq:StatementShapeTheorem2}
\P \left( |G_{1}(0,x) -   (\sqrt{x_1}+\sqrt{x_2})^2| \geq  \theta |x_1|^{1/2}|x_2|^{-1/6}   \right) \leq \bar{c} \exp\left(-c \theta \right) 
\end{equation} for all $\theta \in (0, \frac{1}{2} \sqrt{x_1}+\sqrt{x_2})^2|x_1|^{-1/2}|x_2|^{1/6})$.
In particular, there exists some absolute constant $C>0$ such that
\begin{equation}\label{eq:ExpectationShapeTheorem}
| \E[G_{1}(0,x)] -   (\sqrt{x_1}+\sqrt{x_2})^2 | \leq C |x_1|^{1/2}|x_2|^{-1/6} \, .
\end{equation}
\end{lemma}

\subsection{Geodesics and Busemann functions}\label{sec:GeodesicsBusemann}

For two sites $y \succeq x$, let $\pi^{\ast}_{x,y}$ denote a lattice path which maximizes the last passage time between $x$ and $y$. We call $\pi^{\ast}_{x,y}$ a \textbf{geodesic} between the sites $x$ and $y$. It will be convenient to work with geodesics on certain subsets of sites. Fix some $N \in \N$ and define for each $n \geq N$ the \textbf{anti-diagonal} $\mathbb{L}_n$ by
\begin{equation}\label{def:Antidiagonal}
\mathbb{L}_n := \{ x\in \mathcal{S} \colon |x|_1=n \}
\end{equation} for $\mathcal{S}=\mathcal{S}(p)$ with respect to $p$ from \eqref{def:EnvironmentStripe}. We denote by $R_n^m$ the rectangle
\begin{equation}\label{def:Rectangle}
R_n^m := \mathcal{S}  \cap \bigcup_{n \leq i < n+m} \mathbb{L}_i 
\end{equation} spanned between $\mathbb{L}_n$ and $\mathbb{L}_{n+m-1}$. 
We say that an up-right path $(z^n)_{n \in \N}$ of sites in $\Z^2$ with $z^1=x$ is a \textbf{semi-infinite geodesic} starting at $x\in \Z^{2}$, and write $\pi^{\ast}_{x}$, if we have that $\pi^{\ast}_{x}$ restricted between any two sites $z^i,z^j$ for some $i,j \in \N$ with $i<j$ is equal to $\pi^{\ast}_{z^i,z^j}$. The next lemma guarantees the existence and coalescence of a semi-infinite geodesic in $(\omega^p_x)_{x\in \Z^2}$ with $p$ from \eqref{def:EnvironmentStripe} for any starting vertex $x\in \mathcal{S}(p)$. While a similar result for homogeneous environments requires some work, it follows for the strip from a compactness argument. 
We give in the following an elementary proof of this fact.
\begin{lemma}\label{lem:ExistenceSemiInfinite} For every vertex $x\in \mathcal{S}(p)$, there exists almost surely a unique semi-infinite geodesic $\pi^{\ast}_{x}$ in $(\omega^p_x)_{x\in \Z^2}$ for $p$ from \eqref{def:EnvironmentStripe}. Moreover, for any two sites $x,y$, the semi-infinite geodesics $\pi^{\ast}_{x}$ and $\pi^{\ast}_{y}$ will almost surely coincide from some finite point $\mathbf{c}(x,y) \in \Z^2$ onwards.
\end{lemma}
\begin{proof} Without loss of generality, let $x=(1,1)$. For all $m\in \N$, set $v_m:=(m,m) \in \Z^2$ and consider the family of geodesics $(\pi^{\ast}_{x,v_m})_{m\in \N}$ where
\begin{equation*}
\pi^{\ast}_{x,v_m} = (x=v_m^1, v_m^2,v_m^3, \dots, v_m^{2m-1}=v_m) \, .
\end{equation*}  We claim that for every $j \in \N$, there exists some finite $M(j)$ such that
\begin{equation*}
v_m^{i} = v_{m^{\prime}}^{i}  
\end{equation*} holds for all $i \leq j$, and all $m,m^{\prime} \geq M(j)$. To see this,  let $A_n$ be the event 
\begin{equation*}
A_n := \left\{ \omega^p_{(\frac{n}{2}+\frac{N}{2},\frac{n}{2})} \geq N^3 \text{ and }  \omega^p_y \leq 1 \text{ for all } y\in R^N_n \setminus \Big\{ \Big(\frac{n}{2}+\frac{N}{2},\frac{n}{2}\Big)\Big\}\right\} \, .
\end{equation*}
When $A_n$ holds, $\pi^{\ast}_{x,v_m}$ and $\pi^{\ast}_{x,v_{m^{\prime}}}$  must agree on the first at least $n$ sites for all $m,m^{\prime} \geq \frac{n}{2}+N/2$ as any geodesic from $(1,1)$ to a site $z$ with $|z|_1 > n+N$ must pass through $(\frac{n}{2}+\frac{N}{2},\frac{n}{2})$. This follows from the fact that on $A_n$ any up-right path in $R_n^N$ which does not intersect $(\frac{n}{2}+\frac{N}{2},\frac{n}{2})$ has a passage time of most $2N$, while any up-right path in $R_n^N$ which intersects $(\frac{n}{2}+\frac{N}{2},\frac{n}{2})$ has a passage time of at least $N^3$; see Figure \ref{fig:Boxes}. Further, note that $A_n$ holds with positive probability,  and that $A_n$ and $A_{n+m}$ are independent for all $m> N$. We obtain the claim by applying the Borel-Cantelli lemma. Thus, we let the $j^{\textup{th}}$ vertex in $\pi^{\ast}_{x}$ be $v_j= v_{M(j)}$ to conclude the existence of the semi-infinite geodesic $\pi^{\ast}_{x}$. Similarly, the Borel-Cantelli lemma yields that $|\mathbf{c}(x,y) |_1$ is almost surely finite since
\begin{equation*}
|\mathbf{c}(x,y) |_1 \leq \inf\{ n \geq \max( |x|_1, |y|_1 ) \colon A_n \text{ occurs}\} \, .
\end{equation*} 
 Uniqueness of $\pi^{\ast}_x$ follows from the uniqueness of the paths $\pi^{\ast}_{x,v_m}$ for all $m\in \N$.
\end{proof}

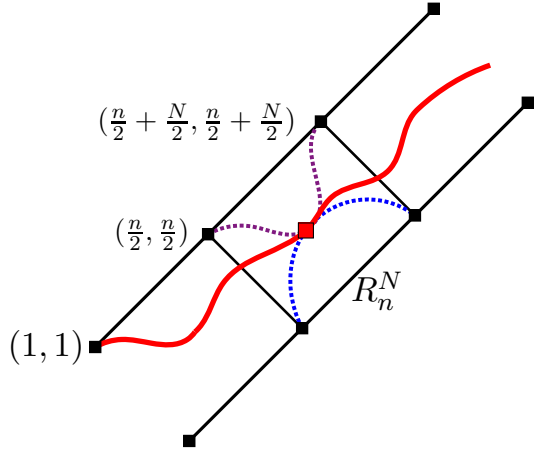
\begin{figure}
\begin{center}
\begin{tikzpicture}[scale=.5]

%\draw[densely dotted] (0,0) -- (2.5,-2.5);
\draw[line width =1 pt] (0+3,0+3) -- (2.5+3,-2.5+3);
\draw[line width =1 pt] (0+6,0+6) -- (2.5+6,-2.5+6);
%\draw[densely dotted] (0+9,0+9) -- (2.5+9,-2.5+9);

   \draw[line width=1.2pt] (0,0) -- (9,9);
   \draw[line width=1.2pt] (2.5,-2.5) -- (11.5,6.5);
   
%\draw[darkblue, line width =2 pt,densely dotted] (0,0) to[curve through={(0.4,0.3) ..(1.6,1) .. (2.4,1.35)..(2.95,2.65) }] (3,3);

\draw[darkblue, line width =1.5 pt,densely dotted] (3,3) to[curve through={(4.2,3.2)..(5.6,3.1) .. (5.8,5.7)}] (6,6);

%\draw[darkblue, line width =2 pt,densely dotted] (6,6) to[curve through={(6.5,6.3) ..(7.8,7) .. (8.4,7.35)..(8.95,8.65) }] (9,9);

%\draw[blue, line width =2 pt,densely dotted] (2.5,-2.5) to[curve through={(1+2.5,1.6-2.5) .. (1.35+2.5,2.4-2.5)..(2.65+2.5,2.95-2.5) }] (2.5+3,-2.5+3);

\draw[blue, line width =1.5 pt,densely dotted] (2.5+3,-2.5+3) to[curve through={(5.2,1.5)..(5.6,3.1)}] (2.5+6,-2.5+6);

%\draw[blue, line width =2 pt,densely dotted] (2.5+6,-2.5+6) to[curve through={(6.5+2.5,6.7-2.5) .. (8.4+2.5,9.15-2.5) }] (11.5,6.5);

%\visible<2>{

\draw[red, line width =2 pt] (0,0) to[curve through={(0.7,0.2) ..(1.5,0.1)..(2.45,0.2) }] (2.65,0.4);
%\draw[red, line width =2 pt] (2.5,-2.5) to[curve through={(2.45,0.2) }] (2.65,0.4);
 
%\draw[red, line width =2 pt] (9,9) to[curve through={(9+0.2,9-0.7) ..(9-0.1,9-1.5)..(9-0.2,9-2.45) }] (9-0.4,9-2.65);
%\draw[red, line width =2 pt] (9+2.5,9-2.5) to[curve through={(9+1.2,9-2)..(9-0.2,9-2.45)}] (9-0.4,9-2.65);

\draw[red, line width =2 pt] (2.65,0.4) to[curve through={(2.75,0.5).. (5.6-2.2,3.1-1.6+0.2)..(5.6,3.1)..(5.68,3.18) ..(5.6+0.7,3.1+0.9)..(5.6+2.2,3.1+1.6) .. (9-0.5,9-2.75)..(9-0.4,9-2.65)}] (9+1.5,9-1.5);

%}
   
	\filldraw [fill=black] (2.5-0.15,-2.5-0.15) rectangle (2.5+0.15,-2.5+0.15);   
	
		\node[scale=1.2] (x1) at (-1.3,0){$(1,1)$} ;
	%	\node[scale=1.2] (x2) at (2.5+0.7,-2.5){$a_2$} ;
	%	\node[scale=1.2] (x3) at (9+0.7,9){$a_4$} ;
	%	\node[scale=1.2] (x4) at (9+2.5+0.7,9-2.5){$a_3$} ;
	
	\node[scale=1] (x1) at (2.7,6){$(\frac{n}{2}+\frac{N}{2},\frac{n}{2}+\frac{N}{2})$} ;	
    \node[scale=1] (x1) at (1.5,3){$(\frac{n}{2},\frac{n}{2})$} ;	
		
		\node[scale=1.2] (x1) at (7.5,1.5){$R_{n}^{N}$} ;

 	\filldraw [fill=black] (2.5+3-0.15,-2.5+3-0.15) rectangle (2.5+3+0.15,-2.5+3+0.15);   
 	\filldraw [fill=black] (2.5+6-0.15,-2.5+6-0.15) rectangle (2.5+6+0.15,-2.5+6+0.15);     
 	\filldraw [fill=black] (2.5+9-0.15,-2.5+9-0.15) rectangle (2.5+9+0.15,-2.5+9+0.15);     
 	
 	\filldraw [fill=red] (5.6-0.2,3.1-0.2) rectangle (5.6+0.2,3.1+0.2);     
 	
	\filldraw [fill=black] (-0.15,-0.15) rectangle (0.15,0.15);   
 	\filldraw [fill=black] (3-0.15,3-0.15) rectangle (3+0.15,3+0.15);   
 	\filldraw [fill=black] (6-0.15,6-0.15) rectangle (6+0.15,6+0.15);     
 	\filldraw [fill=black] (9-0.15,9-0.15) rectangle (9+0.15,9+0.15);     	
 	
	\end{tikzpicture}	
\end{center}	
	\caption{ \label{fig:Boxes}Visualization of the rectangle $R_{n}^N$ used in the proof of Lemma~\ref{lem:ExistenceSemiInfinite}. The red dot in the middle of the box $R_{n}^N$ corresponds to the site $(\frac{n}{2}+\frac{N}{2},\frac{n}{2})$, and any semi-infinite geodesic starting from $(1,1)$ must pass through this site when $A_n$ holds. }
 \end{figure}

Note that Lemma \ref{lem:ExistenceSemiInfinite} can be extended for any uniformly bounded function $\tilde{p}$ with the same support as $p$ from \eqref{def:EnvironmentStripe}. Further, using Lemma \ref{lem:ExistenceSemiInfinite}, we can for all $x,y \in \mathcal{S}(p)$ define 
\begin{equation}\label{def:Busemann}
B(x,y):= G_p(\pi^{\ast}_{x,\mathbf{c}(x,y)}) - G_p(\pi^{\ast}_{y,\mathbf{c}(x,y)}) 
\end{equation} to be their  \textbf{Busemann function}. In words, $B(x,y)$ is the signed difference between the last passage times for any point on the semi-infinite geodesics of $x$ and $y$ beyond $\mathbf{c}(x,y)$.

\subsection{Second class particles and competition interfaces}\label{sec:CompetitionInterfaces}

Recall the notion of second class particles  from Section \ref{sec:TASEPInteracting}. We will now provide a connection to competition interfaces in the CGM representation of the TASEP with open boundaries. A connection of this type was first observed by Ferrari and Pimentel in \cite{FP:CompetitionInterface} for last passage percolation on $\Z^2$. \\ %We adapt the arguments from \cite{FMP:CompetitionInterface,FP:CompetitionInterface} to the TASEP with open boundaries. \\

We start by defining the competition interface, following the construction in \cite{FMP:CompetitionInterface,FP:CompetitionInterface}. Recall the initial growth interface $\gamma_0$ for the strip of width $N$. Suppose that we have some $m\in [N-1]$ with
\begin{equation}\label{def:Gamma+-}
\gamma^{m}_0 - \gamma^{m-1}_0 = \eone \ \text{ and } \ \gamma^{m}_0 - \gamma^{m+1}_0 =  \etwo \, .
\end{equation}
Fix such an $m$, and set $\gamma_+ := \{ \gamma^{i}_0 \colon  i \in \{0,\dots,m-1\}$ and $\gamma_- := \gamma_0 \setminus \gamma_+$. We define 
\begin{align*}
\Gamma^+ &:= \left\{ x \in \Z^2 \colon G(\gamma_+ , x) > G(\gamma_- , x)\text{ and } x \succeq y+(1,1) \text{ for some } y \in \gamma_0\right\} \\
\Gamma^- &:= \left\{ x \in \Z^2 \colon G(\gamma_- , x) > G(\gamma_+ , x)\text{ and } x \succeq y+(1,1) \text{ for some } y \in \gamma_0\right\} 
\end{align*} with respect to the environment with $p$ from \eqref{def:EnvironmentStripe}. Note that $\Gamma^+ $ and $\Gamma^-$ partition $\mathcal{S}(p)$ above $\gamma$ almost surely. The \textbf{competition interface} $(\phi_n)_{n \in \N}$ is recursively defined by $\phi_1:=\gamma_0^m$ and
\begin{equation}\label{def:CompetionInterface}
\phi_{n+1}:= \begin{cases} \phi_n +(0,1) &\text{ if }  \phi_n + (1,1) \in \Gamma^+ \\
\phi_n +(1,0) &\text{ if }  \phi_n + (1,1) \in \Gamma^- \\
\phi_n &\text{ otherwise} \, .
\end{cases}
\end{equation} Intuitively, we color the regions $\Gamma^+$ and $\Gamma^-$, and $(\phi_n)_{n \in \N}$  travels along the transition of the colors. Alternatively,  the competition interface is characterized by the recursion
\begin{equation}\label{eq:CompetionInterfaceArgMin}
\phi_{n+1} = \textup{arg min} \left\{ G(\gamma_0, \phi_{n}+(0,1)), G(\gamma_0, \phi_{n}+(1,0)) \right\} \, ;
\end{equation} see equation (2.2) in \cite{FMP:CompetitionInterface}. A visualization of the competition interface is given in Figure~\ref{fig:Competition}. For $A \subseteq \mathcal{S}(p)$, we say that $A$ is \textbf{monochromatic} if either $A \subseteq \Gamma^+$  or $A \subseteq \Gamma^-$ holds.\\

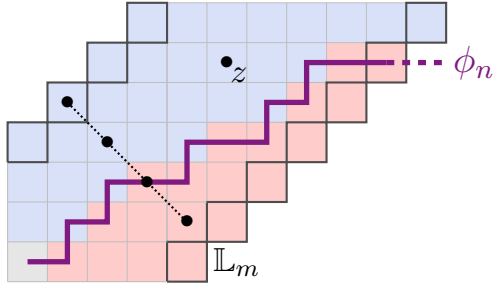
\begin{figure}  
\begin{center}
\begin{tikzpicture}[scale=0.53]

\fill [fill=gray!20] (4,1) rectangle ++(1,1);

\fill [fill=red!20] (5,1) rectangle ++(4,1);
\fill [fill=red!20] (6,2) rectangle ++(4,1);
\fill [fill=red!20] (7,3) rectangle ++(4,1);
\fill [fill=red!20] (9,4) rectangle ++(3,1);
\fill [fill=red!20] (11,5) rectangle ++(2,1);
\fill [fill=red!20] (12,6) rectangle ++(2,1);

\fill [fill=deepblue!20] (4,2) rectangle ++(2,1);
\fill [fill=deepblue!20] (4,3) rectangle ++(3,1);
\fill [fill=deepblue!20] (4,4) rectangle ++(5,1);
\fill [fill=deepblue!20] (5,5) rectangle ++(6,1);
\fill [fill=deepblue!20] (6,6) rectangle ++(6,1);
\fill [fill=deepblue!20] (7,7) rectangle ++(8,1);

\foreach \x in{4,...,8}{
	\draw[gray!50,thin](\x,\x) to (7+\x,\x); 
	\draw[gray!50,thin](\x,1) to (\x,\x);
	\draw[gray!50,thin](7+\x,\x) to (7+\x,8);  }
	
\foreach \x in{1,...,3}{
		\draw[gray!50,thin](7+\x,\x) to (7+\x,8);
}

	\foreach \x in{1,...,7}{

	\draw[black!70,thick](\x+7,\x) -- (\x+1+7,\x) -- (\x+1+7,\x+1) -- (\x+7,\x+1)-- (\x+7,\x);
	}
	
		\foreach \x in{4,...,7}{
	
	\draw[black!70,thick](\x,\x) -- (\x+1,\x) -- (\x+1,\x+1) -- (\x,\x+1)-- (\x,\x);

	}
	
	\draw[gray!50,thin](4,1) to (8,1); 	
	\draw[gray!50,thin](4,2) to (8,2); 		
	\draw[gray!50,thin](4,3) to (9,3); 	
	
\draw[darkblue,line width =2pt] 	 (4.5,1.5) -- ++(1,0) -- ++(0,1) -- ++(1,0) -- ++(0,1) -- ++ (2,0) -- ++ (0,1) -- ++ (2,0) -- ++ (0,1) -- ++(1,0) -- ++(0,1) -- ++ (2,0);
	
\draw[darkblue,line width =2pt,dashed] 	 (13.5,6.5) -- ++(1.5,0);	

	\node[darkblue,scale=1.3] (y) at (15.7,6.5){$\phi_n$} ;	

	\node (z) at (9.5,6.5) {$\bullet$};
	\node[scale=1.15] (z) at (9.8,6.2) {$z$};
	
%	\node[purple,scale=1.5] (y) at (11.5,2.5){$\gamma_t$} ;		

%\draw[black,ultra thick](6,5) -- ++(1,0) -- ++(0,1) -- ++(-1,0) -- ++(0,-1) ;
%\draw[black,ultra thick](7,4) -- ++(1,0) -- ++(0,1) -- ++(-1,0) -- ++(0,-1) ;
%\draw[black,ultra thick](8,3) -- ++(1,0) -- ++(0,1) -- ++(-1,0) -- ++(0,-1) ;
%\draw[black,ultra thick](9,2) -- ++(1,0) -- ++(0,1) -- ++(-1,0) -- ++(0,-1) ;

	\node (z1) at (5.5,5.5) {$\bullet$};
	\node (z2) at (6.5,4.5) {$\bullet$};
	\node (z3) at (7.5,3.5) {$\bullet$};
	\node (z4) at (8.5,2.5) {$\bullet$};
	
	\draw[densely dotted,thick] (5.5,5.5) to (8.5,2.5);
    \node[scale=1.15] (z) at (9.75,1.5) {$\mathbb{L}_m$};
\end{tikzpicture}
\end{center}
\caption{\label{fig:Competition}Visualization of the competition interface $(\phi_n)_{n \geq 1}$ drawn in purple. Suppose that for some $k,m \geq 1$, all geodesics between $\mathbb{L}_m$ and $\mathbb{L}_{m+k}$ pass through $z$. Since $z$ is blue, all sites in $\mathbb{L}_{m+k}$ must be blue as well.}
\end{figure}
Next, we consider a disagreement process $(\xi_t)_{t \geq 0}$ with a single second class particle. Let $\tau_{\textup{ex}}$ be the exit time of this second class particle. The following lemma intertwines the competition interface $(\phi_n)_{n \in \N}=(\phi^1_n,\phi^2_n)_{n \in \N}$ and the trajectory $(X_t)_{t \in [0,\tau_{\textup{ex}})}$ of the second class particle. Since it follows similar to Proposition 3 in \cite{FP:CompetitionInterface} and Proposition 2.2 in \cite{FMP:CompetitionInterface}, we will only highlight the required adjustments for the strip, rather than giving full details.  

\begin{lemma}\label{lem:SecondClassCompetition} For the disagreement process $(\xi_t)_{t \geq 0}$ on the segment of size $N$ with a single second class particle starting at $x$, let $\gamma_0$ be the initial growth interface on the strip of width $N+1$ corresponding to the configuration where we replace the second class particle in $\xi_0$ by a $(0,1)$ pair, and take  $m=x$ in \eqref{def:Gamma+-}. There exists a coupling such that 
\begin{equation}\label{eq:SecondClassCompetition}
X_t -X_0= (\phi^1_n - \phi^1_1) - (\phi^2_n - \phi^2_1)
\end{equation} holds almost surely for all $t \in [G(\gamma_0,\phi_n),G(\gamma_0,\phi_{n+1}))$ with $t <\tau_{\textup{ex}}$, and $n\in \N$. Moreover
\begin{equation}\label{eq:ExitImplication}
\mathbb{L}_k \subseteq  \Gamma^+ \text{ or } \ \mathbb{L}_k \subseteq  \Gamma^- \quad \Rightarrow \quad \tau_{\textup{ex}} \leq G(\gamma_0, \mathbb{L}_k)
\end{equation} holds for all $k > N$, i.e., if the anti-diagonal $\mathbb{L}_k$ from \eqref{def:Antidiagonal} is monochromatic, then the second class particle has left the segment by time $G(\gamma_0,\mathbb{L}_k)$.
\end{lemma}

\begin{proof}[Sketch of the proof] 
We adopt the arguments of Proposition 3 and Lemma 6 in \cite{FP:CompetitionInterface}. In a first step, we replace the second class particle by a $(0,1)$-pair, also called $\ast$-pair, by extending the dynamic to the segment $[N+1]$ using the same boundary conditions, i.e.\ a configuration $\xi \in \{0,1,2\}^N$ with a single second class particle at site $x$ gets mapped into the configuration $\eta \in \{0,1\}^{N+1}$ with
\begin{equation}
\eta(y) := \begin{cases} \xi(y) & \text{ if } y<x \\
0 & \text{ if } y=x \\
1 & \text{ if } y=x+1 \\
\xi(y-1) & \text{ if } y>x+1 \\
\end{cases}
\end{equation} for all $y\in [N+1]$. In a second step, we then describe the motion of the $\ast$-pair using the competition interface. \\

More precisely, for the first step, let $x$ be the starting location of the second class particle in $(\xi_t)_{t \geq 0}$, and let $(\eta^{01}_t)_{t \geq 0}$ be the same TASEP with open boundaries started from $\xi_0$, but where we replace the second class particle at $x$ by a $(0,1)$-pair, increasing the size of the segment by $1$. We use the same clocks for corresponding particles.   The particle in the $\ast$-pair agrees with the clock of the second class particle. Let the $\ast$-pair evolve in the following way: Whenever the particle in the $\ast$-pair moves, we let the $\ast$-pair move one step to the right. Whenever a particle moves to the empty site in the $\ast$-pair, we let the $\ast$-pair move one step to the left. Note that after every such move until time $\tau_{\text{ex}}$, we will obtain a $(0,1)$-pair.  This gives a bijection between $(X_t)_{t \geq 0}$ and the $\ast$-pair in $(\eta^{01}_t)_{t \geq 0}$ until time $\tau_{\text{ex}}$. \\

For the second step, we use \eqref{eq:CompetionInterfaceArgMin} and an induction argument as in Proposition~3 of~\cite{FP:CompetitionInterface} to see that the $\ast$-pair follows the competition interface as described in \eqref{eq:SecondClassCompetition}, e.g., every time we make a step to the right in  the competition interface, \eqref{eq:CompetionInterfaceArgMin} ensures that the particle in the $\ast$-pair jumps before the empty site gets replaced. A special case occurs when $(\phi_n)_{n \in \N}$ reaches the boundary of the strip. Without loss of generality, suppose that $\phi_k=(\ell,\ell)$ for some $\ell\in \N$. Then at time $\tilde{t}=G(\gamma_0,\phi_{k-1})$, the $\ast$-pair is located at the sites $1$ and $2$, while we have $X(\tilde{t})=1$. Using  \eqref{eq:CompetionInterfaceArgMin}, we see that the empty site in the $\ast$-pair gets replaced before the particle in the $\ast$-pair moves. Since this  replacement occurs at rate $\alpha$, we can extend the above coupling such that the second class particle in $(\xi_t)_{t \geq 0}$ leaves the segment at time $G(\gamma_0,\phi_{k})$. In particular, when we see a monochromatic anti-diagonal, the competition interface must have reached the boundary of the strip, giving \eqref{eq:ExitImplication}. \end{proof}

In Section \ref{sec:MixingTimesMaxCurrent}, our goal will be to establish criteria which ensure that with positive probability, we see a monochromatic anti-diagonal at a distance of order at most $N^{3/2}$ away from $x$ for any possible starting configuration. This will then allow us to prove Theorem \ref{thm:MaxCurrent}.

\subsection{Stationary corner growth model on the strip}\label{sec:StationaryLPP}

In this section, we discuss a stationary corner growth model for the triple point $\alpha=\beta=\frac{1}{2}$, which will be an essential tool in Section~\ref{sec:MixingTimesTriplePoint}. It is the analogue of the stationary CGM on $\Z^2$; see for example Section~3 in \cite{S:LectureNotes} for an introduction. Recall from Section \ref{sec:TASEPInteracting} that in the triple point, the stationary measure of the TASEP with open boundaries has a product structure, i.e., it is the Uniform distribution on $\{ 0,1\}^N$. We see that \textbf{Burke's property} holds, which means that in equilibrium, the particle trajectories have marginal distributions according to rate $\frac{1}{2}$ Poisson clocks; see also Lemma~\ref{lem:StationarityLPPprocess}. Formally, we define the stationary CGM on the strip as the corner growth model on $\Z^2$ in environment $(\omega_x^{p_{\textup{stat}}})_{x \in \Z^2}$ given according to
\begin{equation}\label{def:StationaryEnvironment}
p_{\textup{stat}}(x):=\begin{cases} 2 & \text{ if } x_1=x_2 \geq  1,  \text{ or } x_1-N= x_2 \geq 1, \text{ or }  x_2=0 \text{ and } x_1 \in [N] \\
1& \text{ if }  x_1 - N <  x_2 \geq 0   \text{ and } x_1> x_2>0 \\
0 & \text{ otherwise} \, , 
\end{cases}
\end{equation} and denote by $G_{\textup{stat}}$ the corresponding last passage times. A similar version of a stationary CGM was also studied on the half-quadrant; see \cite{BBCS:Halfspace,BFO:HalfspaceStationary,PS:CurrentFluctuations}. The following lemma characterizes the distribution of the increments of the last passage times. Further, it justifies that the CGM with respect to $p_{\textup{stat}}$ is called stationary, in the sense of \cite{BCS:CubeRoot}, and that Burke's property holds. Since the proof follows from similar arguments as Lemma 4.1 and Lemma~4.2 in \cite{BCS:CubeRoot}, we  only give a sketch of the proof; see also Lemma 2.1 and Lemma 2.2 in \cite{BFO:HalfspaceStationary}.

\begin{lemma}\label{lem:StationarityLPPprocess} Fix some $i,j$ with $(i+1,j+1) \in \mathcal{S}=\mathcal{S}(p_{\textup{stat}})$. If $(i,j+1) \in \mathcal{S}$ holds, then
\begin{equation}\label{def:HoriIncrement}
I_{i+1,j+1} := G_{\textup{stat}}(0, (i+1,j+1)) - G_{\textup{stat}}(0, (i,j+1)) 
\end{equation}  is Exponential-$\frac{1}{2}$-distributed. If $(i+1,j) \in \mathcal{S}$ holds, we see that
\begin{equation}\label{def:VertiIncrement}
J_{i+1,j+1} := G_{\textup{stat}}(0, (i+1,j+1)) - G_{\textup{stat}}(0, (i+1,j)) 
\end{equation} is Exponential-$\frac{1}{2}$-distributed. If $(i,j),(i,j+1),(i+1,j) \in \mathcal{S}$ holds, we have that
\begin{equation}\label{def:MinIncrement}
Y_{i,j} := \min(I_{i+1,j} , J_{i,j+1} ) 
\end{equation} is Exponential-$1$-distributed. Moreover, $I_{i+1,j+1},J_{i+1,j+1}$ and $Y_{i,j}$ are independent. Furthermore, along every down-right path $\gamma=\{\gamma^{i}\}_{i \in \{0,\dots,N\}}$ in $\mathcal{S}$,  where
\begin{equation}\label{def:DownRight}
\gamma^{i}- \gamma^{i-1} \in \{ \eone, -\etwo \}
\end{equation} for all $i\in [N]$, the increments $|G_{\textup{stat}}(0,\gamma^{i}) - G_{\textup{stat}}(0, \gamma^{i-1})|$ are i.i.d.\ Exponential-$\frac{1}{2}$-distributed, and independent of $Y_{i,j}$ for all $(i,j) \in \mathcal{S}$ with $(i+1,j+1) \preceq y$ for some $y\in \gamma$.
\end{lemma}

\begin{proof}[Sketch of the proof] The first claim in the lemma is shown using induction. Observe that the increments of the last passage times satisfy
\begin{equation}\label{eq:InductionStep}
I_{i,j} = (I_{i,{j-1}} - J_{i-1,j} )^+ + \omega^{p_{\textup{stat}}}_{(i,j)} \, ,
\end{equation} where $z^+:= \max(z,0)$ for $z \in \R$, and similarly for $J_{i,j}$. Note that for $I_{2,0}$ and $J_{1,1}$, the statements \eqref{def:HoriIncrement} and \eqref{def:VertiIncrement} hold by construction. Using this fact as induction basis, we obtain the independence and the correct marginal distributions by calculating the moment generating function of $(I_{i+1,j+1},J_{i+1,j+1},Y_{i,j})$, and using  \eqref{eq:InductionStep};  see Lemma 4.1 in \cite{BCS:CubeRoot}.
A similar induction argument on the down-right paths contained in $\mathcal{S}$ yields independence of the increments along every down-right path; see also Lemma 4.2 in \cite{BCS:CubeRoot}. 
\end{proof}
\begin{remark} For last passage percolation on $\Z^2$, the stationary CGM is closely related to Busemann functions; see \cite{CP:BusemannSecondClass,S:LectureNotes,S:ExistenceGeodesics}. Note that Busemann functions are defined in \eqref{def:Busemann} for the strip, but we will not further discuss a connection to Lemma \ref{lem:StationarityLPPprocess} at this point.
\end{remark}

\section{Mixing times in the maximal current phase} \label{sec:MixingTimesMaxCurrent}

In this section, our goal is to show Theorem \ref{thm:MaxCurrent}. We have the following strategy: Using Lemma  \ref{lem:ShapeTheorem} and Lemma \ref{lem:SecondClassCompetition}, we bound the exit time of a second class particle by finding a monochromatic anti-diagonal. We see in Section \ref{sec:ExitTimesByCoalescence} that  a sufficient condition for a monochromatic anti-diagonal is that all geodesics between two anti-diagonals  coalesce. 
In Section \ref{sec:Coalescence}, we use similar ideas as Dey et al.\ in \cite{DJP:CriticalStripe} for Poissonian LPP on a strip to show that the geodesics coalesce with positive probability if the anti-diagonals have a distance of order $N^{3/2}$. To do so, we need to control the fluctuations of the last passage times on the strip, which we accomplish in Section \ref{sec:VarianceUpperBound}. 
While the required estimates follow from known results when $\alpha,\beta \geq 1$, the analysis for $\alpha \in [\frac{1}{2},1)$ and $\beta \in [\frac{1}{2},1)$ is more delicate and uses a path decomposition together with recent results on last passage times in the half-quadrant \cite{BBCS:Halfspace}, and the stationary CGM on $\Z^2$ \cite{B:ModerateDeviationsStationary}.

\subsection{Exit times via the coalescence of geodesics}\label{sec:ExitTimesByCoalescence}

We start with a simple observation on the ordering of geodesics in last passage percolation; see also Lemma 11.2 in \cite{BSV:SlowBond}. We define a natural ordering $\succeq_{\g}$ of up-right paths of finite or infinite length on the strip as follows: For two up-right paths $\pi_1, \pi_2$, we have that
\begin{equation*}
\pi_1 \succeq_{\g} \pi_2  \quad \Leftrightarrow \quad  y_1 \geq y_2  \text{ for all } (x,y_1) \in \pi_1, (x,y_2)\in \pi_2  \, ,
\end{equation*} i.e., for every pair of sites on the paths with the same first component, the site in $\pi_1$ lies above the site in $\pi_2$. The next lemma gives a sufficient condition on the endpoints of finite geodesics 
such that they are ordered with respect to $\succeq_{\g}$.
\begin{lemma} \label{lem:OrderingGeodesics} Fix some environment $(\omega^p_x)_{x\in \Z^2}$. For $i \in [4]$,  let $a_i=(a^1_i,a^2_i) \in \Z^2$. If
\begin{equation}\label{eq:OrderingAssumptions}
a^1_1 \leq a^1_2  \leq a^1_4 \leq a^1_3 \ \   \text{ and } \  \  a^2_2 \leq  a^2_1 \leq a^2_3 \leq  a^2_4
\end{equation} holds, and the geodesics $\pi^{\ast}_{a_1,a_4}$ and $\pi^{\ast}_{a_2,a_3} $ are fully contained in $\mathcal{S}(p)$, then almost surely 
\begin{equation}\label{eq:OrderingStatement}
\pi^{\ast}_{a_1,a_4}  \succeq_{\g}  \pi^{\ast}_{a_2,a_3}  \, .
\end{equation}
\end{lemma}
\begin{proof} The statement \eqref{eq:OrderingStatement} is immediate from a picture, see for example Figure 16 in \cite{BSV:SlowBond} for a visualization, using that between any such pair of points, there is almost surely a unique maximizing path.
\end{proof}

As a direct consequence, we get a sufficient condition for monochromatic anti-diagonals.

\begin{lemma}\label{lem:Coalescence} Consider some initial growth interface for the strip of width $N$, which satisfies \eqref{def:Gamma+-}, and color all sites according to $\Gamma^+$ and $\Gamma^{-}$. Let $n,k \geq N$ and recall the rectangle $R_n^k$ from \eqref{def:Rectangle}. For $i \in [4]$,  let $a_i=(a^1_i,a^2_i) \in \Z^2$ be the corners of $R_n^k$ in counter-clockwise order such that $a_1,a_2 \in \mathbb{L}_n$ and $a_3,a_4 \in \mathbb{L}_{n+k-1}$. Then  almost surely 
\begin{equation}\label{eq:ImplicationCoalesceMono}
\pi^{\ast}_{a_1,a_4} \cap \pi^{\ast}_{a_2,a_3} \neq \emptyset \quad \Rightarrow \quad \mathbb{L}_{n+k} \text{ is monochromatic} \, . 
\end{equation}
\end{lemma}

\begin{proof} Using the ordering in Lemma \ref{lem:OrderingGeodesics}, we see that any geodesic between $x\in \mathbb{L}_n$ and $y \in \mathbb{L}_{n+k}$ must pass through some common point $z$ whenever the event on the left-hand side of \eqref{eq:ImplicationCoalesceMono} occurs. By construction, any geodesic connecting the initial growth interface $\gamma_0$ to $\mathbb{L}_{n+k}$ must intersect with $\mathbb{L}_n$, and hence also pass through $z$; see Figure \ref{fig:Competition}. Without loss of generality, assume that $z \in \Gamma^{+}$. This implies that all sites in $\mathbb{L}_{n+k}$ must be contained in $\Gamma^+$ as there is almost surely a unique maximal path connecting $\gamma_0$ and $z$ which every geodesic between $\gamma_0$ and some site in $\mathbb{L}_{n+k}$ follows. \end{proof}

 Note that Lemma \ref{lem:ExistenceSemiInfinite} guarantees the coalescence of the geodesics between any two points in $\mathbb{L}_{n}$ and $\mathbb{L}_{n+k}$ on the strip of width $N$ when $k$ is sufficiently large. The next proposition bounds the probability of coalescence for $k=k(N)$. Its proof is deferred to Section~\ref{sec:Coalescence}.
 
 \begin{proposition}\label{pro:CoalescenceTimes}  Recall the environment for $p$ from \eqref{def:EnvironmentStripe} for some $\alpha,\beta \geq \frac{1}{2}$, and some initial growth interface. Consider the  rectangle $R^k_N$ and denote by $a_i$ for $i\in [4]$ its corners in counter-clockwise order such that $a_1,a_2 \in \mathbb{L}_N$ and $a_3,a_4 \in \mathbb{L}_{N+k-1}$. There exist  $c,C>0$ and $N_0 \in \N$ such that for all  $N \geq N_0$, and for all $m \in \N$, we have for $k=mcN^{3/2}$
 \begin{equation}\label{eq:CoalescenceTimesQuanti}
 \P(\pi^{\ast}_{a_1,a_4} \cap \pi^{\ast}_{a_2,a_3} \neq \emptyset) \geq 1 - \exp(- C m) \, .
 \end{equation}
 \end{proposition}
 
 Assuming that Proposition \ref{pro:CoalescenceTimes} holds, we have all the tools to show Theorem \ref{thm:MaxCurrent}.
 
 \begin{proof}[Proof of Theorem \ref{thm:MaxCurrent}] By Lemma \ref{lem:MixingBoundByExiting}, we can bound the mixing time by the exit  time of the second class particles in the disagreement process between $\mathbf{1}$ and $\mathbf{0}$. Instead of bounding the exit time of all second class particles at once, we instead consider $N$ disagreement processes $(\xi^{i}_t)_{t \geq 0}$ with respect to the step initial configurations $\eta_0^{i}$ and $\eta_0^{i-1}$ in $\Omega_N$ for $i \in[N]$, recalling $[N]=\{1,\dots,N\}$, where we define
 \begin{equation*}
 \eta_0^{i}(j) := \mathds{1}_{j > i}
 \end{equation*} for all $j\in [N]$. Note that $\eta_0^{0}=\mathbf{1}$ and $\eta_0^{N}=\mathbf{0}$. Fix some $\varepsilon=\varepsilon(N)>0$. Using the triangle inequality for the total-variation distance in \eqref{def:TVDistance}, it suffices to show that the exit time $\tau_{\textup{ex}}^{i}$ of the second class particles in $(\xi^{i}_t)_{t \geq 0}$ satisfies
 \begin{equation}\label{eq:ManyDisagreementProcesses}
 \P\left( \tau_{\textup{ex}}^{i} \leq s \right) \geq 1- \frac{\varepsilon}{N}
 \end{equation}
 for all $i \in [N]$ in order to conclude that $t^N_\mix(\varepsilon) \leq s$ for the mixing time of the TASEP with open boundaries.  Note that each of the disagreement processes $(\xi^{i}_t)_{t \geq 0}$ contains initially exactly one second class particle. Let $\gamma_{0,i}$ denote the initial growth interface with respect to $(\xi^{i}_t)_{t \geq 0}$.  By combining Lemma \ref{lem:SecondClassCompetition} and Lemma \ref{lem:Coalescence}, together with Proposition \ref{pro:CoalescenceTimes} for $m=2C^{-1}\log(N)$, and thus $k=2 C^{-1} c \log(N) N^{3/2}$, we see that
 \begin{equation}\label{eq:ReductionMixingTheorem}
\max_{i \in [N]} \P\left( \tau_{\textup{ex}}^{i} \leq 5k \right) \geq 1 -  \frac{1}{N^2} - N \max_{i \in [N]}\P\left(  G(\gamma_{0,i}, \mathbb{L}_{N+k}) \geq 5k \right)  
 \end{equation} for all $N$ large enough, where $c$ and $C$ are taken from Proposition \ref{pro:CoalescenceTimes}.  We claim that the right-hand side in \eqref{eq:ReductionMixingTheorem} must be larger than $1-2N^{-2}$ for sufficiently large $N$. To see this, we dominate the environment with respect to $p$ with a homogeneous environment on $\Z^2$ for $p \equiv 2$. Let $G_1(\cdot,\cdot)$ and $G_2(\cdot,\cdot)$ denote the last passage times with respect to $p \equiv 1$ and $p \equiv 2$, and note that $G_2(x,y)$ has the same law as $2G_1(x,y)$ for all $y \succeq x$. By Lemma \ref{lem:ShapeTheorem}, we get 
 \begin{align*}
 \P( G_2(x,y) &> 2 \E[ G_2(x,y)]\text{ for all } x\in \gamma_{0,i} \text{ and } y\in \mathbb{L}_{N+k} ) \\
 &=  \P( G_1(x,y) > 2 \E[ G_1(x,y)]\text{ for all } x\in \gamma_{0,i} \text{ and } y\in \mathbb{L}_{N+k} ) \leq N^2 \exp(-N^{1/3}) \, .
\end{align*}  for all $N$ sufficiently large.
 Using the well-known fact that $t^N_{\textup{mix}}(\varepsilon)$ is for fixed $N$ monotone decreasing in $\varepsilon>0$, we conclude the upper bound on $t^N_{\textup{mix}}(\varepsilon)$ by  \eqref{eq:ManyDisagreementProcesses} with $\varepsilon=2N^{-1}$.
 \end{proof}

\subsection{Coalescence times of geodesics}\label{sec:Coalescence}

In order to show Proposition \ref{pro:CoalescenceTimes}, we use a similar strategy as Dey et al.\ in Section 6 of \cite{DJP:CriticalStripe} for Poissonian LPP on a strip in the plane. We first give a lower bound on the probability that the last passage time between $\mathbb{L}_N$ and $\mathbb{L}_{N+k}$ is large when $k$ is of order $N^{3/2}$; see Lemma \ref{lem:LowerTails}. We then show that with high probability, there are no two disjoint  paths with large passage times. We then argue that with sufficiently large probability, we can construct a large path between any two points of $\mathbb{L}_N$ and $\mathbb{L}_{N+k}$, letting us conclude that all the corresponding geodesics must in fact intersect with positive probability. We iterate this argument along multiples of $k$ to conclude. \\

We start with a bound on the probability to see a large last passage time in $R_N^{nN^{3/2}}$. In the following, we consider an environment $(\omega_x^p)_{x \in \Z^2}$ for $p$ from \eqref{def:EnvironmentStripe} for some $\alpha,\beta \geq \frac{1}{2}$ and $N\in \N$, starting from  the growth interface corresponding  to the all empty configuration $\mathbf{0}$. 

\begin{lemma}\label{lem:LowerTails} There exist  $\tilde{c} >0$ and $\tilde{N}_0 \in \N$ such that for all $N\geq \tilde{N}_0$, and all $n \in \N$ and $\theta \geq 1$, the last passage time on the strip satisfies
\begin{equation}\label{eq:LowerBoundLargeGeodesic}
\P\left( G(\mathbb{L}_N, \mathbb{L}_{N+n N^{3/2}} ) > 2 n N^{\frac{3}{2}} +  \theta n \sqrt{N} \right)  \geq  \exp(-\tilde{c} n \theta^{3/2} ) \, .
\end{equation}
\end{lemma}

\begin{proof} For $n \in \N$, let $x_{\textup{mid}}^n:= (\lfloor n/2+ N/4\rfloor,\lfloor n/2- N/4\rfloor)$ denote the midpoint of the anti-diagonal $\mathbb{L}_n$. By Lemma B.2 in \cite{BG:TimeCorrelation}, and monotonicity in $\alpha$ and $\beta$, we have that
\begin{equation}\label{eq:MidToMid}
\P\left( G(x_{\textup{mid}}^N, x_{\textup{mid}}^{N+N^{3/2}-1} ) > 2 N^{\frac{3}{2}} + \theta \sqrt{N} \right)  \geq  \exp(-\tilde{c} \theta^{3/2} ) 
\end{equation}
holds for all $\alpha,\beta> 0$, $\theta \geq 1$, and some constant $\tilde{c}$, provided that $N$ is sufficiently large. We concatenate the geodesics along the midpoints of the anti-diagonals $\mathbb{L}_{N+i N^{3/2}}$ for all $i \in \{0,\dots,n\}$ to see that the left-hand side in \eqref{eq:LowerBoundLargeGeodesic} is bounded from below by
\begin{equation*}
\prod_{i=1}^{n}  \P\left( G\Big(x_{\textup{mid}}^{N+(i-1)N^{3/2}}, x_{\textup{mid}}^{N+iN^{3/2}-1} \Big) > 2 N^{\frac{3}{2}} + \theta \sqrt{N} \right) \, .
\end{equation*}
Together with \eqref{eq:MidToMid}, this yields the claim.
\end{proof}

For the next statement, which states bounds on the fluctuations of the passage times along the geodesics between $\mathbb{L}_N$ and $\mathbb{L}_{N+\theta^{-1}N^{3/2}}$ for some   $\theta \geq 1$, we let for all $n,k \geq N$
\begin{align*}
G_{\min}(\mathbb{L}_n, \mathbb{L}_{n+k}) := \min_{x \in \mathbb{L}_n, y \in \mathbb{L}_{n+k}} G(x,y)
\end{align*} denote the minimal last passage time connecting two points on the anti-diagonals.
\begin{proposition}\label{pro:VarianceUpperBoundRephased} There exist constants  $c_1,c_2,\tilde{\theta}>0$ and $N^{\prime}_0 \in \N$ such that for all $\theta \geq \tilde{\theta}$, and all $n\geq N \geq N^{\prime}_0$, we have that
\begin{align}
\label{eq:UpperBoundFluctuations}\P\big( G(\mathbb{L}_n, \mathbb{L}_{n+\theta^{-1}N^{3/2}}) - 2\theta^{-1}N^{\frac{3}{2}}\geq \theta \sqrt{N} \big) &\leq \exp(-c_1 \theta^{\frac{3}{2}}) \\
\label{eq:LowerBoundFluctuations}\P\big( G_{\min}(\mathbb{L}_n, \mathbb{L}_{n+\theta^{-1}N^{3/2}})  - 2\theta^{-1}N^{\frac{3}{2}} \leq  -  2\theta \sqrt{N} \big) &\leq  \exp(- c_2 \theta) \, .
\end{align}
\end{proposition}
%In words, Proposition \ref{pro:VarianceUpperBoundRephased} says that with high probability, all geodesics connecting two sites on $\mathbb{L}_n$ and $\mathbb{L}_{n+N^{3/2}})$ have weight $4N^{3/2}$ up to corrections of order $\sqrt{N}$. The second statement \eqref{eq:LowerBoundFluctuations} is immediate from Proposition A.4 in \cite{BG:TimeCorrelation} and a monotonicity argument, see also Proposition 12.2.\ in \cite{BSV:SlowBond} for the original version for Poissonian environments. The first statement \eqref{eq:UpperBoundFluctuations} on the maximal last passage time between the two anti-diagonals requires more work, and is deferred to Section \ref{sec:VarianceUpperBound}. 
Since the proof of Proposition \ref{pro:VarianceUpperBoundRephased} requires a bit of setup, it is deferred to Section \ref{sec:VarianceUpperBound}. 

\begin{remark}\label{rem:VarianceCurrent}
Using Lemma \ref{lem:CurrentVsGeodesic}, the upper bound on $\mathcal{J}^N_{t_N}$ in Proposition \ref{pro:VarianceCurrent} follows from Theorem 4.2 in \cite{BGZ:TemporalCorrelation}. The lower bound follows from \eqref{eq:UpperBoundFluctuations}, using that 
\begin{equation}
\P\big( G(1,N^{\frac{3}{2}}) \geq  s \big) \leq n \P\big( G(1,n^{-1}N^{\frac{3}{2}}) \geq  n^{-1} s \big)
\end{equation}  for all $n \in \N$ and $s \geq 0$.
\end{remark}

In order to show Proposition \ref{pro:CoalescenceTimes}, we require the following bound on the intersection of geodesics. For $y \succeq x$ and $\tilde{y} \succeq \tilde{x}$, and fixed times $t,\tilde{t}>0$, we consider the events 
\begin{align*}
A_{x,y} &:= \{ G_{x,y} > t \} \quad \text{ and } \quad 
B_{\tilde{x},\tilde{y}} := \{ G_{\tilde{x},\tilde{y}}>\tilde{t}\} \, .
\end{align*}
We define the
\textbf{disjoint occurrence} of $A_{x,y}$ and $B_{\tilde{x},\tilde{y}}$ as the event 
\begin{align*}
A_{x,y} \circ B_{\tilde{x},\tilde{y}} := \big\{  \exists  \pi_{x,y} \in \Pi_{x,y} , \,   \pi_{\tilde{x},\tilde{y}} \in \Pi_{\tilde{x},\tilde{y}} \colon G(x,y) \geq t  \, \wedge \,  G(\tilde{x},\tilde{y}) \geq \tilde{t} \, \wedge \,  \pi_{x,y} \cap \pi_{\tilde{x},\tilde{y}}  = \emptyset \big\} \, .
\end{align*} In words, the event $A_{x,y}\circ B_{\tilde{x},\tilde{y}}$ occurs when we find a lattice path $\pi_{\tilde{x},\tilde{y}}$ from $\tilde{x}$ to $\tilde{y}$ of weight at least $\tilde{t}$, as well as a lattice path $\pi_{x,y}$ from $x$ to $y$ of weight at least $t$ such that $\pi_{x,y}$ and $\pi_{\tilde{x},\tilde{y}}$ do not intersect. In next lemma follows from the BKR inequality; see Theorem 7 in~\cite{AGH:BKRgeneral}.

\begin{lemma}\label{lem:BKinequality} For all $y \succeq x$ and $\tilde{y} \succeq \tilde{x}$, and all $t,\tilde{t}>0$ in the definition of $A_{x,y}$ and $B_{\tilde{x},\tilde{y}}$
\begin{equation}
\P(A_{x,y} \circ B_{\tilde{x},\tilde{y}}) \leq \P(A_{x,y})\P(B_{\tilde{x},\tilde{y}}) \, .
\end{equation}
\end{lemma}
%\begin{proof} Note that for all $\pi \in \Pi_{x,y}$, we have that
%\begin{equation*}
%\P\Big( \max_{\tilde{\pi} \in \Pi_{\tilde{x},\tilde{y}} \colon \tilde{\pi} \cap \pi = \emptyset}G(\tilde{\pi}) \geq \tilde{t} \Big) \leq \P(B_{\tilde{x},\tilde{y}}) 
%\end{equation*} 
%Since for fixed $\pi$, the passage times $G(\pi)$ and $G(\tilde{\pi})$ for some $\tilde{\pi} \cap \pi=\emptyset$ are defined with respect to disjoint parts of the environment, we obtain 
%\begin{align*}
%\P(A_{x,y} \circ B_{\tilde{x},\tilde{y}})  &=  \sum_{\pi \in \Pi_{x,y}}\P\Big(\max_{\tilde{\pi} \in \Pi_{\tilde{x},\tilde{y}} \colon \tilde{\pi} \cap \pi = \emptyset}G(\tilde{\pi}) \geq \tilde{t} \, \text{ and } \, G(\pi)>t \,  \Big| \, \pi^{\ast}_{x,y} =\pi  \Big)\P( \pi^{\ast}_{x,y} =\pi ) \\
%&\leq  \sum_{\pi \in \Pi_{x,y}}\P(B_{\tilde{x},\tilde{y}})\P( A_{x,y}\,  | \, \pi^{\ast}_{x,y} =\pi  )\P( \pi^{\ast}_{x,y} =\pi ) \\
%&= \P(A_{x,y})\P(B_{\tilde{x},\tilde{y}}) 
%\end{align*} which allows us to conclude.
%\end{proof}

Assuming that Proposition \ref{pro:VarianceUpperBoundRephased} holds, we have all the tools to show Proposition \ref{pro:CoalescenceTimes}. %, which in return gives Theorem \ref{thm:MaxCurrent}.

\begin{proof}[Proof of Proposition \ref{pro:CoalescenceTimes}] We claim that it suffices to show \eqref{eq:CoalescenceTimesQuanti} for $m=1$. To see this, note that for $m\geq 2$, we can consider the coalescence event on the left-hand side of \eqref{eq:CoalescenceTimesQuanti} for the anti-diagonals $\mathbb{L}_{N+c(i-1)N^{3/2}}$ and  $\mathbb{L}_{N+ciN^{3/2}-1}$, where $i \in [m]$. Note that these events are independent, and whenever one of these events occurs for a pair of geodesics, we have that by Lemma \ref{lem:OrderingGeodesics}, the geodesics between $\mathbb{L}_{N}$ and $\mathbb{L}_{N+cmN^{3/2}-1}$ must coalesce as well. Fix $m=1$, and let $c > 0$ be some constant, which we will determine later on.
In order to show \eqref{eq:CoalescenceTimesQuanti}, we consider the events
\begin{align*}
A :=& \left\{  G(a_1, a_4)\geq  2 c N^{\frac{3}{2}} +  9c^2\sqrt{N}  \right\} \\
B :=& \left\{  G(a_2, a_3)\geq  2 c N^{\frac{3}{2}} +   c^2  \sqrt{N}  \right\} 
\end{align*} on the rectangle $R_N^{c N^{3/2}}$, where we recall that $a_i$ for $i\in [4]$ denotes the corners of the rectangle $R_N^{c N^{3/2}}$. Further, recall that we denote by $A \circ B$  the event of  disjoint occurrence of $A$ and $B$, and that Lemma \ref{lem:BKinequality}  yields \begin{equation}\label{eq:BKinequality}
\P(A\circ B) \leq \P(A )\P(B)
\end{equation} for all choices of $c>0$. Next, for all $n\in \N$ with $n\geq N$,  we define the events
\begin{align*}
D_{n} :=& \{ G(\mathbb{L}_n, \mathbb{L}_{n+c^{-2}N^{3/2}}) - 2c^{-2}N^{\frac{3}{2}}\leq  c^2 \sqrt{N} \}   \\
E_{n} :=&  \{ G_{\min}(\mathbb{L}_n, \mathbb{L}_{n+c^{-2}N^{3/2}}) - 2c^{-2}N^{\frac{3}{2}}\geq  -2c^2 \sqrt{N}  \}   \, ; 
\end{align*} see also Proposition \ref{pro:VarianceUpperBoundRephased}. Since $A$ and $E_{n}$ are increasing events with respect to the environment $(\omega^p_x)_{x\in \Z^2}$, we apply the FKG inequality to get
\begin{equation*}%\label{eq:FKGforLPP}
\P\big( D_{n} \cap E_{n} \big|  A \big) \geq \P( E_{n} ) - (1-\P( D_{n} ))\P(A)^{-1}
\end{equation*}  for all $c>0$. By Lemma \ref{lem:LowerTails}, we have that for some constant $c_0>0$
\begin{align*}
\P\Big( \min_{x\in \mathbb{L}_{N^{3/2}}} G(a_1,x) -2 N^{3/2} > -c^{-1}_0 \sqrt{N}\Big) \geq c_0 \\ \P\Big( \min_{y\in \mathbb{L}_{(c-1)N^{3/2}}} G(y,a_4) -2 N^{3/2} > - c^{-1}_0 \sqrt{N}\Big) \geq c_0 \, .
\end{align*} Together with Proposition \ref{pro:VarianceUpperBoundRephased} to estimate $G({\mathbb{L}_{N^{3/2}+1},\mathbb{L}_{(c-1)N^{3/2}-1}})$, we see that
\begin{equation}
\P(A) \geq \exp\big(-50\tilde{c} c^{\frac{5}{2}} )
\end{equation} for all $c>0$ sufficiently large.
Using Lemma \ref{lem:LowerTails} and Proposition \ref{pro:VarianceUpperBoundRephased} to bound the probability of the events $D_n$ and $E_n$, respectively, this gives us that
\begin{equation}\label{eq:ConnectionEvents}
\P\big( D_{n} \cap E_{n}  \big|  A \big)  \geq 1-\exp(-c_2 c^2) - \exp\big(-c_1 c^3 + 50\tilde{c} c^{\frac{5}{2}}\big)  \geq \frac{7}{8} 
\end{equation}
for $c$ sufficiently large, and all $n\geq N \geq \max(\tilde{N}_0,N^{\prime}_0)$. When the events in \eqref{eq:ConnectionEvents} hold for a suitable choice of $n$, we want to argue that 
% $n=N$ as well as for $n=N+(C-C^{-2})N^{3/2}-1$, 
we can modify the geodesic between $a_1$ and $a_4$ to obtain a large path between $a_2$ and $a_3$. Formally, let $M= N+ \lfloor (c-c^{-2})N^{3/2} \rfloor -1$. We claim that
\begin{equation}\label{eq:ModifedAgain}
 A \cap D_{N} \cap E_{N} \cap D_{M} \cap E_{M} \subseteq  B
\end{equation} holds for all $c>0$ sufficiently large. 
\begin{figure}
\begin{center}
\begin{tikzpicture}[scale=.45]

\draw[densely dotted, line width=1pt] (0,0) -- (2.5,-2.5);

\draw[densely dotted, line width=1pt] (0+9,0+9) -- (2.5+9,-2.5+9);

\draw[densely dotted, line width=1pt] (1.525,1.525) -- (1.525+2.5,1.525-2.5);
\draw[densely dotted, line width=1pt] (9-1.525,9-1.525) -- (9-1.525+2.5,9-1.525-2.5);

   \draw[line width=1.2pt] (0,0) -- (9,9);
   \draw[line width=1.2pt] (2.5,-2.5) -- (11.5,6.5);

\draw[red, line width =2 pt] (0,0) to[curve through={(0.7,0.2) ..(1.5,0.1)..(2.45,0.2) }] (2.65,0.4);

 \draw[red, line width =2 pt] (2.65,0.4) to[curve through={(2.75,0.5).. (5.6-2.2,3.1-1.6+0.2)..(5.6,3.1)..(5.68,3.18) ..(5.6+0.7,3.1+0.9)..(5.6+2.2,3.1+1.6) .. (9-0.5,9-2.75)}] (9-0.4,9-2.65);
\draw[red, line width =2 pt] (9,9) to[curve through={(9+0.2,9-0.7) ..(9-0.1,9-1.5)..(9-0.2,9-2.45) }] (9-0.4,9-2.65);

\draw[darkblue, line width =2 pt] (9+2.5,9-2.5) to[curve through={(9+2.3,9-2.6)..(9+1.2,9-2)..(9-0.2,9-2.45)}] (9-0.4,9-2.65);
\draw[darkblue, line width =2 pt] (2.5,-2.5) to[curve through={(2.5+0.4,-2.5+0.7) ..(2.65-0.2,0.2)}] (2.65,0.4);

		\node[scale=1] (x1) at (-0.7,0){$a_1$} ;
		\node[scale=1] (x2) at (2.5+0.7,-2.5){$a_2$} ;
		\node[scale=1] (x3) at (9+0.7,9){$a_4$} ;
		\node[scale=1] (x4) at (9+2.5+0.7,9-2.5){$a_3$} ;
		
		\draw[line width =1 pt] (-0.3,0.3) -- (-0.5,0.5)   --  (-0.5+1.525,0.5+1.525) --  (-0.3+1.525,0.3+1.525);

		\draw[line width =1 pt] (9+0.3+2.5,9-0.3-2.5) -- (9+0.5+2.5,9-0.5-2.5)   --  (0.5+9+2.5-1.525,-0.5+9-1.525-2.5) --  (0.3+9+2.5-1.525,-0.3+9-1.525-2.5);
		
		\node[scale=1] (x3) at (-0.8+0.75-1.1,0.8+0.75){$\c^{-2}N^{3/2}$} ;
		
			\node[scale=1] (x3) at (12.7,5){$\c^{-2}N^{3/2}$} ;

		%	\node[scale=1] (x3) at (11.5,1){$|a_4-a_1|_1=CN^{3/2}$} ;

%\node[scale=1] (x5) at (9-1.525-0.15-3.5,9-1.525-0.15){$a_4-(C^{-1} N^{3/2}, C^{-1} N^{3/2})$};
%\node[scale=1] (x5) at (1.525-0.15-3.5,1.525-0.15){$a_1+(C^{-1} N^{3/2}, C^{-1} N^{3/2})$};

	\filldraw [fill=black] (2.5-0.15,-2.5-0.15) rectangle (2.5+0.15,-2.5+0.15);   

 	\filldraw [fill=black] (2.5+9-0.15,-2.5+9-0.15) rectangle (2.5+9+0.15,-2.5+9+0.15);     
 	
 %	\visible<2>{
	\filldraw [fill=black] (2.65-0.15,0.4-0.15) rectangle (2.65+0.15,0.4+0.15);    	
	\filldraw [fill=black] (9-0.4-0.15,9-2.65-0.15) rectangle (9-0.4+0.15,9-2.65+0.15);    	 	
	
	\node[scale=1] (x1) at (2.65+0.85,0.4){$a_5$} ;	
    \node[scale=1] (x1) at ((9-0.4+0.85,9-2.65){$a_6$} ;
  % }
 	
	\filldraw [fill=black] (-0.15,-0.15) rectangle (0.15,0.15);   

 	\filldraw [fill=black] (9-0.15,9-0.15) rectangle (9+0.15,9+0.15);     	
 	
	\end{tikzpicture}
	\end{center}
\caption{\label{fig:CoalescenceModified}Visualization of the modified path constructed in the proof of Proposition \ref{pro:CoalescenceTimes}.}		
 \end{figure}
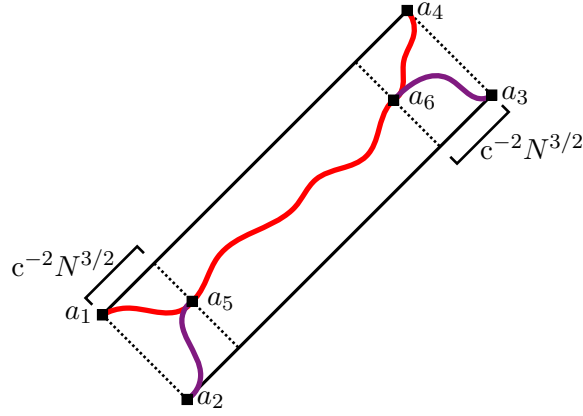
To see this, let $a_5 \in \Z^2$ be the site where $\pi^{\ast}_{a_1,a_4}$ intersects $\mathbb{L}_{N+c^{-2}N^{3/2}}$. Similarly, let 
$a_6 \in \Z^2$ be the site where $\pi^{\ast}_{a_1,a_4}$ intersects $\mathbb{L}_{M}$. We construct a path from $a_2$ to $a_3$ by first following the geodesic from $a_2$ to $a_5$, then the geodesic from $a_5$ to $a_6$ (which agrees with $\pi^{\ast}_{a_1,a_4}$), and then the geodesic from $a_6$ to $a_3$; see Figure \ref{fig:CoalescenceModified} for a visualization. Note that when the events on the left-hand side in \eqref{eq:ModifedAgain} hold, the constructed path ensures that $B$ holds. Thus, using \eqref{eq:ConnectionEvents}, we obtain that
\begin{align}\label{eq:LowerBoundBC}
\P\left( B| A\right) &\geq \P\big( D_{N} \cap E_{N} \cap D_{M} \cap E_{M}  \big|  A \big)  \nonumber \\
&\geq 1 - \P\big( (D_{N} \cap E_{N})^{\textup{c}} \big|  A \big) - \P\big( (D_{M} \cap E_{M})^{\textup{c}}   \big|  A \big)  \geq  \frac{3}{4}  
\end{align} for sufficiently large $c>0$. Note that by \eqref{eq:BKinequality}, and the fact that
\begin{align*}
\P(B) &\leq \P\big( \exists i \in [c^2] \colon G(\mathbb{L}_{N+(i-1) c^{-1} N^{3/2}},\mathbb{L}_{N+i c^{-1} N^{3/2}-1}) \geq 2 N^{\frac{3}{2}} +   c \sqrt{N}  \big) \\
&\leq  c^2 \exp(-c_2c^{\frac{3}{2}})
\end{align*}
by Proposition \ref{pro:VarianceUpperBoundRephased}, we see that
\begin{equation*}
\P\left( A\circ B | A \right) \leq  \P\left( B\right) \leq \frac{1}{4}
\end{equation*} for all $c$ large enough. Together with \eqref{eq:LowerBoundBC} and choosing $c$ large enough, we note that with  probability at least $\frac{1}{2}\P(A)$, the two geodesics $\pi^{\ast}_{a_1,a_4}$ and $\pi^{\ast}_{a_2,a_3}$ must intersect. 
\end{proof}

\subsection{Proof of the fluctuations in Proposition \ref{pro:VarianceUpperBoundRephased}} \label{sec:VarianceUpperBound}
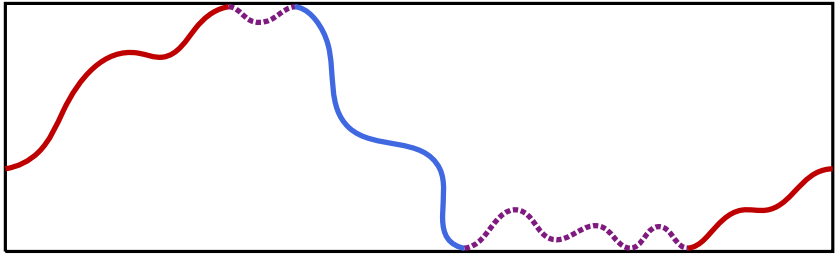
\begin{figure}
\centering
\begin{tikzpicture}[scale=1.1]

%   \draw[line width=1.2pt] (0,0) -- (10,0) -- (10,4) -- (0,4) -- (0,0);
%   
%\draw[nicos-red, line width =1.7 pt] (0,2) to[curve through={(0.4,2.2) .. (0.7,2.7) .. (1.6,3.4) .. (1.9,3.35) .. (2.3,3.7)..(2.65,3.95) }] (2.7,3.96);
%
%\draw[darkblue, line width =1.7 pt] (2.7,3.96) to[curve through={(2.75,3.95) .. (2.95,3.8).. (3.2,3.8) .. (3.45,3.95) }] (3.5,3.96);
%
%
%\draw[nicos-green, line width =1.7 pt] (3.5,3.96) to[curve through={(3.55,3.95).. (3.75,3.8) .. (3.9,3.5) .. (4.15,2.5) .. (4.4,2.4).. (4.8,1.2) .. (5.25,1)..  (5.5,0.05) }] (5.55,0.04);
%
%\draw[darkblue, line width =1.7 pt] (5.55,0.04) to[curve through={(5.6,0.05) .. (5.8,0.2).. (6.2,0.5) .. (6.6,0.15) ..  (7.2,0.3).. (7.5,0.05) .. (7.55,0.04) .. (7.6,0.05) ..(7.9,0.3).. (8.2,0.05) }] (8.25,0.04);
%
%\draw[nicos-red, line width =1.7 pt] (8.25,0.04) to[curve through={(8.3,0.05) .. (8.5,0.2).. (8.9,0.5) ..  (9.2,0.5)..(9.8,0.95)}] (10,1);

   \draw[line width=1.2pt] (0,0) -- (10,0) -- (10,3) -- (0,3) -- (0,0);
   
\draw[nicos-red, line width =2 pt] (0,1) to[curve through={(0.4,1.2) .. (0.7,1.7) .. (1.6,2.4) .. (1.9,2.35) .. (2.3,2.7)..(2.65,2.95) }] (2.7,2.96);

\draw[darkblue, line width =2 pt,densely dotted] (2.7,2.96) to[curve through={(2.75,2.95) .. (2.95,2.8).. (3.2,2.8) .. (3.45,2.95) }] (3.5,2.96);

\draw[deepblue, line width =2 pt] (3.5,2.96) to[curve through={(3.55,2.95).. (3.75,2.8) .. (3.9,2.5) .. (4.15,1.5) ..  .. (5.25,1)..  (5.5,0.05) }] (5.55,0.04);

\draw[darkblue, line width =2 pt, densely dotted] (5.55,0.04) to[curve through={(5.6,0.05) .. (5.8,0.2).. (6.2,0.5) .. (6.6,0.15) ..  (7.2,0.3).. (7.5,0.05) .. (7.55,0.04) .. (7.6,0.05) ..(7.9,0.3).. (8.2,0.05) }] (8.25,0.04);

\draw[nicos-red, line width =2 pt] (8.25,0.04) to[curve through={(8.3,0.05) .. (8.5,0.2).. (8.9,0.5) ..  (9.2,0.5)..(9.8,0.95)}] (10,1);

	\end{tikzpicture}	
\caption{\label{fig:PathDecomposition}Decomposition of the paths used in the proof of Proposition \ref{pro:VarianceUpperBoundRephased} on a strip, rotated by $45^{\circ}$. Paths in $\Pi_1$ are drawn in red, paths in $\Pi_2$ are drawn in blue, and paths in $\Pi_3$ are drawn dotted in purple.}
 \end{figure}

We start with the statement \eqref{eq:LowerBoundFluctuations}. Recall \eqref{eq:ExpectationShapeTheorem} in Lemma \ref{lem:ShapeTheorem} with some constant $C=C(1/2)$. A calculation shows that for any two points $x,y$ with $x\in \mathbb{L}_N$ and $y \in \mathbb{L}_{N+\theta^{-1}N^{3/2}}$, the last passage time $G_1$ from Lemma~\ref{lem:ShapeTheorem} satisfies
\begin{align}\label{eq:ExpectationEstimateLower}
\E[G_1(x,y)] &\geq \Big(\sqrt{\theta^{-1}N^{3/2}  + N/2 } + \sqrt{\theta^{-1}N^{3/2}  - N/2 }\Big)^2 - 2C\theta^{-1/3} N^{\frac{1}{2}}  \\
&\geq 2\theta^{-1} N^{3/2} - (\theta + 2C\theta^{-1/3}) N^{1/2} -  \theta^{2}N^{-1/2} \nonumber 
\end{align} for all $N$ large enough, using the fact that $\sqrt{1-\delta}\leq 1-\delta/2 + \delta^2$ for all $\delta\in [0,1]$. The statement \eqref{eq:LowerBoundFluctuations} is now a consequence of Theorem 4.2 (iii) in \cite{BGZ:TemporalCorrelation}, which covers the special case $\theta=1$, and similarly applies for fixed $\theta>1$ when $N$ is sufficiently large. \\

It remains to show the statement \eqref{eq:UpperBoundFluctuations}. Recall the rectangle $R := R_N^{\theta^{-1}N^{3/2}} \subseteq \Z^2 $ defined in \eqref{def:Rectangle}, and denote by $\partial R$ its boundary. We have the following strategy.
Each lattice path $\pi_{x,y}$ between some $x \in \mathbb{L}_N$ and $y \in \mathbb{L}_{N+\theta^{-1}N^{3/2}}$ is decomposed into subpaths which belong to exactly one of the following three sets $\Pi_1,\Pi_2$ and $\Pi_3$. Paths in $\Pi_1$ have both endpoints in $\partial R$, but do not touch $\partial R$ otherwise. Paths in $\Pi_2$ connect a site in $R_+ := R \cap \{ (n,n) \colon n\in \N\}$ with a site in $R_- := R \cap \{ (n+N,n) \colon n\in \N\}$, and do not touch $\partial R$ otherwise. Paths in $\Pi_3$ connect two sites which are both in $R_+$, respectively both in $R_-$, but do not touch $R_-$, respectively $R_+$, at any point. \\

More precisely, for a lattice path $\pi_{x,y}=(x=z^0,z^1,\dots,z^\ell=y)$ with $\ell=\lfloor\theta^{-1}N^{3/2}\rfloor$, let $\mathcal{I}_{+}$ and $\mathcal{I}_{-}$ denote the sets of indices where $z^i$ is contained in the upper boundary $R_+$, respectively in the lower boundary $R_-$. We apply the following recursive decomposition: \\

If $\mathcal{I}_{+}=\mathcal{I}_{-}=\emptyset$, we let $\pi \in\Pi_1$. Else, let $i_{\min}$ and $i_{\max}$ be the smallest and largest indices in $\mathcal{I}_{+} \cup \mathcal{I}_{-}$, respectively. If $i_{\min} \neq 1$, add the path $(z^1,\dots,z^{i_{\min}})$ to $\Pi_1$, and proceed with the remaining sub-path. Else, assume without loss of generality that $i_{\min} \in \mathcal{I}_{+}$ holds. If $i_{\min}=i_{\max}=1$, we let $\pi \in \Pi_1$ 
and stop. Else, let $j$ be the smallest index in $\mathcal{I}_{-}$, possibly $j=\infty$ if $\mathcal{I}_{-}$ is empty, and let $k$ be the largest index in $\mathcal{I}_{+}$ with $k<j$. We will now add 
\begin{equation*}
(z^{i_{\min}}, \dots, z^{k}) \in \Pi_3 \qquad \text{ and } \qquad (z^{k}, \dots, z^j) \in \Pi_2 \, ,
\end{equation*}
and continue this decomposition in the same way for the remaining sub-path. A visualization of this path  decomposition is given in Figure \ref{fig:PathDecomposition}. 

In order to bound the passage time of paths in $\Pi_1$ and $\Pi_2$, we  rely on a general estimate for last passage times in homogeneous environments, i.e.\ when $p\equiv 1$.
Recall that $|y|_1=|y_1|+|y_2|$ for $y=(y_1,y_2) \in \Z^2$, and let $\mathbb{B}(R)$ be the set of pairs of points $(x,y)\in R \times R$ with $x=(x_1,x_2)$,  $y=(y_1,y_2)$, and $y \succeq x$, which have a slope bounded between $\frac{1}{2}$ and $2$, i.e., where
\begin{equation*}
\frac{y_2-x_2}{y_1-x_1} \in \left[ \frac{1}{2}, 2\right] \, .
\end{equation*} The next result on last passage times between pairs of sites  in $\mathbb{B}(R)$ is stated as Proposition~A.3 in \cite{BG:TimeCorrelation}, where it is noted that the proof follows by a similar argument as Proposition~10.5 in \cite{BSV:SlowBond} for Poissonian last passage percolation. For the sake of completeness, we include a proof in the appendix, relying on Theorem 4.2 (ii) of \cite{BGZ:TemporalCorrelation}. 
\begin{lemma}\label{lem:LastPassagePercolationGeneral}
Let $p\equiv 1$. There exist constants $C,\theta_0>0$ such that for $\theta \geq \theta_0$
\begin{equation}\label{eq:GeneralLPT}
\P\left(  \max_{(x,y) \in \mathbb{B}(R)}  \left[ G_p(x,y)  - \E[G_p(x,y)] \right]  >  \theta \sqrt{N}   \right)  \leq \exp(-C \theta^{\frac{3}{2}})
\end{equation} holds for all $N$ sufficiently large.
\end{lemma}

The following lemma states an upper bound on the passage time for any path in  $\Pi_1$. For a  pair of sites $x$ and $y$, we write $(x,y) \in \partial R$ if $x,y \in \partial R$ and $y \succeq x$ holds.
\begin{lemma}\label{lem:Category1} Take $p$ from \eqref{def:EnvironmentStripe}. Recall from \eqref{def:PathLPT} that $G(\pi_{x,y})=G_p(\pi_{x,y})$ denotes the passage time of some lattice path $\pi_{x,y}$ connecting $x=(x_1,x_2)$ and $y=(y_1,y_2)$. There exist constants $C_1,\theta_0>0$ such that for   $\theta \geq \theta_0$
\begin{equation}\label{eq:Category1Bound}
\P\left(  \max_{(x,y) \in \partial R} \max_{\pi_{x,y} \in \Pi_1}  \left[ G(\pi_{x,y})  - 2 |y-x|_{1} \right]  >  \frac{\theta}{40} \sqrt{N}   \right)  \leq \exp(-C_1 \theta^{\frac{3}{2}})
\end{equation} holds for all $N$ sufficiently large.
\end{lemma}
\begin{proof} 
Recall the upper bound on the expectation of $G(x,y)$ from \eqref{eq:ExpectationShapeTheorem}, and that all paths in $\Pi_1$ touch the boundary $\partial R$ only at their endpoints. We distinguish two cases. For all pairs of sites with a slope bounded between $\frac{1}{2}$ and $2$, we obtain \eqref{eq:Category1Bound} from Lemma \ref{lem:LastPassagePercolationGeneral}. For the remaining pairs of points $(x,y)$, note that we get from \eqref{eq:ExpectationShapeTheorem} in Lemma \ref{lem:ShapeTheorem} that
\begin{equation*}
\E[G(x,y)] - 2|y-x|_1  < - c |y-x|_1 
\end{equation*} for some absolute constant $c>0$, and all $N$ sufficiently large. We then conclude \eqref{eq:Category1Bound} from \eqref{eq:StatementShapeTheorem2} in Lemma \ref{lem:ShapeTheorem} and a union bound over the at most $N^4$ such pairs of sites in $\partial R$.
\end{proof}

Similarly, we can estimate the passage time of all paths in $\Pi_2$, which connect $R_+$ and $R_-$.

\begin{lemma}\label{lem:Category2} Recall that the rectangle $R$ has length $\theta^{-1}N^{3/2}$ for some $\theta \geq 1$. There exist  constants $C_2,\theta_0>0$ such that for $\theta>\theta_0$ and $\lambda>0$
\begin{equation}\label{eq:Category2Bound}
\P\left(  \max_{(x,y) \in \partial R} \max_{\pi_{x,y} \in \Pi_2}  \left[ G(\pi_{x,y})  - 2 |y-x|_{1} \right]  > \Big( \lambda - \frac{\theta}{4}\Big) \sqrt{N}   \right)  \leq  \exp(-C_2 \lambda^{\frac{3}{2}})
\end{equation} holds for all $N$ sufficiently large.
\end{lemma}
\begin{proof} For all paths  $\pi_{x,y}\in \Pi_2$, which connect two points $x=(x_1,x_2)$ and $y=(y_1,y_2)$ with a slope bounded  between $\frac{1}{2}$ and $2$, we have that
\begin{align*}%\label{eq:ExpectationBoundCor}
\E[G(\pi_{x,y})] &\leq \big(\sqrt{y_1-x_1 + N/2} + \sqrt{y_2-x_2 -N/2}\big)^2  \\
& \leq 2 |y-x|_1 -  \frac{N^2}{4|y-x|_1} +2C |y-x|_1^{1/3} 
\end{align*} holds for the constant $C>0$ in \eqref{eq:ExpectationShapeTheorem} of  Lemma \ref{lem:ShapeTheorem}. Similar to  Lemma \ref{lem:Category1}, since we touch $R_+ \cup R_-$ only at the endpoints, \eqref{eq:Category2Bound} follows from Lemma \ref{lem:LastPassagePercolationGeneral} for all pairs of points $\mathbb{B}(R)$ with a slope bounded between $\frac{1}{2}$ and $2$. Again, for the remaining paths, we apply \eqref{eq:StatementShapeTheorem2} in Lemma \ref{lem:ShapeTheorem} and a union bound to conclude.
\end{proof}
It remains to show that the passage time of paths in $\Pi_3$ can not be too large either.

\begin{lemma}\label{lem:Category3}
There exist constants $C_3,\theta_0>0$ such that for $\theta \geq \theta_0$, 
\begin{equation}\label{eq:Category3Bound} 
\P\left(  \max_{(x,y) \in \partial R} \max_{\pi_{x,y} \in \Pi_3}  \left[ G(\pi_{x,y})  - 2 |y-x|_{1} \right]  >  \frac{\theta}{40} \sqrt{N}   \right)  \leq \exp(-C_3 \theta^{\frac{3}{2}}) 
\end{equation} holds for all $N$ sufficiently large.
\end{lemma}

Before coming to the proof of Lemma \ref{lem:Category3}, we show how to obtain Proposition \ref{pro:VarianceUpperBoundRephased} by combining Lemma \ref{lem:Category1}, Lemma \ref{lem:Category2},  and Lemma \ref{lem:Category3}.

\begin{proof}[Proof of Proposition \ref{pro:VarianceUpperBoundRephased}] 

Fix some lattice path $\pi$ connecting a site in $\mathbb{L}_n$ to some site in $\mathbb{L}_{n+\theta^{-1}N^{3/2}}$. According to the above described decomposition into sub-paths, let $\Pi^{\pi}_1$ be the set of subpaths of $\pi$ belonging to the set $\Pi_1$. Note that $\Pi^{\pi}_1$ contains at most two elements. Further, let  $\Pi^{\pi}_2$ and $\Pi^{\pi}_3$ denote the sets of subpaths of $\pi$ belonging to $\Pi_2$ and $\Pi_3$, respectively. Let $|\pi|$ be the length of a lattice path $\pi$, i.e.\ the number of elements in $\pi$ minus~$1$. Assuming the complement of the events in Lemma \ref{lem:Category1} and Lemma \ref{lem:Category3} holds, we get
\begin{align}\label{eq:Decom1}
\sum_{\tilde{\pi} \in \Pi^{\pi}_1 }(G(\tilde{\pi}) - 2 |\tilde{\pi}|) \leq 2\sqrt{N} \frac{\theta}{40} \quad \text{ and } \quad 
\sum_{\tilde{\pi} \in \Pi^{\pi}_3 }(G(\tilde{\pi}) - 2 |\tilde{\pi}|) \leq   |\Pi^{\pi}_3| \sqrt{N} \frac{\theta}{40} \, .
\end{align}
Similarly, choosing $\lambda=\theta/8$ in Lemma \ref{lem:Category2} and assuming that the complement of the event in Lemma \ref{lem:Category2} holds, we see that 
\begin{align}\label{eq:Decom2}
\sum_{\tilde{\pi} \in \Pi^{\pi}_2 }(G(\tilde{\pi}) - 2 |\tilde{\pi}|) \leq - |\Pi^{\pi}_2|\sqrt{N} \frac{\theta}{8} \, .
\end{align}
By construction of the decomposition, we have that $|\Pi^{\pi}_3| \leq |\Pi^{\pi}_2| +1$ for any 
path $\pi$, and thus combining \eqref{eq:Decom1} and \eqref{eq:Decom2}, we conclude that with probability at least $1-\exp(-c_1\theta^{3/2})$
\begin{equation}\label{eq:RecomposePaths}
G(\pi) - 2|\pi| \leq  \sqrt{N} \frac{\theta}{20} + \Big(|\Pi^{\pi}_3|- |\Pi^{\pi}_2|\Big) \sqrt{N} \frac{\theta}{8} \leq \sqrt{N}\theta
\end{equation}
for some suitable constant $c_1>0$. Since $\pi$ was arbitrary, this finishes the proof.
\end{proof}

Using the same arguments as in the above proof of Proposition \ref{pro:VarianceUpperBoundRephased}, we obtain the following bound on the last passage time between two sites on different boundaries of the strip. We use this result later on in Section \ref{sec:LowerBounds} for a lower bound on the mixing time.

\begin{corollary}\label{cor:LowerBoundLPP}
Take $p$ from \eqref{def:EnvironmentStripe} and consider the rectangle $R= R_N^{\theta^{-1}N^{3/2}} \subseteq \Z^2 $. Fix some $x=(x_1,x_2) \in R_+$ and $y=(y_1,y_2) \in R_-$ in the upper and lower boundary of $R$, respectively. There exist constants $C_4,\theta_0>0$ such that for $\theta \geq \theta_0$ and all $N$ large enough
\begin{equation}
\P\left(   G(\pi^{\ast}_{x,y})  - 2 |y-x|_{1}  >  - \frac{\theta}{40} \sqrt{N}   \right)  \leq \exp(-C_4 \theta^{\frac{3}{2}}) \, .
\end{equation}
\end{corollary}
\begin{proof} Consider the same path decomposition as the proof of Proposition \ref{pro:VarianceUpperBoundRephased} for the geodesic $\pi=\pi^{\ast}_{x,y}$ from $x$ to $y$ into sets of sub-paths $(\Pi^{\pi}_{i})_{i \in [3]}$. Recall that $\Pi^{\pi}_{1}$ contains at most two elements, and we further have that
\begin{equation}
|\Pi^{\pi}_3| \leq |\Pi^{\pi}_2| +1 \quad \text{ and } \quad |\Pi^{\pi}_2| \geq 1 \, .
\end{equation} As for \eqref{eq:RecomposePaths}, we get from Lemma \ref{lem:Category1}, Lemma \ref{lem:Category2} for $\lambda=\theta/8$, and Lemma \ref{lem:Category3} that
\begin{equation}\label{eq:RecomposePaths2}
G(\pi) - 2|\pi| \leq  \sqrt{N} \frac{\theta}{20} + |\Pi^{\pi}_3| \sqrt{N} \frac{\theta}{40} - |\Pi^{\pi}_2| \sqrt{N} \frac{\theta}{8} \leq -\sqrt{N}\frac{\theta}{40} 
\end{equation} with probability at least $1-\exp(-C_4\theta^{3/2})$ for some  constant $C_4>0$.
\end{proof}

It remains to show Lemma \ref{lem:Category3}. Due to the monotonicity of the last passage times with respect to the boundary parameters, it suffices to consider $\alpha=\beta=\frac{1}{2}$. We will proceed in two steps. First, we argue that the passage time for a path in $\Pi_3$ can be dominated using a  corner growth model on the half-quadrant, for which we have moderate deviation estimates for fixed paths. In a second step, we convert these estimates into uniform bounds on the last passage time using similar arguments as in Proposition 10.1 of \cite{BSV:SlowBond}. \\

Consider the corner growth model on the half-quadrant with boundary parameter $\frac{1}{2}$ given with respect to the environment $(\omega_x^{p_{\textup{half}}})_{x \in \Z^2}$,   where
\begin{equation}\label{def:HalfspaceEnvironment}
p_{\textup{half}}(x) := \begin{cases} 2 & \text{ if } x_1=x_2 \geq 0 \\
1 & \text{ if } x_1 > x_2 \geq 0 \\
0 & \text{ otherwise}
\end{cases}
\end{equation} for all $x=(x_1,x_2) \in \Z^2$; see  \cite{BBCS:Halfspace,BFO:HalfspaceStationary,PS:CurrentFluctuations}. This CGM on the half-quadrant has a natural interpretation as an interacting particle system: It is the TASEP on $\N$, where particles enter at site $1$ at rate~$\frac{1}{2}$. The next lemma follows from the fact that each path in $\Pi_3$ starts and ends at the same boundary part $R_+$ or $R_-$ of $\partial R$, but does not touch the opposite side.

\begin{lemma}\label{lem:DominationByHalfspace} Recall the environment $(\omega^p_x)$ with respect to $p$ from \eqref{def:EnvironmentStripe}. Let $\pi_{x,y} \in \Pi_3 $ be a path connecting $x=(x_1,x_2) \in R_+$ to $y=(y_1,y_2) \in R_+$. There exists a coupling $\mathbf{P}$ with
\begin{equation}
\mathbf{P}\left( \sum_{z \in \pi_{x,y}} \omega_z^{p_{\textup{half}}} = \sum_{z \in \pi_{x,y}} \omega_z^{p} \right) = 1 \, .
\end{equation} In particular,  the last passage time on the strip between some $x,y \in R_+$, when restricted to paths in $\Pi_3$, is stochastically dominated by the last passage time between $x$ and $y$ on the half-quadrant. By symmetry, a similar statement holds when $x,y\in R_-$. 
\end{lemma}

We remark that last passage times in the corner growth model on the half-quadrant, in the following denoted  by $G_{\textup{half}}(x,y)$ for $x,y\in R_+$, were investigated in \cite{BBCS:Halfspace}. The authors show that the last passage time between two points in $R_+$ of distance $n$  has fluctuations of order $n^{1/3}$, and that the limiting distribution is given by the Tracy-Widom GOE distribution. We will use the following moderate deviation estimate.

\begin{lemma}\label{lem:DominationByHalfspacePointwise} There exists a constant $c>0$ such that the geodesic $\pi^{\ast}_{x,y}$ in the environment $(\omega_x^{p_{\textup{half}}})$ with respect to $p_{\textup{half}}$ from \eqref{def:StationaryEnvironment} connecting $x,y \in \{(n,n) \colon n\in \N\}$ with $y \succeq x$ satisfies
\begin{equation}\label{eq:ModerateDeviationsStationaryBound}
\P\left( G_{\textup{half}}(\pi^{\ast}_{x,y}) > 2 |y-x|_1 + \lambda  |y-x|_1^{1/3} \right)  \leq  \exp(-c \lambda^{\frac{3}{2}})
\end{equation} for any $1 \leq \lambda \leq  |y-x|_1^{1/3}$. 
\end{lemma}
\begin{proof} Consider the corner growth model with respect to the environment given by the function $p_{\textup{swap}}:\Z^{2} \rightarrow [0,\infty)$ with 
\begin{equation*}%\label{def:HalfspaceEnvironmentTilted}
p_{\textup{swap}}(x) := \begin{cases} 2 & \text{ if } x_1 \geq  0 \text{ and } x_2= 0 \\
1 & \text{ if } x_1 \geq x_2 > 0 \\
0 & \text{ otherwise} \, ,
\end{cases}
\end{equation*}
and denote by $G_{\textup{swap}}(x,y)$ the corresponding last passage time.
As observed in Lemma 6.1 of \cite{BBCS:Halfspace}, we have for all $x,y$ on the diagonal $\{(n,n) \colon n\in \N\}$ that
\begin{equation}\label{eq:EquivalenceHalfspaceShift}
G_{\textup{half}}(x,y) \ \stackon{=}{\tiny(d)}  \ G_{\textup{swap}}(0,y-x)
\end{equation}
holds,  i.e., instead of putting i.i.d.\ Exponential-$\frac{1}{2}$-random variables on the diagonal, it is equivalent to study the last passage time when placing  i.i.d.\ Exponential-$\frac{1}{2}$-random variables on the $x_1$-axis. Note that the last passage time in the environment $(\omega_x^{p_{\textup{swap}}})$ with respect to $p_{\textup{swap}}$ is stochastically dominated by the last passage time of the stationary corner growth model on $\Z^2$, which is defined as the CGM on $\Z^2$ in the environment $(\omega_x^{p_{\Z}})$~with 
\begin{equation*}%\label{def:StationaryLPP} 
p_{\Z}(x) := \begin{cases} 2 & \text{ if } x_1= 0 \text{ or }  x_2= 0\\
1 & \text{ if } x_1, x_2 > 0 \\
0 & \text{ otherwise}
\end{cases}
\end{equation*}  for all $x\in \Z^2$. Let $G_{\Z}(x,y)$ denote the corresponding last passage time between $x,y\in \Z^2$, and recall that we write $G(i,j)=G((i,i),(j,j))$ for $i,j\in \N$ with $i<j$. Theorem 2.5 in \cite{B:ModerateDeviationsStationary} states that there exist constants $c,n_0>0$, such that 
\begin{equation*}
\P\left( G_{\Z}(0,n) > 4  n + \lambda  n^{\frac{1}{3}} \right)  \leq  \exp\big(-c \lambda^{\frac{3}{2}}\big) \, .
\end{equation*} holds for all $n>n_0$ and $1 \leq \lambda \leq n^{1/3}$.  Using  \eqref{eq:EquivalenceHalfspaceShift}, this implies \eqref{eq:ModerateDeviationsStationaryBound}.
\end{proof}

It remains to convert the path-wise bound in Lemma \ref{lem:DominationByHalfspacePointwise} into a uniform bound. % on all paths in $\Pi_3$ in order to conclude Lemma \ref{lem:Category3}.

\begin{proof}[Proof of Lemma \ref{lem:Category3}] Note that by Lemma \ref{lem:DominationByHalfspace}, it suffices to show \eqref{eq:Category3Bound}  for all $x,y \in R_+$ with respect to the environment according to $p_{\textup{half}}$ from \eqref{def:HalfspaceEnvironment} instead of $p$ from \eqref{def:EnvironmentStripe}. By symmetry, we get a similar bound for $x,y \in R_{-}$. We will now  follow the strategy of Proposition 10.5 in~\cite{BSV:SlowBond}.  By Theorem 1.5 in \cite{BBCS:Halfspace}, which bounds the fluctuations of the last passage times on the half-quadrant between two points on the diagonal, there exists some $c>0$ such that
\begin{equation}\label{eq:ExpectationHalfspace}
-c n^{\frac{1}{3}} \leq \E\left[ G_{\textup{half}}(0, n) \right] - 4n \leq c n^{\frac{1}{3}}
\end{equation} 
holds for all $n \in \N $. Note that we have the bound
\begin{equation}\label{eq:MinMaxDecomposition}
G_{\textup{half}}(0,3n) \geq \max_{i,j \in [n,2n], i<j} G_{\textup{half}}(0,i) + G_{\textup{half}}(i,j)  +  G_{\textup{half}}(j,3n) 
\end{equation} for all $n\geq 1$. For all $x,y$ with $y \succeq x$, we let in the following
\begin{equation*}%\label{eq:CenteredLPP}
\tilde{G}_{\textup{half}}(x,y) := G_{\textup{half}}(x,y) - \E\left[G_{\textup{half}}(x,y)\right]
\end{equation*} be the\textbf{ normalized last passage time} between $x$ and $y$. From \eqref{eq:ExpectationHalfspace} and \eqref{eq:MinMaxDecomposition}, we get
\begin{equation}\label{eq:LowerBoundUniformG}
\tilde{G}_{\textup{half}}(0,3n) \geq \min_{ i \in [n,2n]} \tilde{G}_{\textup{half}}(0,i) +  \min_{ j \in [n,2n]} \tilde{G}_{\textup{half}}(j,3n) + \max_{\substack{i,j \in [n,2n]}} \tilde{G}_{\textup{half}}(i,j)  -  8cn^{\frac{1}{3}}
\end{equation} for all $n\geq 1$.  Now for fixed  $\theta>0$, let $F_1,F_2$ and $F_3$ be the events 
\begin{align*}
F_1 &:= \left\{  \min_{ i \in [n,2n]} \tilde{G}_{\textup{half}}(0,i) \geq -\frac{\theta}{32} n^{\frac{1}{3}}  \right\}     \\
F_2 &:= 	\left\{ \min_{ j \in [n,2n]} \tilde{G}_{\textup{half}}(j,3n) \geq -\frac{\theta}{32} n^{\frac{1}{3}}  \right\}	\\
F_3 &:= 	\left\{ \max_{i,j \in [n,2n]} \tilde{G}_{\textup{half}}(i,j) \geq \frac{\theta}{4} n^{\frac{1}{3}} \right\} \, .
\end{align*}
As a consequence of Theorem 4.2(iii) in \cite{BGZ:TemporalCorrelation}, which states a lower bound on the last passage times between sufficiently far enough points when restricting the geodesics to stay within a parallelogram, we see that on the environment $p$ from \eqref{def:EnvironmentStripe} with $\alpha=\beta=1$, there exists constants $\bar{c},t_0>0$ such that 
\begin{equation}
\P( \min_{i,j \in [3n] \colon i\leq j-n} G_{p}(i,j) - 2(j-i) \geq -t n^{1/3}) \geq 1- \exp(-\bar{c} t^{1/3})
\end{equation} for all $t>t_0$. Together with Lemma \ref{lem:DominationByHalfspace} in order to bound $\tilde{G}_{\textup{half}}(\cdot,\cdot)$ using $G_{p}(\cdot,\cdot)$, and \eqref{eq:ExpectationHalfspace}, we see that
\begin{equation*}
\P(F_1 ) = \P(F_2 ) \geq \frac{1}{2}
\end{equation*} holds for all $n\geq 1$, provided that $\theta$ is sufficiently large. %, see also Proposition 12.2 in \cite{BSV:SlowBond} for the corresponding statement for Poissonian LPP. 
Choosing $\theta\geq 128 c$ large enough, and using the FKG-inequality together with \eqref{eq:LowerBoundUniformG}, we obtain for all $n \geq 1$ that
\begin{equation*}
\P\left( \tilde{G}_{\textup{half}}(0,3n) \geq \frac{\theta}{8} n^{\frac{1}{3}} \right) \geq \P(F_1  \cap F_2 \cap F_3 )\geq \P(F_1 )\P(F_2 )\P(F_3 ) \geq \frac{1}{4} \P(F_3 ) \, .
\end{equation*} Recall that it suffices by Lemma \ref{lem:DominationByHalfspace} to consider the environment according to $p_{\textup{half}}$ from \eqref{def:HalfspaceEnvironment} instead of $p$ from \eqref{def:EnvironmentStripe}. We apply now Lemma \ref{lem:DominationByHalfspacePointwise} for $n=\theta^{-1}N^{3/2}$ and $\lambda=\theta/8$ in order to bound $\P(F_3)$. Using \eqref{eq:ExpectationHalfspace} in order to bound the expectation of $\tilde{G}_{\textup{half}}(i,j)$ for $0<i<j<n$, and noting that we can replace the rectangle $R=R^n_N$ in \eqref{eq:Category3Bound} by $R^n_{n}$ without changing the statement, we conclude Lemma \ref{lem:Category3}. 
\end{proof}

\section{Mixing times in the triple point} \label{sec:MixingTimesTriplePoint}

In this part, we prove Theorem \ref{thm:Triple}. Our key tool is a duality result for the stationary corner growth model on the strip; see Section \ref{sec:StationaryLPP}. More precisely, similar to \cite{BCS:CubeRoot,P:Duality}, we relate the semi-infinite geodesics in the stationary CGM to the competition interface for the reversed process, which will again be a stationary CGM. The arguments used in the proof of Theorem \ref{thm:MaxCurrent} allow us to estimate the time it takes until a semi-infinite geodesic hits the boundary of the strip. We then couple the stationary CGM to the TASEP  in the triple point, where all sites are initially occupied by second class particles. We bound the time until all second class particles have left the segment by the number of steps until a particular semi-infinite geodesics touches both boundaries of the strip. \\

Recall the stationary CGM in environment $(\omega_x^{p_{\textup{stat}}})_{x \in \Z^2}$ with $p_{\textup{stat}}$ from \eqref{def:StationaryEnvironment}. We let $\pi^{\ast}_{(1,0)}=(v^{\ast}_1=(1,0),v^{\ast}_2,\dots)$ be the semi-infinite geodesic starting from $(1,0)$, and set
\begin{equation}
\nu^{n}_{\textup{hit}} := \inf\{ j \geq n \colon v^{\ast}_j  \in \{ (m,m),(m+N,m) \colon m \in \N \} \}
\end{equation} as the first generation after reaching $\mathbb{L}_n$ at which the semi-infinite geodesic touches the boundary of the strip. Recall that Lemma \ref{lem:ExistenceSemiInfinite} ensures the existence and uniqueness of $\pi^{\ast}_{(1,0)}$, and note that the same arguments as in the proof of Lemma \ref{lem:ExistenceSemiInfinite} ensure that $\nu^{n}_{\textup{hit}}$ is almost surely finite for every $n\in \N$. The following proposition quantifies the first  generation where the boundary is hit in $\pi^{\ast}_{(1,0)}$ after reaching $\mathbb{L}_n$. 

\begin{proposition}\label{pro:BoundaryHittingStationary} For all $\varepsilon>0$, there exist constants  $\tilde{c}=\tilde{c}(\varepsilon)>0$ and $N_0\in \N$ such that 
\begin{equation}\label{eq:BoundaryHittingStationary}
\P\left( \nu^{n}_{\textup{hit}} - n \leq \tilde{c}N^{\frac{3}{2}} \right) \geq 1 - \varepsilon
\end{equation} holds for all $n \geq N \geq N_0$.
\end{proposition}

In order to show Proposition \ref{pro:BoundaryHittingStationary}, we introduce a reversed process of the stationary CGM on the strip; see also Section 4.1 in \cite{BCS:CubeRoot} for a similar process for the stationary CGM on $\Z^2$. Recall that $G_{\textup{stat}}$ denotes the last  passage time in the environment $\omega^{p_{\textup{stat}}}=(\omega_x^{p_{\textup{stat}}})_{x \in \Z^2}$ for $p_{\textup{stat}}$ from \eqref{def:StationaryEnvironment} on the strip of width  $N\in \N$. Further, recall $I_{i,j},J_{i,j}$ and $Y_{i,j}$ with respect to $G_{\textup{stat}}$ defined in Lemma \ref{lem:StationarityLPPprocess}. In the following, we fix some $M \geq N$ and denote by
\begin{equation*}
\mathcal{P}_M := \left\{  x\in \Z^2 \colon (1,0) \preceq x \preceq (M+N-1,M)  \right\} \cap \mathcal{S}(p_{\textup{stat}})
\end{equation*} the parallelogram which we get from all sites below $(N+M-1,M)$ in the support of $p_{\textup{stat}}$. We define an environment $(\omega^{\textup{rev}}_x)_{x \in \mathcal{P}_M}$ on $\mathcal{P}_M \subseteq \Z^2$ by 
\begin{equation}\label{def:ReversedEnvironmentStripe}
\omega^{\textup{rev}}_x := \begin{cases}
I_{N+M+1-x_1,M-x_2} & \text{ if }  x_1-M= x_2 \geq 1, \text{ or }  x_2=0 \text{ and } x_1 \in [N] \\
J_{N+M-x_1,M-x_2+1} & \text{ if }  x_1=x_2 \geq  1 \\
Y_{N+M-x_1,M-x_2}& \text{ otherwise }   \\
\end{cases}
\end{equation} for all $x=(x_1,x_2) \in \mathcal{P}_M$.  Let $G_{\textup{rev}}(i,j)$ denote the last passage time between $(1,0)$ and $(i,j)\in \mathcal{P}_M$ with respect to  $\omega^{{\textup{rev}}}=(\omega^{\textup{rev}}_x)_{x \in \mathcal{P}_M}$. We will  see in the next lemma that the last passage times $G_{\textup{rev}}$ describe a time-reversal of the stationary TASEP in the triple point, which is again a TASEP with open boundaries, and thus has a CGM representation.
% In the following, we fix some $x_{\textup{rev}} \in \Z^2$ with $x_{\textup{rev}}=(m+N,m)$ for some $m> N$, and let for all $y\in \Z^2$ with $(1,1) \preceq_{\c} y  \preceq_{\c} x_{\textup{rev}} $ the \textbf{reversed last passage time} $G^{\textup{rev}}(y)$ be given by 
%\begin{equation}\label{def:ReversedLPP}
%G_{\textup{rev}}(y) := G(1,x_0) - G(1,x_{\textup{rev}}  - y) \, .
%\end{equation} We denote by 
%\begin{equation}
%R_{x_0} := \{ y \in \Z^2 \colon (1,1) \preceq_{\c} y  \preceq_{\c} x_0)  \}  \cap \mathcal{S}
%\end{equation} the rectangle, in which we define the reversed process. Due to a version of Burke's property, which says that the time-reversal of the TASEP in the triple point has again the law of a TASEP in the triple point, but with particles traveling to the other direction, we have the following characterization of the reversed stationary time.
\begin{lemma}\label{lem:ReversedProcessStationaryLPP} The collection of random variables $(\omega^{\textup{rev}}, G_{\textup{rev}} )$ has the same law as $ (\omega^{p_{\textup{stat}}}, G_{\textup{stat}}) $ when restricted to $\mathcal{P}_M$. Moreover, the last passage times $G_{\textup{rev}}$ in $\mathcal{P}_M$ satisfy %recursion
%\begin{equation}\label{eq:RucursionLastPassage}
%G_{\textup{rev}}(x,y) = \max\big(G_{\textup{rev}}(x,y-(1,0)), G_{\textup{rev}}(x,y-(0,1))) +  \omega^{\textup{rev}}_y 
%\end{equation} and the
\begin{equation} \label{eq:ReversalsOfEachOther}
G_{\textup{rev}}(i,j) =  G_{\textup{stat}}((1,0),(N+M,M))- G_{\textup{stat}}((1,0),(N+M-i,M-j)) \, .
\end{equation}
\end{lemma}
\begin{proof} By Lemma \ref{lem:StationarityLPPprocess}, the random variables $(\omega^{\textup{rev}}_x)$ have on $\mathcal{P}_M$ the same marginal distributions as $(\omega_x^{p_{\textup{stat}}})$, and are independent. %The recursion \eqref{eq:RucursionLastPassage} is an equivalent formulation of the last passage times independent environments for $\Z^2$. 
As in Section 4.1 of  \cite{BCS:CubeRoot}, a calculation shows~that $$\bar{G}_{\textup{stat}}(i,j):=G_{\textup{stat}}((1,0),(N+M,M))- G_{\textup{stat}}((1,0),(i,j)) $$ for $(i,j)\in \mathcal{P}_M$ satisfies the recursion
\begin{equation}\label{eq:NecessaryLPP}
\bar{G}_{\textup{stat}}(i,j) = \max(\bar{G}_{\textup{stat}}(i+1,j), \bar{G}_{\textup{stat}}(i,j+1) ) + \omega^{\textup{rev}}_{(N+M-i,M-j)} \, .
\end{equation} This gives \eqref{eq:ReversalsOfEachOther} since the property  \eqref{eq:NecessaryLPP} characterizes the last passage times $G_{\textup{rev}}$.
%We follow a similar argument as for Lemma 4.3 in \cite{BCS:CubeRoot}. Recall the notation from Lemma \ref{lem:StationarityLPPprocess}. We define an environment $(\omega_{\textup{rev}}^x)$ where 
%\begin{equation}\label{def:ReversedEnvironment}
%\omega_{\textup{rev}}^x = \begin{cases}  I_{(x^1_0-x^1, x^2_0-x^2)} & \text{ if } x^1=0 \\
%Y_{x^1_0-x^1, x^2_0-x^2} & \text{ otherwise}
%\end{cases}
%\end{equation} for all $1 \preceq x=(x^1,x^2) \preceq x_0$. By Lemma \ref{lem:StationarityLPPprocess}, we see that the recursion in \eqref{eq:RucursionLastPassage} holds for the last passage times with respect to $(\omega_{\textup{rev}}^x)$. Together with the fact that the boundary laws agree by construction, we see that the last passage times $(G_{\textup{rev}}(y))_{y \in R_{x_0}}$ must have the same law as a stationary corner growth model on the strip.
\end{proof}

Next, we relate  on $\mathcal{P}_M$ the semi-infinite geodesic $\pi^{\ast}_{(1,0)}=(v_1^{\ast},v_2^{\ast},\dots)$ in the stationary CGM for $\omega^{p_{\textup{stat}}}$ to competition interfaces in the reversed process for $\omega^{{\textup{rev}}}$.

\begin{lemma}\label{lem:CompetitionInterfaceReversed} Fix some $n,m \geq N$ with $n+m<M$. Let $(v^{\ast}_{n},\dots,v^{\ast}_{n+m})$ be the trajectory of $\pi^{\ast}_{(1,0)}$ restricted between the anti-diagonals $\mathbb{L}_{n}$ and $\mathbb{L}_{n+m}$. Let $(z^1,\dots,z^{m+1})$ denote the corresponding lattice path which we obtain by rotating the environment, i.e., we set
\begin{equation}\label{def:ReversedPathCompetition}
z^i := (N+M,M)- v^{\ast}_{n+m+1-i}
\end{equation} for all $i\in [m+1]$. Then for all $z^i \notin \partial \mathcal{P}_M$ with $i \in [m]$, the recursion \eqref{eq:CompetionInterfaceArgMin} holds, i.e.,
\begin{equation}\label{eq:InverseRecursion}
z^{i+1} = \textup{arg}\min\big( G_{\textup{rev}}(z^{i}+(0,1)), G_{\textup{rev}}(z^{i}+(1,0))\big)  \,  .
\end{equation} 
\end{lemma}

\begin{proof} Note that all $v_i^{\ast} \in \pi^{\ast}_{(1,0)}$ with $v^{\ast}_{i+1} \notin \partial \mathcal{P}_M$ satisfy the recursion
\begin{equation*}
v_{i}^{\ast}  = \begin{cases} v^{\ast}_{i+1}-(0,1) & \text{ if } G_{\textup{stat}}((1,0),v^{\ast}_{i+1}-(0,1))> G_{\textup{stat}}((1,0),v^{\ast}_{i+1}-(1,0)) \\
v^{\ast}_{i+1}-(1,0) & \text{ if } G_{\textup{stat}}((1,0),v^{\ast}_{i+1}-(0,1))< G_{\textup{stat}}((1,0),v^{\ast}_{i+1}-(1,0))\,  ; \end{cases}
\end{equation*} see for example equation (2.3) in \cite{BCS:CubeRoot}. Using \eqref{eq:ReversalsOfEachOther} in Lemma \ref{lem:ReversedProcessStationaryLPP}, this yields the recursion \eqref{eq:InverseRecursion} for all sites in the path $(z^1,\dots,z^{m+1})$  which are not contained in $\partial \mathcal{P}_M$.
\end{proof}

%We have now all tools in order to show Proposition \ref{pro:BoundaryHittingStationary}.

\begin{proof}[Proof of Proposition \ref{pro:BoundaryHittingStationary}] In the following, we take the setup from Lemma \ref{lem:CompetitionInterfaceReversed}. We fix $m=cC^{-1}\log(\varepsilon^{-1})N^{3/2}$ where the constants $c$ and $C$ are taken from Proposition \ref{pro:CoalescenceTimes} with $M=m+n+1$. We claim that with probability at least $1-\varepsilon$, the semi-infinite geodesic $\pi^{\ast}_{(1,0)}$ hits the boundary $\partial \mathcal{S}(p_{\textup{stat}})$ of the strip between the anti-diagonals $\mathbb{L}_n$ and $\mathbb{L}_{n+m}$. If $v_{n+m}^{\ast} \in \partial \mathcal{S}$, there is nothing to show. Otherwise, we follow the path $(z^1,\dots,z^{m+1})$ from \eqref{def:ReversedPathCompetition} between $\mathbb{L}_N$ and $\mathbb{L}_{N+m}$ in the environment $(\omega^{\textup{rev}}_{x})_{x \in \mathcal{P}_M}$ until we reach $\partial \mathcal{P}_M$. \\ 

Let $k$ be the first generation such that $z^k\in \partial \mathcal{S}$, and $k=\infty$ if no such $k$ exists. We claim that due to Lemma~\ref{lem:ReversedProcessStationaryLPP}, the path $(z^1,\dots,z^{\min(k,m+1)})$ has the law of a competition interface in $(\omega^{\textup{rev}}_{x})_{x \in \mathcal{P}_M}$. More precisely, we decompose $\mathbb{L}_{N+1}$ according to $z^1=(z_1^1,z_2^1) \in \mathbb{L}_{N+1}$ as
\begin{equation*}
\mathbb{L}^+_{N+1} := \{ (x_1,x_2) \in \Z^2 \colon x_1+x_2=N+1\text{ and }  x_1 \leq z_1^1 \}  
\end{equation*} and $\mathbb{L}^-_{N+1} := \mathbb{L}_{N+1}  \setminus \mathbb{L}^+_{N+1}  $. We construct the competition interface starting at $z^1$ until we reach $\partial \mathcal{S}$, using the recursive definition \eqref{def:CompetionInterface} with respect to the coloring $(\tilde{\Gamma}^+,\tilde{\Gamma}^-)$ for
\begin{align*}
\tilde{\Gamma}^+ &:= \left\{ x \in \mathcal{P}_M \colon |x|_1 >N \text{ and } \pi^{\ast}_{((1,0),x)} \cap \mathbb{L}^+_{N+1} \neq \emptyset \right\} \\
\tilde{\Gamma}^- &:= \left\{ x \in \mathcal{P}_M  \colon |x|_1 >N \text{ and } \pi^{\ast}_{((1,0),x)} \cap \mathbb{L}^-_{N+1} \neq \emptyset\right\} \, ,
\end{align*} 
i.e., we color $x$ according to the site where the geodesic from $(1,0)$ to $x$ passes through $\mathbb{L}_{N+1}$; see also Figure \ref{fig:SemiCompetition}. Note  that this definition agrees with the construction in Section~\ref{sec:CompetitionInterfaces} for a fixed initial growth interface. In particular, the recursion  \eqref{eq:CompetionInterfaceArgMin}, respectively \eqref{eq:InverseRecursion}, holds. Since by Lemma~\ref{lem:ReversedProcessStationaryLPP},  $(\omega^{\textup{rev}}_{x})_{x \in \mathcal{P}_M}$ is an environment of a stationary corner growth model on the strip,  Lemma~\ref{lem:Coalescence} ensures that the coalescence of all geodesics connecting $\mathbb{L}_{N}$ and $\mathbb{L}_{N+m}$ in the environment $(\omega_x^{\textup{rev}})_{x \in \mathcal{P}_M}$ is sufficient  for the event that the competition interface reaches the boundary before hitting $\mathbb{L}_{N+m}$. Apply Proposition~\ref{pro:CoalescenceTimes} to conclude.
\end{proof}

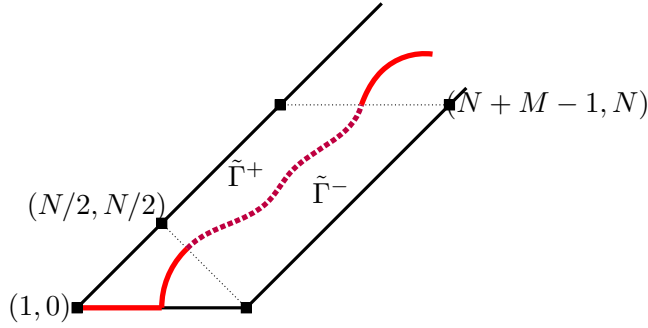
\begin{figure}
\begin{center}
\begin{tikzpicture}[scale=.45]

\draw[densely dotted] (0+2.5,0+2.5) -- (5,0);
\draw[densely dotted] (0+6,0+6) -- (11,6);

   \draw[line width=1.2pt] (0,0) -- (5,0);

   \draw[line width=1.2pt] (0,0) -- (9,9);
   \draw[line width=1.2pt] (5,0) -- (11.5,6.5);
   
%\draw[darkblue, line width =2 pt,densely dotted] (0,0) to[curve through={(0.4,0.3) ..(1.6,1) .. (2.4,1.35)..(2.95,2.65) }] (3,3);

%\draw[darkblue, line width =2 pt,densely dotted] (3,3) to[curve through={(4.2,3.2)..(5.6,3.1) .. (5.8,5.7)}] (6,6);

%\draw[darkblue, line width =2 pt,densely dotted] (6,6) to[curve through={(6.5,6.3) ..(7.8,7) .. (8.4,7.35)..(8.95,8.65) }] (9,9);

%\draw[blue, line width =2 pt,densely dotted] (2.5,-2.5) to[curve through={(1+2.5,1.6-2.5) .. (1.35+2.5,2.4-2.5)..(2.65+2.5,2.95-2.5) }] (2.5+3,-2.5+3);

%\draw[blue, line width =2 pt,densely dotted] (2.5+3,-2.5+3) to[curve through={(5.2,1.5)..(5.6,3.1)}] (2.5+6,-2.5+6);

%\draw[blue, line width =2 pt,densely dotted] (2.5+6,-2.5+6) to[curve through={(6.5+2.5,6.7-2.5) .. (8.4+2.5,9.15-2.5) }] (11.5,6.5);

%\visible<2>{

\draw[red, line width =2 pt] (0,0) to (2.5,0);

\draw[red, line width =2 pt] (2.5,0) to[curve through={(2.5+0.65,2.5-0.85)}] (2.5+0.75,2.5-0.75);

\draw[red, line width =2 pt] (8.4,6) to[curve through={(9-0.5,9-2.75)..(9,7)}] (9+1.5,7.5);

\draw[purple, line width =2 pt,densely dotted] (2.5+0.75,2.5-0.75) to[curve through={(2.5+0.75+0.1,2.5-0.75+0.1)..(5.6,3.1)..(5.68,3.18) ..(5.6+0.7,3.1+0.9)..(8.37,5.9)}] (8.4,6);

%}
   
	%\filldraw [fill=black] (2.5-0.15,-2.5-0.15) rectangle (2.5+0.15,-2.5+0.15);   
	
		\node[scale=1] (x1) at (-1.1,0){$(1,0)$} ;
	\node[scale=1] (x1) at (-2.4+3,3){$(N/2,N/2)$} ;	
	%	\node[scale=1] (x3) at (9+0.7,9){$a_4$} ;
		\node[scale=1] (x4) at (9+4.9,6){$(N+M-1,N)$} ;
		
		\node[scale=1] (x1) at (5,4){$\tilde{\Gamma}^+$} ;		
		\node[scale=1] (x1) at (7.5,3.5){$\tilde{\Gamma}^-$} ;

 	\filldraw [fill=black] (2.5+2.5-0.15,-2.5+2.5-0.15) rectangle (2.5+2.5+0.15,-2.5+2.5+0.15);   
 	\filldraw [fill=black] (11-0.15,6-0.15) rectangle (11+0.15,6+0.15);

 %	\filldraw [fill=red] (5.6-0.2,3.1-0.2) rectangle (5.6+0.2,3.1+0.2);     
 	
	\filldraw [fill=black] (-0.15,-0.15) rectangle (0.15,0.15);   
 	\filldraw [fill=black] (2.5-0.15,2.5-0.15) rectangle (2.5+0.15,2.5+0.15);   
 	\filldraw [fill=black] (6-0.15,6-0.15) rectangle (6+0.15,6+0.15);

	\end{tikzpicture}	
\end{center}	
	\caption{ \label{fig:SemiCompetition}Visualization of the semi-infinite geodesic in the proof of Proposition \ref{pro:BoundaryHittingStationary} started from $(1,0)$. On the dotted purple part, the semi-infinite geodesic corresponds to a competition interface in the reversed process.}
 \end{figure}

Before coming to the proof of Theorem \ref{thm:Triple}, we require the following two lemmas on the coupling and crossing  of  geodesics. 

\begin{lemma}\label{lem:CoalescenceStationaryAllFull} Fix $M\in \N$. Let $\pi_{1,M}^{\ast}=(\pi_{1,M}^{\ast}(m))_{m\in [2M-1]}$ be the geodesic between $(1,0)$ and $(M,M)$ in environment $(\omega^p_x)$ for $p$ from \eqref{def:EnvironmentStripe} with boundary parameters $\alpha=\frac{1}{2}$ and $\beta=\frac{1}{2}$.  Let  $\pi^{\ast}_{\textup{stat}}=(\pi_{\textup{stat}}^{\ast}(m))_{m\in \N}$  be the semi-infinite geodesic in environment $(\omega_x^{p_{\textup{stat}}})$ with $p_{\textup{stat}}$ from \eqref{def:StationaryEnvironment}, starting at $(1,0)$. Then there exists a coupling of the two environments such that for all $\varepsilon>0$, we find some constant $c^{\prime}(\varepsilon)>0$ such that for all $M> 2c^{\prime}N^{\frac{3}{2}}$
\begin{equation}\label{eq:CoalescenceInHomo}
 \liminf_{N \rightarrow \infty} \P\big( \pi_{1,M}^{\ast}(m) = \pi^{\ast}_{\textup{stat}}(m)  \text{ for all } m \in [ c^{\prime}N^{\frac{3}{2}} , M- c^{\prime}N^{\frac{3}{2}}] \big) \geq 1- \varepsilon \, .
\end{equation}
\end{lemma}
\begin{proof} Note that $p$ and $p_{\textup{stat}}$ differ only in the sites $\{ (x_1,x_2) \in \mathcal{S}(p) \colon x_2 \leq 0 \}$. Consider the coupling, where the environments $(\omega_x^{p})$ and $(\omega_x^{p_{\textup{stat}}})$ agree on $\Z^2 \setminus \{ (x_1,x_2) \in \mathcal{S}(p) \colon x_2 \leq 0 \}$, and are independently chosen otherwise. Note that by  Proposition \ref{pro:CoalescenceTimes}, we find some constant $c^{\prime}=c^{\prime}(\varepsilon)$ such that 
\begin{align}
\label{eq:CoalescenceAll}\P( \pi^{\ast}_{x,y} \cap \pi^{\ast}_{x^{\prime},y^{\prime}} \neq \emptyset \text{ for all } x,x^{\prime} \in  \mathbb{L}_{N+1} \text{ and } y,y^{\prime} \in \mathbb{L}_{c^{\prime}N^{3/2}}) \geq 1- \frac{\varepsilon}{2}  \\
\P( \pi^{\ast}_{x,y} \cap \pi^{\ast}_{x^{\prime},y^{\prime}} \neq \emptyset \text{ for all } x,x^{\prime} \in  \mathbb{L}_{M-c^{\prime}N^{3/2}} \text{ and } y,y^{\prime} \in \mathbb{L}_{M} ) \geq 1- \frac{\varepsilon}{2} \nonumber
\end{align} holds. Since $\pi_{1,M}^{\ast}$ and $\pi^{\ast}_{\textup{stat}}$ are defined with respect to the same environment between $\mathbb{L}_{N+1}$ and $\mathbb{L}_M$, this allows us to conclude \eqref{eq:CoalescenceInHomo}.
\end{proof}

\begin{lemma}\label{lem:CrossingProbability} For $n,m\in \N$, recall the rectangle $R_n^m$ from \eqref{def:Rectangle} with corners $(a_i)_{i \in [4]}$ in counter-clockwise order and width $N\in \N$  such that $a_1,a_2 \in \mathbb{L}_n$ and $a_3,a_4 \in \mathbb{L}_{n+m-1}$. Let 
\begin{align}
B^{n,m}_{+} &:= \{ \pi^{\ast}_{a_2,a_3} \cap \{ (i,i) \colon i \in \N \} \neq \emptyset \} \\
 B^{n,m}_{-} &:= \{ \pi^{\ast}_{a_1,a_4} \cap \{ (i+N,i) \colon i \in \N \} \neq \emptyset \} \, ,
\end{align} i.e., the geodesics $\pi^{\ast}_{a_1,a_4}$  and $\pi^{\ast}_{a_2,a_3}$ touch both diagonal boundaries of $R$, respectively.  Then for all $\varepsilon>0$, there exists some $c=c(\varepsilon)>0$ such that for all $n \geq N$ and all $m \geq cN^{3/2}$
\begin{equation}
\P(B^{n,m}_{+}) = \P(B^{n,m}_{-}) \geq \frac{1}{2} - \varepsilon \, .
\end{equation}
\end{lemma}

\begin{proof} By symmetry, it suffices to show $\P(B^{n,m}_{+}) \geq 1/2-\varepsilon$. Recall $c^{\prime}=c^{\prime}(\varepsilon/3)$ from Lemma~\ref{lem:CoalescenceStationaryAllFull} and $\tilde{c}=\tilde{c}(\varepsilon/3)$ from Proposition~\ref{pro:BoundaryHittingStationary}. In the following, we set $c:=3c^{\prime} +\tilde{c}$. By Lemma~\ref{lem:CoalescenceStationaryAllFull}, the geodesic $\pi^{\ast}_{a_2,a_3}$ agrees with probability at least $1-\varepsilon/3$ with the stationary geodesic $\pi^{\ast}_{\textup{stat}}$ started from $(1,0)$ for all positions $i\in [c^{\prime}N^{3/2}+N,m-c^{\prime}N^{3/2}]$. 
By Proposition~\ref{pro:BoundaryHittingStationary}, we see that $\pi^{\ast}_{\textup{stat}}$ touches with probability at least $1-\varepsilon/3$ the boundary of $R$ in $[n+2c^{\prime}N^{3/2},m-c^{\prime}N^{3/2}]$. \\

We claim that on the event that all geodesics between $\mathbb{L}_n$ and $\mathbb{L}_{n+c^{\prime}N^{3/2}}$ coalesce, which has probability at least $1-\varepsilon/3$ by Proposition \ref{pro:CoalescenceTimes}, both boundary sides are equally likely to be hit by $\pi^{\ast}_{\textup{stat}}$. To see this, assume that $N$ is even as a similar argument will hold for odd $N$. Consider the semi-infinite geodesic $\tilde{\pi}^{\ast}_{\textup{stat}}$ started from $(N/2,1-N/2)$ with respect to the environment $(\omega_x^{\tilde{p}_{\textup{stat}}})$ for the function $\tilde{p}_{\textup{stat}} \colon \Z^2 \mapsto \R$ given by
\begin{equation}
\tilde{p}_{\textup{stat}}(x_1,x_2):= p_{\textup{stat}}((x_2+N/2),(x_1-N/2))
\end{equation} for all $(x_1,x_2)\in \Z^2$ and $p_{\textup{stat}}$ from \eqref{def:StationaryEnvironment}. In words, we obtain $\tilde{p}_{\textup{stat}}$ by flipping $p_{\textup{stat}}$ along the axis $\{(x_1,x_2) \colon x_2=x_1-N/2 \} $. Since $p_{\textup{stat}}(x)=\tilde{p}_{\textup{stat}}(x)$ for all $x\in \N^2$ with $|x|_1>N$, we use the same coupling as in Lemma \ref{lem:CoalescenceStationaryAllFull} for the environments $(\omega_x^{p_{\textup{stat}}})$ and $(\omega_x^{\tilde{p}_{\textup{stat}}})$ to get that $\tilde{\pi}^{\ast}_{\textup{stat}}$ and $\pi^{\ast}_{\textup{stat}}$ must agree above $\mathbb{L}_{n+c^{\prime}N^{3/2}}$ when all geodesics between $\mathbb{L}_n$ and $\mathbb{L}_{n+c^{\prime}N^{3/2}}$ coalesce. The claim follows now by the symmetry of the environments with respect to $\tilde{p}_{\textup{stat}}$ and $p_{\textup{stat}}$. Thus, combining all of the above observations, this finishes the proof.
\end{proof}

%We now have all tools in order to show Theorem \ref{thm:Triple}.

\begin{proof}[Proof of Theorem \ref{thm:Triple}] %In the following, we show that if a specific semi-infinite geodesic touches the boundary of the stripe, this provides a sufficient condition such that all second class particles in the disagreement process have exited, and we conclude by Lemma \ref{lem:MixingBoundByExiting}. \\ 

  Recall from Section \ref{sec:TASEPInteracting} the TASEP with open boundaries as an interacting particle system and consider the disagreement process $(\xi_t)_{t \geq 0}$ between $\mathbf{1}$ and $\mathbf{0}$ taking values in $\{0,1,2\}^{N}$, noting that $\xi_0(i)=2$ for all $i\in [N]$ by construction. Further, recall from Section \ref{sec:DefinitionTASEPLPP} that we assign labels $u \in \N$ to the first class particles as they enter the segment at site $1$, and we denote by $\ell_t(u)$ the position of the particle with label $u$ at time $t \geq 0$. Here, we use the convention that $\xi_s({\ell_s(u) + 1}) = 0$ if label $u$ was not assigned by time $s$, the particle $u$ has reached site $N$ by time~$s$, or the particle $u$ has left the segment by time~$s$. For $T \geq 0$, we define the following event $A_T$: There exists some $k=k(T) \in \N$, and a sequence of times $0 = t_0 \leq  t_1 \leq t_2 \leq \dots \leq t_k \in [0,T]$ such that 
\begin{equation}\label{eq:NoParticleBefore}
\xi_s({\ell_s(i) + 1}) = 0
\end{equation} holds for all $s \in [t_{i-1},t_{i})$ and $i\in [k]$. Further, we require that there exists some $\tau \in  [t_{k-1},t_{k})$ such that
\begin{equation}\label{eq:TerminationCondition}
\ell_\tau(k) = N \, . 
\end{equation}
In words, the event $A_T$ says that we can find a sequence of time intervals such that during the $i^{\text{th}}$ interval, the particle of label $i$ never sees a particle to its right while it is contained in the segment, even if it jumps during this time. Further, the $k^{\text{th}}$ particle is located at site $N$ at some time $\tau \in  [t_{k-1},t_{k})$ less than $T$. The proof of Theorem \ref{thm:Triple} is now split into four steps. First, we argue that if $\P(A_T) \geq \varepsilon$ holds for some $\varepsilon>0$ and $T>0$, then we have that
\begin{equation}\label{eq:DisagreeingRelation}
\P( \xi_T \in \{0,1\}^{N}) \geq \varepsilon \, ,
\end{equation} i.e.\ with probability at least $\varepsilon$, the corresponding disagreement process contains no second class particles at time $T$. Second, we give an alternative representation of the event $A_T$ using tagged particles. In a third step, we link the tagging mechanism to semi-infinite geodesics. In a last step, we combine the above observations to deduce the desired upper bound on the mixing time of the TASEP with open boundaries. \\

Let us start with the first claim, i.e.\ we show that when the event $A_T$ holds for some $T \geq 0$ with some $k=k(T)$ and $t_1,t_2,\dots,t_k$ depending on $k$ and $T$, we get that all second class particles in the disagreement process have left by time $T$. To see this, let $j-1$ be the largest integer such that the particle with label $j-1$ has not entered by time $t_{j-1}$, and note that $j \in [k]$ by construction. We claim that for all $i\in [j,k]$, there can not be a second class particle to the left of the particle with label $i$ at time $s$ for all $s \in [t_{i-1},t_i)$. To see this, observe that after the particle with label $j$ enters the segment, it will never see a second class particle at time $s$ directly to its right for all $s<t_{j}$. Since the particle with label $j+1$ is placed to the left of the particle with label $j$, we obtain the claim by induction over $i\in [j,k]$. Condition \eqref{eq:TerminationCondition} now ensures that all second class particles must have left the segment by time $\tau$. \\

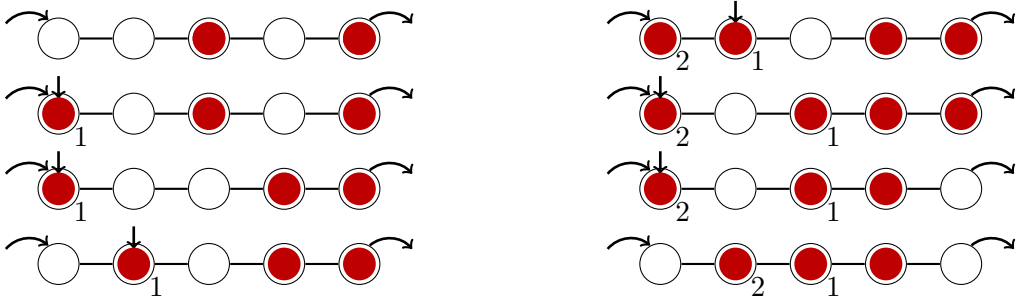
\begin{figure}
\centering
\begin{tikzpicture}[scale=1]

\def\spiral[#1](#2)(#3:#4:#5){% \spiral[draw options](placement)(end angle:revolutions:final radius)
\pgfmathsetmacro{\domain}{pi*#3/180+#4*2*pi}
\draw [#1,
       shift={(#2)},
       domain=0:\domain,
       variable=\t,
       smooth,
       samples=int(\domain/0.08)] plot ({\t r}: {#5*\t/\domain})
}

\def\particles(#1)(#2){

	\node[shape=circle,scale=1.5,draw] (A) at (#1+1,#2){} ;
    \node[shape=circle,scale=1.5,draw] (B1) at (#1+2,#2){} ;
	\node[shape=circle,scale=1.5,draw] (B2) at (#1+3,#2){} ;
    \node[shape=circle,scale=1.5,draw] (B3) at (#1+4,#2){} ;
	\node[shape=circle,scale=1.5,draw] (B4) at (#1+5,#2){} ;	
	
	\draw[thick] (A) to (B1);		
	\draw[thick] (B1) to (B2);		
	\draw[thick] (B2) to (B3);		
	\draw[thick] (B3) to (B4);		
    
    \draw [->,line width=1pt] (#1+0.3,#2+0.2) to [bend right,in=135,out=45] (A);
    \draw [->,line width=1pt] (B4) to [bend right,in=135,out=45] (#1+5.7,#2+0.2);	 	
 
  	 }

 	\node[shape=circle,scale=1.2,fill=nicos-red] (YZ2) at (3,0) {};

	\node[shape=circle,scale=1.2,fill=nicos-red] (YZ3) at (5,0) {};

\particles(0)(0);

 	\node[shape=circle,scale=1.2,fill=nicos-red] (YZ2) at (3,0) {};
	\node[shape=circle,scale=1.2,fill=nicos-red] (YZ3) at (5,0) {};

\particles(0)(-1);

    \draw [->,line width=1pt] (1,0.5-1) to (1,0.2-1);	 	
	\node[shape=circle,scale=1.2,fill=nicos-red] (YZ2) at (1,0-1) {};
 	\node[shape=circle,scale=1.2,fill=nicos-red] (YZ2) at (3,0-1) {};
	\node[shape=circle,scale=1.2,fill=nicos-red] (YZ3) at (5,0-1) {};	
	\node[scale=1]  (x1) at (1+0.3,0-1.3){$1$} ;

\particles(0)(-2);

    \draw [->,line width=1pt] (1,0.5-2) to (1,0.2-2);	 	
	\node[shape=circle,scale=1.2,fill=nicos-red] (YZ2) at (1,0-2) {};
 	\node[shape=circle,scale=1.2,fill=nicos-red] (YZ2) at (4,0-2) {};
	\node[shape=circle,scale=1.2,fill=nicos-red] (YZ3) at (5,0-2) {};	
	\node[scale=1]  (x1) at (1+0.3,0-2.3){$1$} ;

\particles(0)(-3);

    \draw [->,line width=1pt] (2,0.5-3) to (2,0.2-3);	 	
	\node[shape=circle,scale=1.2,fill=nicos-red] (YZ2) at (2,0-3) {};
 	\node[shape=circle,scale=1.2,fill=nicos-red] (YZ2) at (4,0-3) {};
	\node[shape=circle,scale=1.2,fill=nicos-red] (YZ3) at (5,0-3) {};	
	\node[scale=1]  (x1) at (2+0.3,0-3.3){$1$} ;

\particles(8)(0);

    \draw [->,line width=1pt] (8+2,0.5-0) to (8+2,0.2-0);	 	
	\node[shape=circle,scale=1.2,fill=nicos-red] (YZ2) at (8+2,0-0) {};

  %  \draw [->,line width=1pt] (8+1,0.5-0) to (8+1,0.2-0);	 	
	\node[shape=circle,scale=1.2,fill=nicos-red] (YZ2) at (8+1,0-0) {};

 	\node[shape=circle,scale=1.2,fill=nicos-red] (YZ2) at (8+4,0-0) {};
	\node[shape=circle,scale=1.2,fill=nicos-red] (YZ3) at (8+5,0-0) {};	
	\node[scale=1]  (x1) at (8+2+0.3,0-0.3){$1$} ;
	\node[scale=1]  (x1) at (8+1+0.3,0-0.3){$2$} ;

\particles(8)(-1);

 % \draw [->,line width=1pt] (8+2,0.5-1) to (8+2,0.2-1);	 	
	\node[shape=circle,scale=1.2,fill=nicos-red] (YZ2) at (8+1,0-1) {};

    \draw [->,line width=1pt] (8+1,0.5-1) to (8+1,0.2-1);	 	
	\node[shape=circle,scale=1.2,fill=nicos-red] (YZ2) at (8+3,0-1) {};

 	\node[shape=circle,scale=1.2,fill=nicos-red] (YZ2) at (8+4,0-1) {};
	\node[shape=circle,scale=1.2,fill=nicos-red] (YZ3) at (8+5,0-1) {};	
	\node[scale=1]  (x1) at (8+2+1.3,0-1.3){$1$} ;
	\node[scale=1]  (x1) at (8+1+0.3,0-1.3){$2$} ;

\particles(8)(-2);

    \draw [->,line width=1pt] (8+1,0.5-2) to (8+1,0.2-2);	 	
	\node[shape=circle,scale=1.2,fill=nicos-red] (YZ2) at (8+1,0-2) {};

 %   \draw [->,line width=1pt] (8+1,0.5-1) to (8+1,0.2-1);	 	
	\node[shape=circle,scale=1.2,fill=nicos-red] (YZ2) at (8+3,0-2) {};

 	\node[shape=circle,scale=1.2,fill=nicos-red] (YZ2) at (8+4,0-2) {};
	%\node[shape=circle,scale=1.2,fill=nicos-red] (YZ3) at (8+5,0-2) {};	
	\node[scale=1]  (x1) at (8+2+1.3,0-2.3){$1$} ;
	\node[scale=1]  (x1) at (8+1+0.3,0-2.3){$2$} ;

\particles(8)(-3);

 % \draw [->,line width=1pt] (8+2,0.5-1) to (8+2,0.2-1);	 	
%	\node[shape=circle,scale=1.2,fill=nicos-red] (YZ2) at (8+1,0-3) {};
	\node[shape=circle,scale=1.2,fill=nicos-red] (YZ2) at (8+2,0-3) {};

 %   \draw [->,line width=1pt] (8+1,0.5-1) to (8+1,0.2-1);	 	
	\node[shape=circle,scale=1.2,fill=nicos-red] (YZ2) at (8+3,0-3) {};

 	\node[shape=circle,scale=1.2,fill=nicos-red] (YZ2) at (8+4,0-3) {};
%	\node[shape=circle,scale=1.2,fill=nicos-red] (YZ3) at (8+5,0-3) {};	
	\node[scale=1]  (x1) at (8+2+1.3,0-3.3){$1$} ;
	\node[scale=1]  (x1) at (8+2+0.3,0-3.3){$2$} ;
%	\node[scale=1]  (x1) at (8+1+0.3,0-3.3){$3$} ;

%\annhil(0)(0.7);
%\annhil(-13.4)(-0.7);

%\draw [->,line width=1pt]  (B2) to [bend right,in=135,out=45,->] (C);
  
  % \draw [->,line width=1pt] (B2) to [bend right,in=-135,out=-45] (C);
  % \draw [->,line width=1pt] (B2) to [bend right,in=-135,out=-45] (A);
 %     \draw [->,line width=1pt] (A) to [bend right,in=-135,out=-45] (Z2);

	\end{tikzpicture}	
\caption{\label{fig:Tagging}A possible evolution of the tagging mechanism. Particles which entered the segment are shown with their labels. In each step, the tagged particle is marked with an arrow, provided it is contained in the segment. }
\end{figure}
Next, we turn to providing sufficient conditions for the event $A_T$ to hold. A key observation is that we can without loss of generality start from the initial configuration $\xi_0 \equiv 1$, i.e. $\xi_0(i)=1$ for all $i\in [N]$, instead of $\xi_0 \equiv 2$. This is due to fact that we do in \eqref{eq:NoParticleBefore} not distinguish between seeing a first or second class particle to the right of the particle with label~$i$. In particular, under the canonical coupling from Section \ref{sec:CanonicalCoupling}, the event $A_T$ occurs 
when starting from $\xi_0 \equiv 1$ if and only if it occurs when starting from $\xi_0 \equiv 2$, potentially using different times $(t_i)$ and a different $k\in \N$ in the definition of $A_T$ for the two initial conditions $\xi_0 \equiv 1$ and $\xi_0 \equiv 2$, respectively. 
Starting now with $\xi_0 \equiv 1$, we interpret the conditions for the event $A_T$ to hold by tagging particles. We define a \textbf{tag} $(X_t)_{t \geq 0}$ to be a non-decreasing sequence with $X_t \in \N$ for all $t \geq 0$, and declare at time $t$ the first class particle of label $X_t$ as \textbf{tagged}. Further, we require that a tag $(X_t)_{t \geq 0}$ satisfies
\begin{equation}\label{eq:TaggedAlternative}
\xi_t(\ell_{X_t}+1) = 0
\end{equation} for all $t \geq 0$, noting that a tagged particle at time $t$ might not be contained in the segment at time $t\geq 0$, in which case \eqref{eq:TaggedAlternative} trivially holds. Observe that by taking 
\begin{equation}
t_i =  \sup\{ t \geq 0 \colon X_{t} \leq  i \}
\end{equation} for all $i \in \N$, the event $A_T$ holds whenever we can find some $\tau \leq T$ such that $\ell_{\tau}(X_{\tau})=N$; see Figure \ref{fig:Tagging} for an example. \\

\begin{figure}\centering
\begin{tikzpicture}[scale=0.38]

\foreach \x in{1,...,9}{
	\draw[gray!50,thin](\x,\x-1) --++ (1,0) --++ (0,1) --++ (-1,0) --++(0,-1) ; 
}

\foreach \x in{1,...,10}{
	\draw[gray!50,thin](\x,\x-2) --++ (1,0) --++ (0,1) --++ (-1,0) --++(0,-1) ; 
}
\foreach \x in{1,...,11}{
	\draw[gray!50,thin](\x,\x-3) --++ (1,0) --++ (0,1) --++ (-1,0) --++(0,-1) ; 
}
\foreach \x in{1,...,12}{
	\draw[gray!50,thin](\x,\x-4) --++ (1,0) --++ (0,1) --++ (-1,0) --++(0,-1) ; 
}
\foreach \x in{1,...,13}{
	\draw[gray!50,thin](\x,\x-5) --++ (1,0) --++ (0,1) --++ (-1,0) --++(0,-1) ; 
}
	
	\foreach \x in{1,...,8}{
	
	\draw[black!70,thick](\x,\x) --++ (1,0) --++ (0,1) --++ (-1,0) --++(0,-1) ;
	}
	
		\foreach \x in{1,...,14}{
	
	\draw[black!70,thick](\x,\x-6) --++ (1,0) --++ (0,1) --++ (-1,0) --++(0,-1) ;
	}

\draw[darkblue, line width =2 pt] (1.5,-4.5) --++(0,3)--++(1,0) --++(0,2) --++(1,0) --++(0,3) --++(2,0) --++(0,1) --++(2,0) --++(0,1)--++(3,0) --++(0,1) --++(2,0) --++(0,1) --++(1,0) --++(0,1.5);

%\draw[darkblue,line width =2pt,dashed] 	 (6.5,6.5) -- ++(0,1.5);	
%\draw[darkblue,line width =2pt,dashed] 	 (9.5,2.5) -- ++(1.5,0);	

\filldraw [fill=black] (1.5-0.15,1.5-0.15-6) rectangle (1.5+0.15,1.5+0.15-6);   
\filldraw [fill=black] (1.5-0.15+2,1.5-0.15+2) rectangle (1.5+0.15+2,1.5+0.15+2);   
\filldraw [fill=black] (1.5-0.15+3,1.5-0.15+1) rectangle (1.5+0.15+3,1.5+0.15+1);   
\filldraw [fill=black] (1.5+4+7-0.15,1.5+5-0.15) rectangle (1.5+4+7+0.15,1.5+5+0.15);   

	\node[scale=1]  (x1) at (1.53+1,3.5){$x$} ;
	\node[scale=1]  (x1) at (1.53+3.8,2.5){$z$} ;
	\node[scale=1]  (x1) at (1.53+2,3.5-8.1){$(1,-N+1)$} ;	

	\node[scale=1]  (x1) at (1.53+6+6,1.5+5){$y$} ;

\end{tikzpicture}\caption{\label{fig:SemiinfiniteHitting}Visualization of a sufficient condition for $A_T$ when $G_{p}((1,-N),y) < T$. The semi-infinite geodesic from $(1,-N+1)$ for $N=6$ drawn in purple touches both boundaries of the strip in sites $x=(n,n)$ and $y=(n^{\prime}+N,n^{\prime})$ for $n^{\prime}=6$, respectively. In particular, note that since $G_p((1,-N+1),x)>G_p((1,-N+1),z)$ with $x=(3,3)$ and $z=(4,2)$, the particle with label $3$ sees an empty site at position $2$ when entering.}
\end{figure}
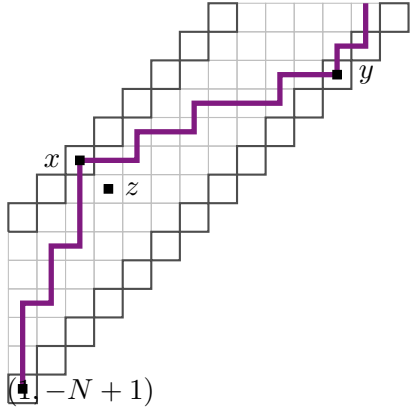
Using the CGM representation of $(\xi_t)_{t \geq 0}$ started from $\xi_0 \equiv 1$, we provide a sufficient condition for the event $A_T$ via semi-infinite geodesics on the strip; see Section \ref{sec:GeodesicsBusemann}. We claim that the event $A_T$ holds whenever the corresponding semi-infinite geodesic $\pi^{\ast}$ starting at $(1,-N+1)$ contains sites $(n,n)$ and $(n^{\prime}+N,n^{\prime})$ for some $1 < n < n^{\prime}$ while we have that 
\begin{equation}\label{eq:ControlLPTSemi}
G_{p}( (1,-N+1), (n^{\prime}+N,n^{\prime}))  < T \, ;
\end{equation}
 see Figure \ref{fig:SemiinfiniteHitting} for a visualization. In other words, $A_T$ holds whenever we ensure that the semi-infinite geodesic started from $(1,-N+1)$ first touches the upper boundary of the strip before returning to the lower boundary at some site $(n^{\prime}+N,n^{\prime})$ while collecting a passage time between $(1,-N+1)$ and $(n^{\prime}+N,n^{\prime})$ of at most $T$. \\

  To show this claim, note that the semi-infinite geodesic $\pi^{\ast}=(v_1=(1,-N+1),v_2,\dots)$ satisfies for $p$ from \eqref{def:EnvironmentStripe} the backwards equations
\begin{equation*}
v_{i} = \begin{cases} v_{i+1}-(0,1) & \text{ if } G_{p}((1,-N+1),v_{i+1}-(0,1))> G_{p}((1,-N+1),v_{i+1}-(1,0)) \\
v_{i+1}-(1,0) & \text{ if } G_{p}((1,-N+1),v_{i+1}-(0,1))< G_{p}((1,-N+1),v_{i+1}-(1,0)) \end{cases}
\end{equation*} for all $v_i,v_{i+1} \in \mathcal{S}(p)$ with $i>N$, recalling the support $\mathcal{S}(p)$ of $p$; see for example (2.8) in \cite{FMP:CompetitionInterface} for a reference on the backwards equations. Consider now  a particle of label $n$ such that $(n,n) \in \pi^{\ast}$. As necessarily $(n+1,n)\in \pi^{\ast}$ holds since the semi-infinite geodesic is contained in the strip, the backwards equations ensure that
\begin{equation}
 G_p((1,-N+1),(n,n)) >  G_p((1,-N+1),(n+1,n-1)) \, .
\end{equation} In particular, at time $G_p((1,-N+1),(n,n))$ when the particle of label $n$ enters the segment, site~$2$ must be empty; see also Figures \ref{fig:SemiinfiniteHitting} and \ref{fig:SemiinfiniteToTagged}. Continuing this argument, set
\begin{equation}
t_1 = \sup\{ G_p( (1,-N+1),(n+j,n)) \colon  (n+j,n)\in \pi^{\ast} \text{ for some } j\in [N] \} 
\end{equation} and $X_t=n$ for all $t\in [0,t_1)$, noting that \eqref{eq:TaggedAlternative} holds due to the backwards equations for all $t\in [0,t_1)$, i.e.\ until time $t_1$, the tagged particle of label $n$ has not seen a particle to its right. At time $t_1$, we declare particle $n+1$ as tagged, noting that as the particle of label $n$ jumps at time $t_1$, we have that $\xi_{t_1}(\ell_{n+1}+1)=0$. Using again the backwards equations, we can set
\begin{equation}
t_2 = \sup\{ G_p( (1,-N+1),(n+j,n+1)) \colon  (n+j,n+1)\in \pi^{\ast} \text{ for some } j\in [N] \} 
\end{equation} and $X_t=n+1$ for all $t\in [t_1,t_2)$, observing that \eqref{eq:TaggedAlternative} holds for all $t\in [t_1,t_2)$. A visualization of this procedure is given in Figure \ref{fig:SemiinfiniteToTagged}. \\
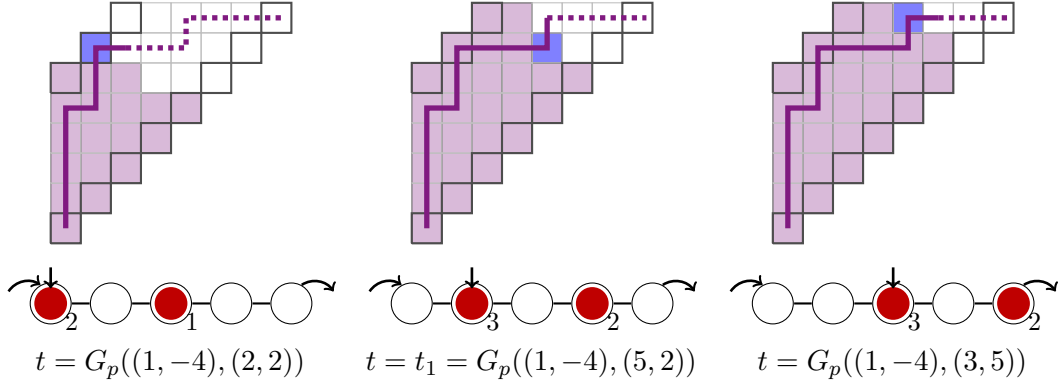
\begin{figure}
\centering
\begin{tikzpicture}[scale=0.8]

\def\particles(#1)(#2){

	\node[shape=circle,scale=1.5,draw] (A) at (#1+1,#2){} ;
    \node[shape=circle,scale=1.5,draw] (B1) at (#1+2,#2){} ;
	\node[shape=circle,scale=1.5,draw] (B2) at (#1+3,#2){} ;
    \node[shape=circle,scale=1.5,draw] (B3) at (#1+4,#2){} ;
	\node[shape=circle,scale=1.5,draw] (B4) at (#1+5,#2){} ;	
	
	\draw[thick] (A) to (B1);		
	\draw[thick] (B1) to (B2);		
	\draw[thick] (B2) to (B3);		
	\draw[thick] (B3) to (B4);		
    
    \draw [->,line width=1pt] (#1+0.3,#2+0.2) to [bend right,in=135,out=45] (A);
    \draw [->,line width=1pt] (B4) to [bend right,in=135,out=45] (#1+5.7,#2+0.2);	 	
 
  	 }
  	 
\def\shape(#1)(#2){

	\foreach \x in{1,...,4}{
		\draw[gray!50,thin](\x/2+#1,\x/2-2/2+#2) --++ (1/2,0) --++ (0,1/2) --++ (-1/2,0) --++(0,-1/2) ; 
	}
	\foreach \x in{1,...,5}{
		\draw[gray!50,thin](\x/2+#1,\x/2-3/2+#2) --++ (1/2,0) --++ (0,1/2) --++ (-1/2,0) --++(0,-1/2) ; 
	}
	\foreach \x in{1,...,6}{
		\draw[gray!50,thin](\x/2+#1,\x/2-4/2+#2) --++ (1/2,0) --++ (0,1/2) --++ (-1/2,0) --++(0,-1/2) ; 
	}
	\foreach \x in{1,...,7}{
		\draw[gray!50,thin](\x/2+#1,\x/2-5/2+#2) --++ (1/2,0) --++ (0,1/2) --++ (-1/2,0) --++(0,-1/2) ; 
	}
	\foreach \x in{1,...,3}{
		\draw[black!70,thick](\x/2+#1,\x/2-1/2+#2) --++ (1/2,0) --++ (0,1/2) --++ (-1/2,0) --++(0,-1/2) ;
	}
	\foreach \x in{1,...,8}{
		\draw[black!70,thick](\x/2+#1,\x/2-6/2+#2) --++ (1/2,0) --++ (0,1/2) --++ (-1/2,0) --++(0,-1/2) ;
	}
}

  	  	 \draw[fill=darkblue!30] (1/2,-5/2) --++ (0,3)--++ (0.5,0) --++(0,-3)--++(-0.5,0);
  		 \draw[fill=darkblue!30] (1,-4/2) --++ (0,2.5)--++ (0.5,0) --++(0,-2.5)--++(-0.5,0);
  		 
 		 \draw[fill=blue!50] (1,0.5) --++ (0,0.5)--++ (0.5,0) --++(0,-0.5)--++(-0.5,0);  		 
  		 
  		 \draw[fill=darkblue!30] (3/2,-3/2) --++ (0,2)--++ (0.5,0) --++(0,-2)--++(-0.5,0);
 		 \draw[fill=darkblue!30] (2,-2/2) --++ (0,1)--++ (0.5,0) --++(0,-1)--++(-0.5,0);
		 \draw[fill=darkblue!30] (5/2,-1/2) --++ (0,0.5)--++ (0.5,0) --++(0,-0.5)--++(-0.5,0);

  	  	 \draw[fill=darkblue!30] (1/2+6,-5/2) --++ (0,3)--++ (0.5,0) --++(0,-3)--++(-0.5,0);
  		 \draw[fill=darkblue!30] (1+6,-4/2) --++ (0,3)--++ (0.5,0) --++(0,-3)--++(-0.5,0);
  		 \draw[fill=darkblue!30] (3/2+6,-3/2) --++ (0,3)--++ (0.5,0) --++(0,-3)--++(-0.5,0);
  		 \draw[fill=darkblue!30] (2+6,-2/2) --++ (0,2.5)--++ (0.5,0) --++(0,-2.5)--++(-0.5,0);
  		 \draw[fill=darkblue!30] (2.5+6,-1/2) --++ (0,1)--++ (0.5,0) --++(0,-1)--++(-0.5,0);
  		 \draw[fill=darkblue!30] (3+6,-0) --++ (0,0.5)--++ (0.5,0) --++(0,-0.5)--++(-0.5,0);

 		\draw[fill=blue!50] (2.5+6,0.5) --++ (0,0.5)--++ (0.5,0) --++(0,-0.5)--++(-0.5,0);

  	  	 \draw[fill=darkblue!30] (1/2+12,-5/2) --++ (0,3)--++ (0.5,0) --++(0,-3)--++(-0.5,0);
  		 \draw[fill=darkblue!30] (1+12,-4/2) --++ (0,3)--++ (0.5,0) --++(0,-3)--++(-0.5,0);
  		 \draw[fill=darkblue!30] (3/2+12,-3/2) --++ (0,3)--++ (0.5,0) --++(0,-3)--++(-0.5,0);
		 \draw[fill=darkblue!30] (2+12,-2/2) --++ (0,2.5)--++ (0.5,0) --++(0,-2.5)--++(-0.5,0);
  		 \draw[fill=darkblue!30] (2.5+12,-1/2) --++ (0,1.5)--++ (0.5,0) --++(0,-1.5)--++(-0.5,0);
  		 \draw[fill=darkblue!30] (3+12,-0) --++ (0,1)--++ (0.5,0) --++(0,-1)--++(-0.5,0);

 		\draw[fill=blue!50] (2.5+12,1) --++ (0,0.5)--++ (0.5,0) --++(0,-0.5)--++(-0.5,0);  	

 \particles(-0.5)(-3.5);
 
\node[shape=circle,scale=1.2,fill=nicos-red] (YZ2) at (0.5,-3.5) {}; 
\node[shape=circle,scale=1.2,fill=nicos-red] (YZ2) at (0.5+2,-3.5) {}; 
 
	\node[scale=0.9]  (x1) at (0.5+0.35,0-3.8){$2$} ; 
	\node[scale=0.9]  (x1) at (2.5+0.35,0-3.8){$1$} ;  
 
 \particles(-0.5+6)(-3.5);
 
\node[shape=circle,scale=1.2,fill=nicos-red] (YZ2) at (1.5+6,-3.5) {}; 
\node[shape=circle,scale=1.2,fill=nicos-red] (YZ2) at (1.5+2+6,-3.5) {}; 
%\node[shape=circle,scale=1.2,fill=nicos-red] (YZ2) at (2.5+2+6,-3.5) {}; 

	\node[scale=0.9]  (x1) at (1.5+0.35+6,0-3.8){$3$} ;  
	\node[scale=0.9]  (x1) at (3.5+0.35+6,0-3.8){$2$} ; 
%	\node[scale=0.9]  (x1) at (4.5+0.35+6,0-3.8){$1$} ;   

 \particles(-0.5+12)(-3.5);
 
\node[shape=circle,scale=1.2,fill=nicos-red] (YZ2) at (1.5+12+1,-3.5) {}; 
\node[shape=circle,scale=1.2,fill=nicos-red] (YZ2) at (1.5+2+12+1,-3.5) {};  
 
	\node[scale=0.9]  (x1) at (1.5+0.35+12+1,0-3.8){$3$} ;  
	\node[scale=0.9]  (x1) at (3.5+0.35+12+1,0-3.8){$2$} ;

 \draw [->,line width=1pt] (0.5,-2.9) to (0.5,-3.25);	 	
   
 \draw [->,line width=1pt] (0.5+7,-2.9) to (0.5+7,-3.25);	
 
 \draw [->,line width=1pt] (0.5+14,-2.9) to (0.5+14,-3.25);	   
   
 \shape(0)(0);
 \shape(6)(0);
 \shape(12)(0);

 \draw[darkblue, line width =2 pt] (1/2+0.25,-5/2+0.25) --++(0,2)--++(0.5,0) --++(0,1)--++(0.5,0);

 \draw[darkblue, line width =2 pt,dotted] 	(1/2+1.25,-5/2+3.25) --++(1,0) --++(0,0.5) --++(1.6,0);

 \draw[darkblue, line width =2 pt] (1/2+0.25+6,-5/2+0.25) --++(0,2)--++(0.5,0) --++(0,1)--++(0.5,0)--++(1,0) --++(0,0.5);

 \draw[darkblue, line width =2 pt,dotted] 	(1/2+1.25+7,-5/2+3.75) --++(1.65,0);

	\draw[darkblue, line width =2 pt] (1/2+0.25+12,-5/2+0.25) --++(0,2)--++(0.5,0) --++(0,1)--++(0.5,0)--++(1,0) --++(0,0.5)--++(0.5,0);

 \draw[darkblue, line width =2 pt, dotted] 	(1/2+1.25+13.5,-5/2+3.75) --++(1.15,0);

		\node[scale=1]  (x1) at (2.5,-4.5){$t=G_p((1,-4),(2,2))$} ;
		\node[scale=1]  (x2) at (2.5+6,-4.5){$t=t_1=G_p((1,-4),(5,2))$} ;
		\node[scale=1]  (x3) at (2.5+12,-4.5){$t=G_p((1,-4),(3,5))$} ; 	 
 
\end{tikzpicture}
\caption{\label{fig:SemiinfiniteToTagged}
Correspondence between the TASEP with open boundaries and its growth interface at three different time points $t$. The cell added at time $t$ is highlighted in blue. Note that the tagged particle, marked with an arrow, never has a particle to its right. The backwards equations for the semi-infinite geodesic started at $(1,-4)$ yield that the cells $(3,1)$, $(4,3)$, and $(6,2)$ were part of the growth interface before the respective blue cells are added.
}
\end{figure}

In general, suppose that $(i,j),(i+1,j)\in\pi^{\ast} $ for some $i\geq j$ and $n^{\prime} \geq j\geq n$. Then the backwards equations ensure that
\begin{equation}
G_p((1,-N+1),(i,j)) >  G_p((1,-N+1),(i+1,j-1)) \, ,
\end{equation}
i.e. at the time the particle with label $j$ jumps from site $i-j$ to site $i-j+1$, the particle with label $j-1$ has already reached site $i-j+3$ or left the segment. Thus, for all $j\in \N$ with $n^{\prime} \geq j \geq n$, we can choose
\begin{equation}
t_{j-n+1} = \sup\{ G_p( (1,-N+1),(i,j)) \colon  (i,j)\in \pi^{\ast} \text{ for some } i\in \N \} 
\end{equation} and $X_t=j$ for all $t\in [t_{j-n},t_{j-n+1})$, noting \eqref{eq:TaggedAlternative} holds for all $t\in [t_{j-n},t_{j-n+1})$. 
%  Similarly,  for $(i,j),(i,j+1)\in\pi^{\ast} $, the particle with label $j+1$ sees no particles in front of it for all $s\in [G_p((1,-N),(i,j)),G_p((1,-N),(i,j+1)) )$. Thus, we can choose $t_{j}$ to be the largest last passage time of a site $(\cdot,j) \in \pi^{\ast}$.  
Furthermore, notice that for $k=n^{\prime}-n+1$ and $\tau=G_p( (1,-N+1),(n^{\prime}+N,n^{\prime}))$, the condition \eqref{eq:ControlLPTSemi} for the semi-infinite geodesics $\pi^{\ast}$ implies that we can find some $\tau < \min(T,t_k)$ such that $\ell_{\tau}(X_{\tau})=N$. Combining the above observations, this justifies the claim. \\

As a last step, we argue that $A_T$ holds with positive probability for some $T$ of order $N^{3/2}$, using the above correspondence to semi-infinite geodesics. 
By the Markov property of $(\xi_t)_{t \geq 0}$ and \eqref{eq:DisagreeingRelation}, we use independent attempts to see that for any $\varepsilon>0$, there exists some $C>0$ such that all second class particles in the disagreement process have left the segment at time $CN^{3/2}$ with probability at least $1-\varepsilon$. Together with Lemma \ref{lem:MixingBoundByExiting},  we obtain the desired upper bound on the mixing time in Theorem~\ref{thm:Triple}. To see that $T$ can indeed be chosen of order $N^{3/2}$, consider the events $B^{1,m}_+$ and $B^{m,m}_-$ from Lemma~\ref{lem:CrossingProbability} with $m=cN^{3/2}$ for $c=c(1/4)>0$.  From Lemma~\ref{lem:OrderingGeodesics} on the ordering of geodesics, we see that the semi-infinite geodesic touches both boundaries of the strip as required whenever $B^{1,m}_+$ and $B^{m,m}_-$ hold, and $T=16cN^{3/2}$. Since $B^{1,m}_+$ and $B^{m,m}_-$ are independent by construction and have probability at least $1/4$, we conclude.  
\end{proof}

\section{Lower bounds on the mixing times}\label{sec:LowerBounds}

In this section, we give a proof of the lower bound of order $N^{3/2}$ in Theorem~\ref{thm:LowerBound} using the results from Section \ref{sec:MixingTimesMaxCurrent} on last passage times on the strip. Let us stress that the presented arguments do not rely on spectral techniques like the Bethe ansatz. We start with a comparison between geodesics on the strip. We use it as an estimate on the total number of particles in the TASEP with open boundaries for some time $t$  of order $N^{3/2}$, starting from the initial configuration $\mathbf{0}$ with only empty sites.

\begin{lemma}\label{lem:TechnicalGeodesics} Recall the last passage time $G_p$ in the environment $(\omega_x^p)_{x \in \Z^2}$ according to $p$ from \eqref{def:EnvironmentStripe} for boundary parameters $\alpha,\beta$, and the initial growth interface for $\mathbf{0}$.  Then for every $\varepsilon,\theta>0$, there exist  $\delta,d>0$, depending only on $\alpha,\beta,\varepsilon,\theta>0$, such that
\begin{align}
\label{eq:GeodesicLow1} B_{1} := & \left\{G_p(0,\delta N^{3/2}) > 4\delta N^{3/2} - d \sqrt{N} \right\} \\
\label{eq:GeodesicLow2} B_{2} := & \left\{ G_p(0,(\delta N^{3/2}+ \theta\sqrt{N}+ N/2, \delta N^{3/2} + \theta\sqrt{N} -N/2)) < 4\delta N^{3/2} - d \sqrt{N}\right\} 
\end{align} satisfy $\P(B_1) \geq 1- \varepsilon/4$ and $\P(B_2) \geq 1- \varepsilon/4$ for all $N$ sufficiently large.
\end{lemma}
Informally speaking, Lemma \ref{lem:TechnicalGeodesics} says that with high probability, the geodesic in the strip between $(0,0)$ and a site on the diagonal  is larger than the  geodesic connecting to an endpoint, which has a larger $|\cdot|_1$-distance, but is sufficiently far off from the diagonal.
\begin{proof}[Proof of Lemma \ref{lem:TechnicalGeodesics}] Recall from Lemma \ref{lem:ShapeTheorem} the last passage time $G_1$ in a homogeneous environment on $\Z^2$. By \eqref{eq:ExpectationShapeTheorem}, we see that
\begin{align*}
\E\left[ G_1(0,\delta N^{3/2})  \right] \geq 4\delta N^{3/2} - C \delta^{1/3} N^{1/2} 
\end{align*} holds for the constant $C=C(1/2)$ from Lemma \ref{lem:ShapeTheorem}. We obtain that  $\P(B_1) \geq 1 - \varepsilon/4$ by choosing $d>0$ large enough, uniformly in $\delta<1$, and applying Theorem 4.2 (iii) in \cite{BGZ:TemporalCorrelation}, which bounds the fluctuations of the last passage time on the strip with respect to the expected last passage time in a homogeneous environment. % Similarly,
%\begin{align*}%\label{eq:ExpectationAll1}
%\E\left[ G_1(0,(\delta N^{3/2}+ \theta\sqrt{N}+ N/2, \delta N^{3/2} + \theta \sqrt{N} -N/2)  \right] \\
%  \leq 4\delta N^{3/2} + 4 \theta \sqrt{N} - (4\delta)^{-1} N^{1/2}  + 2C \delta^{1/3} N^{1/2} 
%\end{align*} 
%holds for $C=C(1/2)$ from Lemma \ref{lem:ShapeTheorem} when $N$ is large enough. 
To see that $\P(B_2) \geq 1 - \varepsilon/4$ holds for any given $\theta>0$ and $d=d(\theta,\varepsilon)$, we apply Corollary \ref{cor:LowerBoundLPP} with $\delta=\delta(d,\varepsilon)>0$ small enough.  This yields the desired bound on the fluctuations of the last passage time between two points in the strip.
\end{proof}

%Equipped with these estimates, we now prove Theorem \ref{thm:LowerBound}.

\begin{proof}[Proof of Theorem \ref{thm:LowerBound}]

Recall from Section \ref{sec:ModelResults} that we denote by $\mu_{N}$ the unique invariant measure of the TASEP with open boundaries. For all $\theta>0$, we define
\begin{equation*}%\label{eq:invariantEvents}
E_{\theta} := \left\{ \eta \in \{ 0,1\}^N \colon \sum_{j=1}^{N} \Big(\eta(j) -\frac{1}{2}\Big) \geq  -    \theta \sqrt{N}  \right\}  \, .
\end{equation*} 
By Lemma \ref{lem:invariant}, we can find for every $\varepsilon>0$  some $\theta=\theta(\varepsilon)>0$ and $N_0\in \N$ such that 
\begin{equation}\label{eq:invariantCondition}
\mu_N( E_\theta) \geq 1- \frac{\varepsilon}{2}
\end{equation}
holds for all $N\geq N_0$. Let $(\eta_t)_{t \geq 0}$ be the TASEP with open boundaries started from the empty initial configuration, i.e.\ $\eta_0=\mathbf{0}$. We claim that we can find  $\delta,d>0$ such that for $t_N=t_N(\delta,d) :=4\delta N^{3/2}-d\sqrt{N}$ 
\begin{equation}\label{eq:MixingBoundLower}
\P\left(  \sum_{i=1}^N \eta_{t_N}(i)  \geq \frac{N}{2} -  \theta\sqrt{N}  \  \Big| \ \eta_0 = \mathbf{0} \right) \leq  \frac{\varepsilon}{2}
\end{equation} holds for all $N$ sufficiently large, where $\theta=\theta(\varepsilon)$ is taken from \eqref{eq:invariantCondition}. To see this, recall from Lemma \ref{lem:CurrentVsGeodesic} the correspondence between a corner growth model on a strip and the TASEP with open boundaries. Since we start from the all empty initial configuration, whenever 
\begin{equation*}
G_p(0,\delta N^{3/2}) > t_N
\end{equation*}
holds, we know from the location of the corresponding growth interface that at most $\delta N^{3/2}$ particles have entered the segment by time $t_N$. Similarly, when
\begin{equation*}
G_p(0,(\delta N^{3/2}+ \theta\sqrt{N}+ N/2, \delta N^{3/2} + \theta\sqrt{N} -N/2)  < t_N
\end{equation*}
holds, we see that at least $\delta N^{3/2}+ \theta\sqrt{N}-N/2$ particles left the segment by time $t_N$. Now apply Lemma~\ref{lem:TechnicalGeodesics} to obtain \eqref{eq:MixingBoundLower}. This gives
\begin{align}
\TV{\P\left( \eta_{t_N} \in \cdot \ \right | \eta_0 = \mathbf{0}) - \mu_N} \geq \mu_N( E_\theta)  - \P\left( \eta_{t_N} \in E_\theta \ \right | \eta_0 = \mathbf{0}) \geq 1 - \varepsilon \, ,
\end{align} using the characterization of the total-variation distance in \eqref{def:TVDistance} for the first inequality, together with \eqref{eq:invariantCondition} and \eqref{eq:MixingBoundLower} for the second step. Since we take a maximum over all possible initial configurations in the definition of the mixing time, this yields the desired lower bound. 
\end{proof}

\subsection*{Acknowledgment} I want to thank Patrik Ferrari, Nina Gantert, Nicos Georgiou, Evita Nestoridi, and Shangjie Yang for helpful suggestions and answering questions about last passage percolation and random polymer models, and Milton Jara for pointing out Conjecture~\ref{conj:Milton}. Moreover, I am extremely grateful to two anonymous referees for their valuable suggestions, which significantly helped to improve the paper. The Studienstiftung des deutschen Volkes 
%and the TopMath program are
is acknowledged for financial support.

\bibliographystyle{plain}

\bibliography{TASEPmixing}

\vspace{0.1cm}

\appendix

\section{Proof of Lemma \ref{lem:LastPassagePercolationGeneral}}

We now give a proof of Lemma \ref{lem:LastPassagePercolationGeneral} following the arguments from Proposition 10.1 of \cite{BSV:SlowBond} for Poissonian last passage percolation. We start by defining a family of parallelograms, which use in order to cluster the rectangle $R := R_N^{\theta^{-1}N^{3/2}} \subseteq \Z^2 $ from~\eqref{def:Rectangle}. For all $k\in \{0,1,2,\dots\}$, we define the sets
\begin{align*}
V_k &:= \Big\{ (\lfloor x_1 \rfloor, \lfloor x_2-x_1 \rfloor)  \colon x_1= \frac{1}{8} N^{3/2} 2^{-k/4} i \text{ and } x_2= N 2^{-k/6} j\text{ for some } i,j\in \Z \text{ and }\\
&\qquad   |x_1 -z_1| \leq 2^{-k/4} N^{3/2} \text{ and } |(x_2-x_1)-z_2| \leq N 2^{-k/6} \text{ for some } (z_1,z_2) \in R \Big\} \, ,
\end{align*}
i.e.\ the set $V_k$ is a discrete lattice approximating the rectangle $R$ by  $|V_k| \leq 4^{k+2}$ points. For each pair of lattice points $s=(s_1,s_2)$ and $u=(u_1,u_2)$ with $s,u\in V_k$ such that 
\begin{equation}\label{eq:SlopeConditionUniform}
u_1 - s_1 = 2^{-k/4} N^{3/2} \quad \text{ and } \quad  \frac{2}{3} \leq 1 + \frac{u_2-s_2}{2^{-k/12}N^{1/2}} \leq \frac{3}{2} \, ,
\end{equation}
we denote by $U^{(k)}_{s,u}$ the parallelogram spanned by the four points
\begin{equation*}
s + (0,N 2 ^{-k/6}) , \ s - (0,N 2 ^{-k/6}), \ u + (0,N 2 ^{-k/6}) , \ u - (0,N 2 ^{-k/6}) \, .
\end{equation*}
Let $\mathcal{U}_k$ denote the set of all such parallelograms using sites in $V_k$, and note that $|\mathcal{U}_k| \leq 16^{k+2}$. Further, for all $U=U^{(k)}_{s,u} \in \mathcal{U}_k$, we set
\begin{align*}
U_L &:= U \cap \left\{ (x_1,x_2)\in \Z^2 \colon s_1 \leq x_1 \leq s_1 + \frac{1}{8} 2^{-k/4} N^{3/2} \right\} \\
U_R &:= U \cap \left\{ (x_1,x_2)\in \Z^2 \colon u_1 - \frac{1}{8} 2^{-k/4} N^{3/2} \leq x_1 \leq u_1 \right\} \, .
\end{align*}
The following lemma is a reformulation of Theorem 4.2 (ii) of \cite{BGZ:TemporalCorrelation} for our setup. %with $\psi=3/2$ in their notation.
\begin{lemma}\label{lem:LingfusBound} Fix $p \equiv 1$, and let $k=k(N) \in \N$  satisfy 
\begin{equation}
\limsup_{N \rightarrow \infty} \frac{k}{6\log_2(N)} < 1
\end{equation} There exist constants $c,\lambda_0>0$ such that for each parallelogram $U=U^{(k)}_{s,u} \in \mathcal{U}_k$ with sites $s,u \in V_k$ satisfying \eqref{eq:SlopeConditionUniform}, and all $\lambda=\lambda(N) \geq \lambda_0$ and $N$ large enough, we have that
\begin{equation*}
\P\Big( \max_{ x \in U_L,y\in U_R } G_p(x,y) - \E[G_p(x,y)] \geq \lambda 2^{-k/12} \sqrt{N}\Big) \leq \exp(-c\min(\lambda^{\frac{3}{2}},\lambda \sqrt{N} 2^{-k/12} )) \, .
\end{equation*}
\end{lemma}

\begin{proof}[Proof of Lemma \ref{lem:LastPassagePercolationGeneral}]

In the following, set \begin{equation}
K =K(N):= \lceil 5\log_2(N)  \rceil + 3 \, .
\end{equation} Consider a pair of sites $(x,y) \in \mathbb{B}(R)$ such that $|y_1-x_1| \geq N^{1/4}$. Our key observation is that for all $N$ sufficiently large and $k\in \{0,1,\dots,K\}$, for each pair $(x,y) \in \mathbb{B}(R)$ with
\begin{equation}
y_1-x_1 \in \Big[\frac{3}{4}2^{-k/4}N^{3/2},2^{-k/4}N^{3/2}\Big] \, ,
\end{equation} we find some $U \in \mathcal{U}_k$  such that $x\in L(U)$ and $y \in R(U)$ holds. %, i.e.\ let $s,u \in V_k$ be such that 
%\begin{equation}
% s \preceq x \preceq y \preceq u  \text{ and } \max(x_1-s_1,u_1-y_1) \leq   \text{ and } \max(x_1-s_1,u_1-y_1) \leq    \, ,
%\end{equation}
% and note that the slope condition in \eqref{eq:SlopeConditionUniform} holds for this choice of $u$ and $s$ when $N$ is sufficiently large.  \\
Next, we define for all $k\in \{0,1,\dots,K\}$ the events
\begin{equation}
\mathcal{E}_k := \left\{   \sup_{ U \in \mathcal{U}_k }\max_{ x \in U_L,y\in U_R } G_p(x,y) - \E[G_p(x,y)] \geq  \theta \sqrt{N} \right\} \, ,
\end{equation} and we let
\begin{equation}
\tilde{\mathcal{E}} := \left\{   \max_{ (x,y)\in \mathbb{B}(R) \colon |y-x|_1 \leq N^{\frac{1}{4}}  } G_p(x,y) - \E[G_p(x,y)] \geq  \theta \sqrt{N} \right\} \, .
\end{equation}
From Lemma \ref{lem:LingfusBound} for $\lambda=\theta 2^{k/12}$, we get that 
\begin{align}\label{eq:AppendixBound1}
\P\left( \bigcup_{k\in \{0,1,\dots,K\}} \mathcal{E}_k\right) &\leq \sum_{k\in \{0,1,\dots,K\}} 16^{k+2} \exp(-c_1 \min(2^{k/12}\theta^{\frac{3}{2}}, \theta \sqrt{N})) \nonumber \\
&\leq \exp(-c_2 \theta^{3/2})
\end{align} for some constants $c_1,c_2,\theta_0>0$, and all $\theta\geq \theta_0$, provided that $N$ is sufficiently large.  Moreover, using Lemma \ref{lem:ShapeTheorem} and a union bound, we get that
\begin{equation}\label{eq:AppendixBound2}
\P( \tilde{\mathcal{E}} ) \leq N^{5} \exp(-c_3 N^{\frac{1}{4}})
\end{equation} for some constant $c_3>0$, and $N$ sufficiently large.
Since for all $N$ large enough, each pair of points in $\mathbb{B}(R)$ is covered within the events $(\mathcal{E}_k)_{k \in \{0,1,\dots,K\}}$ and $\tilde{\mathcal{E}}$, we conclude by combining \eqref{eq:AppendixBound1} and \eqref{eq:AppendixBound2}.
\end{proof}

\end{document}